\documentclass[reqno]{amsart}
\usepackage{amsthm, amsmath,amsfonts,mathtools,bm,bbm,MnSymbol}

\usepackage[left=3.4cm,right=3.4cm,top=2.5cm,bottom=3cm]{geometry}
\usepackage[hidelinks]{hyperref}
\usepackage{enumerate}

\allowdisplaybreaks

\newcommand{\ind}{1}

	\newcommand	{\C}			{\mathbb{C}}
	
	\newcommand	{\N}			{\mathbb{N}}

	\newcommand	{\R}			{\mathbb{R}}
	
	\newcommand{\eps}{\varepsilon}
	
	\DeclareMathOperator	{\supp}			{supp}
	\renewcommand{\Re}{\operatorname{Re}}
	\renewcommand{\Im}{\operatorname{Im}}

	\newcommand{\eee}{{\rm e}}
	\newcommand{\ton}{\overset{}{\underset{n\to\infty}\longrightarrow}}

\theoremstyle{plain}
\newtheorem{theorem}{Theorem}[section]
\newtheorem{lemma}[theorem]{Lemma}
\newtheorem{corollary}[theorem]{Corollary}
\newtheorem{prop}[theorem]{Proposition}

\theoremstyle{definition}
\newtheorem{definition}[theorem]{Definition}

\theoremstyle{remark}
\newtheorem{remark}[theorem]{Remark}
\newtheorem*{claim}{Claim}

\numberwithin{equation}{section}

\title{Zeros of polynomial powers under the heat flow}
\date{\today}

\author[Antonia H\"ofert]{Antonia H\"ofert}
\address[Antonia H\"ofert]{Paderborn University, Institute of Mathematics, Warburger Str. 100, 33098 Paderborn, Germany}
\email[Antonia H\"ofert]{ahoefert@math.upb.de}

\author[Jonas Jalowy]{Jonas Jalowy}
\address[Jonas Jalowy]{Paderborn University, Institute of Mathematics, Warburger Str. 100, 33098 Paderborn, Germany}
\email[Jonas Jalowy]{jjalowy@math.upb.de}

\author[Zakhar Kabluchko]{Zakhar Kabluchko}
\address[Zakhar Kabluchko]{Fachbereich Mathematik und Informatik,
	Universit\"at M\"unster,
	Einsteinstra\ss e 62,
	48149 M\"unster,
	Germany}
\email[Zakhar Kabluchko]{zakhar.kabluchko@uni-muenster.de}



\begin{document}
\begin{abstract}
We study the zero evolution of high powers of polynomials under the heat flow. For any fixed polynomial $P(z)$, we prove that the empirical zero distribution of its heat-evolved $n$-th power converges to a distribution on the complex plane as $n$ tends to infinity. We describe this limit distribution $\mu_t$ as a function of the time parameter $t$ of the holomorphic heat flow operator: 

For small time, zeros of the heat-evolved polynomial start to spread out from the initial zeros $\lambda_j$ of $P$, and $\mu_t$ approximates the superposition of semicircle laws around $\lambda_j$. Then for arbitrary time, the support of $\mu_t$ forms intricate curves, which merge as $t$ grows, until for large time, the limit distribution $\mu_t$ approaches a widespread semicircle law through the initial center of mass of the $\lambda_j$. We further show that the Stieltjes transform of $\mu_t$ satisfies a self-consistent equation and a Burgers' equation. 
The present paper deals with general complex-rooted polynomials for which, in contrast to the real-rooted case, no free-probabilistic representation for $\mu_t$ is available.
\end{abstract}
	
	\maketitle
	
	\section{Introduction}
	
	This paper is concerned with the behavior of zeros of high-degree polynomials undergoing the heat flow and is part of an active research area that investigates how differential operators acting on random functions affect their zeros.
	
	The history of finding the roots of a polynomial affected by a differential operator can be traced back to at least the Gauss--Lucas theorem~\cite[Theorem 6.1]{Marden}, locating the roots of the derivative inside the convex hull of the roots of the original polynomial.
	Recently, there has been active development in the study of the zero  evolution of polynomials under the action of differential operators, see e.g.~\cite{totik,Byun,OSteiner,Steiner21,HK21,KT22,COR23,diff-paper,martinezfinkelshtein, arizmendi_garza_vargas_perales,BHS24,CampbellAppell,CORuni,JKM1,GNV25} for (repeated) differentiation and \cite{tao_blog1,tao_blog2,rodgers_tao,ZakharLeeYang,hallho,GAF-paper,heatflowrandompoly,JKM2,CJ} for the heat flow. We refer the interested reader to these works and references therein for fascinating connections to random matrix theory, free probability, analytic number theory, potential theory, and partial differential equations (PDEs).
	
	In this work, we consider a polynomial power
	\begin{align}\label{eq:P^n}
		P^{n}(z)=\prod_{j=1}^d (z-\lambda_j)^{n\alpha_j},\quad n\in \N, z\in \C,
	\end{align}
	for a fixed polynomial $P(z) = \prod_{j=1}^d (z-\lambda_j)^{\alpha_j}$ of degree $\alpha=\sum_j\alpha_j$ with $d\in\N$ distinct zeros $\lambda_1,...,\lambda_d\in\C$, each $\lambda_j$ of multiplicity $\alpha_j\in\N$. We are interested in the evolution of these zeros under the effect of the (holomorphic, backward) heat flow operator. The heat evolution of $P^n$ is defined as a terminating power series
	\begin{align}\label{eq:def_heat_flow}
		P^{n}_t(z):=e^{-\frac{t}{2\alpha n}\partial_z^2}P^{n}(z)\coloneq  \sum_{k=0}^\infty \frac{(-\tfrac t{2\alpha n})^k\partial_z^{2k}}{k!}P^{n}(z),
	\end{align}
	where $\partial_z=\frac{d}{dz}$ is the complex derivative and $t\in \C$ can be seen as a ``time'' parameter, allowed to be complex. We define the empirical zero distribution of $P_t^n$ as $\mu_{n,t}\coloneq\frac 1 {\alpha n}\sum_{\{\lambda:P_t^n(\lambda)=0\}}\delta_\lambda$, counting multiplicities, and our aim is to find and describe its weak limit $\mu_t$ as $n\to\infty$.
	
	\subsection*{Summary of the Results}
	We prove that, as $n\to\infty$, the zero distribution of the heat-evolved polynomial powers $P_t^n$ converges weakly to a certain compactly supported probability measure $\mu_t$ on $\mathbb C$ which we shall characterize in the following. It is natural to describe our main results along the time evolution $t\ge 0$, beginning with the initial distribution $\mu_0=\frac 1 \alpha \sum_{j=1}^d\alpha_j\delta_{\lambda_j}$.
	\begin{enumerate}
		\item As $t\to 0$, we zoom into a scaling window of size $\sqrt t$ centered at an initial zero $\lambda_j$ and identify a semicircle law of total mass $\alpha_j/\alpha$ as the vague limit of $\mu_t(\lambda_j+\cdot \sqrt{t\alpha_j/\alpha})$, see Theorem \ref{thm:small_t}.
		\item If $t>0$ is sufficiently small, then $\mu_t$ is supported on a union of closed disks of radii $(2+\eps)\sqrt{ t\alpha_j/\alpha}$ centered at the zeros  $\{\lambda_j\}_{j=1,\dots,d}$, see Corollary \ref{cor:small_t}.
		\item For arbitrary $t>0$, $\mu_t$ is supported on a union of finitely many smooth curves, see Theorem \ref{thm:description}. The distribution $\mu_t$ is characterized by a self-consistent equation for the Stieltjes transform, via its logarithmic potential, or its density (see Theorem~\ref{thm:description} and Corollary \ref{cor:density}).
		\item If $t>0$ is large, then the support of $\mu_t$ within any sufficiently large disk will approach a horizontal straight segment through the center of mass $\frac 1  \alpha \sum_{j=1}^d\alpha_j\lambda_j$, see Theorem \ref{thm:large_t}.
		\item As $t\to \infty$, we zoom out by $\sqrt t$ and will always retrieve a standard semicircle law as the weak limit of $\mu_t(\sqrt t \cdot)$, see Theorem \ref{thm:semicircle}.
		\item For arbitrary $t\in\C$, $\mu_t$ is the push-forward under the rotation $z\mapsto z e^{i \arg(t)/2}$ of the limiting zero distribution obtained by applying the heat flow at time $|t|>0$ to $P^n(e^{i\arg (t)/2}z)$, see Theorem \ref{thm:measure}.
		\item We prove the Hamilton-Jacobi PDE for the logarithmic potential of $\mu_t$ and the Burgers' equation for its Stieltjes transform, analogous to \cite[\S 5.1]{heatflowrandompoly}, see Theorem \ref{thm:PDEs}.
	\end{enumerate}
	Figure \ref{fig:heat_flow} provides a simulation illustrating our results.
	
	\begin{figure}[t]
		\centering
		\includegraphics[width=0.33\linewidth]{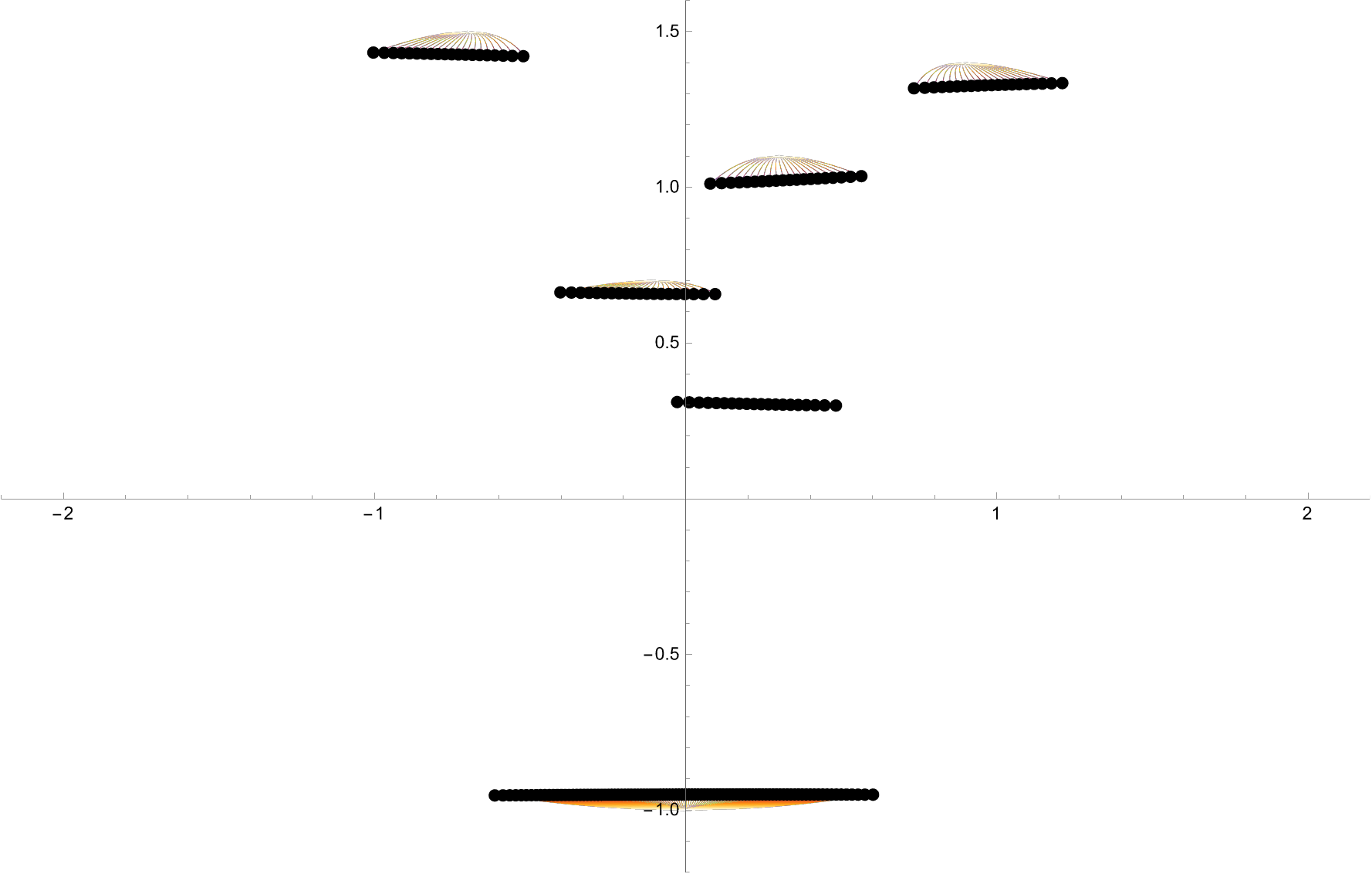}
		\includegraphics[width=0.33\linewidth]{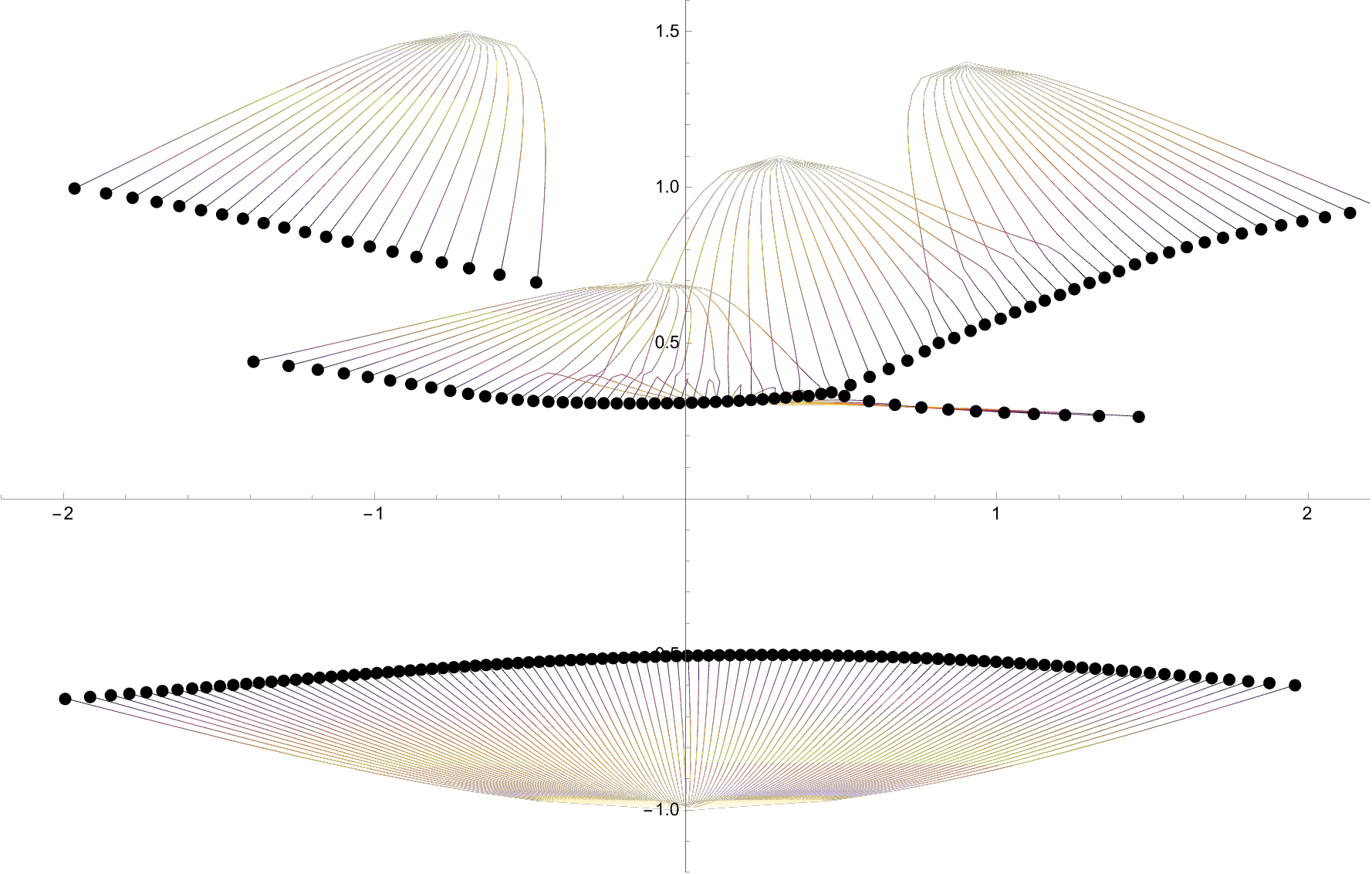}\raisebox{2.3em}{\includegraphics[width=0.33\linewidth]{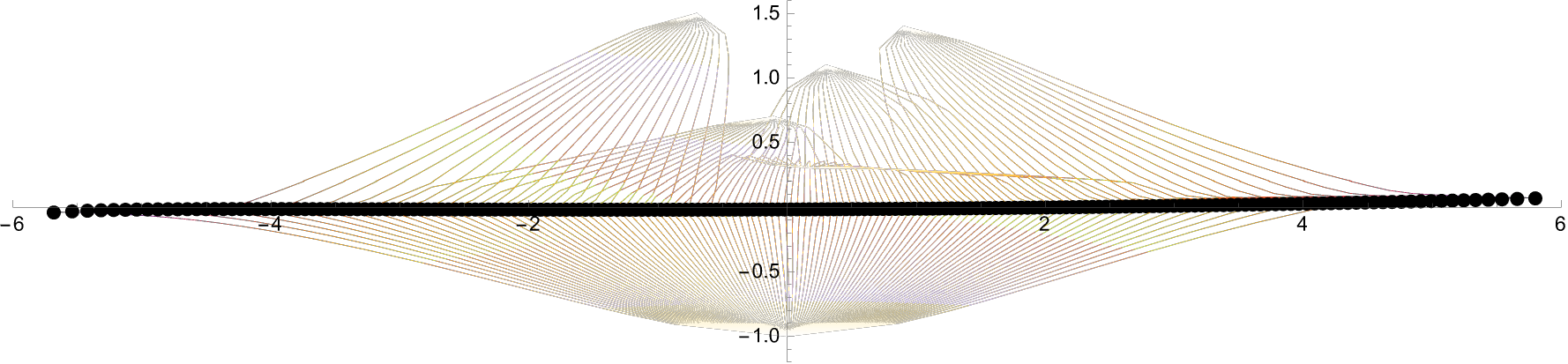}}
		\caption{Evolution of heat-evolved polynomial powers $P_t^n$ for $d=6$ different zeros, $\lambda_1=-i$ with multiplicity $\alpha_1=5$ and five simple zeros $\lambda_2,\dots\lambda_6$ somewhere in the upper half-plane. The $\alpha n=200$ zeros are depicted as black dots with their trajectories in orange. For small time $t=1/5$ (left), we see small semicircle components of $\mu_t$ emerging from the respective initial zeros $\lambda_j$. At arbitrary time $t\approx 1.8$ (center) these lines become intricate to describe analytically and start to merge like a ``zipper''. For large time $t=10$ (right) only one curve remains in the support of $\mu_t$, passing through the center of mass (here, $=0$), which unfolds into the semicircular law when zooming out.}
		\label{fig:heat_flow}
	\end{figure}
	
	\subsection*{Background}
	Our motivation to identify the limiting zero distribution $\mu_t$ of $P^{n}_t$ is threefold.
	
	\textsc{(i)} Let us begin with the simplest setting of a monomial $P^n(z)=z^n$, whose zero distribution is equal to the Dirac measure $\mu_0=\delta_0$ for all $n\in\N$. The backward heat flow maps $P^n$ to a rescaled Hermite polynomial $P_t^n(z)=  (t/n )^{n/2} \mathrm{He}_n(z\sqrt {n/t})$, whose empirical zero distribution converges to the semicircular law $\mathsf{sc}_t$ of variance $t>0$ with Lebesgue density $x\mapsto \frac{1}{2\pi t}\sqrt{4t-x^2}$, supported on the interval $[-2\sqrt t, 2\sqrt t]$. For a general sequence of \emph{real-rooted} polynomials $(Q_n(z))_{n\in \N}$ having an asymptotic distribution of zeros $\mu_0$, the asymptotic zero distribution of the heat-evolved polynomial $e^{-\frac{t}{2 n}\partial_z^2}Q_n(z)$ is known to be the free convolution $\mu_0\boxplus\mathsf{sc}_t$ of the initial distribution $\mu_0$ with the semicircle law; see \cite{ZakharLeeYang,VW22} and the related results \cite[Proposition 4.3.10]{AGZ} and \cite[Theorem 3]{Cuenca}. However, the free convolution interpretation breaks down for non-real-rooted polynomials as in our setting and we aim to identify the novel limiting zero distribution $\mu_t$ on $\C$.
	
	\textsc{(ii)} More generally, consider the following open problem, see e.g.~\cite{hallho}:	
		\begin{quotation}
		\begin{center}	\emph{For an arbitrary distribution $\mu_0$ on $\C$ that is the limiting zero distribution of a sequence of degree-$n$ polynomials $(Q_n(z))_{n\in \N}$ as $n\to\infty$, what is the limiting zero distribution $\mu_t$ of the heat-evolved polynomial $e^{-\frac{t}{2 n}\partial_z^2}Q_n(z)$?}
				\end{center}
		\end{quotation}
Apart from distributions $\mu_0$ on $\R$, only the case of isotropic $\mu_0$ on $\C$ has been studied. For every isotropic $\mu_0$, it is possible to construct a natural sequence of polynomials $Q_n$ with independent coefficients having asymptotic zero distribution $\mu_0$, see \cite{KZ14}. In this setting,  the above question has been answered in \cite{heatflowrandompoly}. By making crucial use of the independence of coefficients, the authors obtain several descriptions of $\mu_t$, for instance as a push-forward of $\mu_0$ under a transport map $T_t$ and via a Hamilton-Jacobi PDE satisfied by the corresponding logarithmic potentials.

The present article proves analogous results in the case where the initial distribution $\mu_0=\frac 1 \alpha \sum_{j=1}^d\alpha_j\delta_{\lambda_j}$ is discrete and for the polynomial power $P^n$ of degree $\alpha n$. In particular, the present article answers the above question for a dense family of initial distributions $\mu_0$.

We expect that $\mu_t$ is essentially universal, i.e.~it depends only on the (global) limit zero distribution $\mu_0$ of the polynomials $Q_n$, see \cite[Conjecture 2.14]{heatflowrandompoly}. However, local clusters of initial zeros seem to arrange themselves in lines under the heat flow (see \cite[Figure 5]{heatflowrandompoly}). The zeros of polynomial powers in \eqref{eq:P^n} are highly clustered by definition, and the limiting distribution $\mu_t$ of $P_t^n$ in our setting will indeed be supported on one-dimensional curves.
	
\textsc{(iii)} Similar line-formation phenomena appear for repeated differentiation $\partial_z^{[t n]}Q_n(z)$ of high-degree polynomials $Q_n$, see \cite{BHS24,HK21,KabRep,OSteiner,diff-paper,GNV25}. Repeated differentiation of polynomial powers has been investigated in \cite{BHS24}. Assuming  $\alpha_1=\dots=\alpha_d=1$, these authors determined  the limiting zero distribution of $\partial_z^{[t n]}P^n(z)$ in terms of its Stieltjes transform. In this work, we will not merely adapt the saddle point method used in \cite{BHS24}, but we will deepen the analysis to prove significantly more explicit descriptions of $\mu_t$ in the generalized setting of multiplicities $\alpha_j\neq 1$.

Notably, high powers of polynomials also have their counterpart in random matrix theory via conditioning a zero of the characteristic polynomial having high multiplicity (i.e.~inducing a point charge into the Coulomb gas), see \cite{Fischmann,Webb,Deano,ByunSeoMeng,Kieburg}. For instance \cite[Theorem 2.7]{ByunSeoMeng} studies analogous high powers of the GinUE characteristic polynomial, by using a (Riemann-Hilbert-) steepest descent analysis in which the zero measure of the corresponding orthogonal polynomials also happen to be curve supported (the so called motherbody).

	\subsection*{Structure of the Paper}
In the upcoming Section 2, we introduce necessary technical tools and present our main results rigorously.
		
		The proof of the existence result is carried out in Section 3.
		We represent the heat-evolved polynomial as a contour integral, and analyze its asymptotic behavior via the saddle point method. This yields convergence of the logarithmic potential of the empirical zero distribution. Additional arguments lift it  to weak convergence of the distributions. Here, we follow and adapt the approach of \cite{BHS24} who studied the distribution of zeros of repeatedly differentiated polynomial powers. However,  our focus lies on probabilistic and analytic methods rather than on aspects of algebraic geometry.
	
	In Section 4, we will derive the explicit descriptions of $\mu_t$ for small and large $t$ and their connection to the semicircle law. These results are not available in the setting of repeated differentiation. In this second technical core of the paper, we shall precisely locate the saddle points by deriving their asymptotics in the regimes of small and large $t$. We use this to identify the saddle point with dominating contribution (which corresponds to the inverse of the transport map $T_t$ introduced in \cite{heatflowrandompoly}).

	 We finish the paper with the derivation of the PDEs and a simple example where all objects from the previous steps can be made explicit.
	
	\section{Main Results}
	\label{sec:main_results}
	Our first main result is the existence of a limiting measure. Recall that we are working in the setting of~\eqref{eq:P^n}, \eqref{eq:def_heat_flow}.  We  define the empirical zero distribution of a heat-evolved polynomial power $P^n_t$ as $\mu_{n,t}\coloneq\frac 1 {\alpha n}\sum_{\{\lambda:P_t^n(\lambda)=0\}}\delta_\lambda$, counting multiplicities.
	\begin{theorem}
		\label{thm:measure}
		For each fixed $t\in\C$, the empirical zero distribution $\mu_{n,t}$ of $P_t^n$ converges weakly to some compactly supported probability measure $\mu_t$ on the complex plane as $n\to\infty$. The limiting measure is singular with respect to the Lebesgue measure on $\C$.
\\
If $t=|t|e^{i\theta}\in\C$ and $\tilde P^n(z):=P^n(e^{i\theta/2}z)$ is the initial polynomial power with rotated zeros, then  $\mu_t$ is the push-forward of $\tilde \mu_{|t|}$ under the rotation map $z\mapsto z e^{i \theta/2}$, where $\tilde \mu_{|t|}$ is the limiting zero distribution of $\tilde P_{|t|}^n$.
	\end{theorem}
	
	Due to the last statement, there is no loss of generality in assuming that  $t>0$ for simplicity and we will do so in the sequel. (Note that for $t=0$ we have $P_t^n=P^n$ such that the statement of Theorem \ref{thm:measure} is immediate.)
	
	In the following, we will describe the limiting measure $\mu_t$ in more detail. However, since it is singular and, as we shall see, supported on a finite family of implicitly defined smooth curves, the explicit description becomes subtle. As announced in the introduction, we shall prove Theorem \ref{thm:measure} via the saddle point method, where the function
	\begin{align}\label{eq:function G}
		G(z,u)
		=\frac{1}{\alpha}\sum_{j=1}^d\alpha_j\log |u-\lambda_j|+\Re\left(\frac{(z-u)^2}{2t}\right)
	\end{align}
	will be crucial, and we will discuss its analytic properties in Section \ref{sec:sp_method}. Note that the logarithmic term reflects the contribution of the zeros of the polynomial $P$, while the quadratic term arises from the heat flow. We will provide more details later. The function $G(z,u)$ is the real part of the multivalued holomorphic function
	\begin{align} \label{eq:function g}
		u\mapsto g(z,u)=\frac{1}{\alpha}\sum_{j=1}^d\alpha_j\log(u-\lambda_j)+\frac{(z-u)^2}{2t},
	\end{align}
	whose saddle points are given by the equation $\partial_u g(z,u)=0$, which turns out to be
	\begin{align}\label{eq:sp_eqn_1}
		u+\frac{t}{\alpha}\sum_{j=1}^d\frac{\alpha_j}{u-\lambda_j}=z.
	\end{align}
		Denote by $u_t^1(z),\dots, u_t^{d+1}(z)$ the solutions to the saddle point equation \eqref{eq:sp_eqn_1} for all $z\in\C\setminus B$, where $B$ is the set of branch points $z$ for which \eqref{eq:sp_eqn_1} has solutions of multiplicity $>1$. Let $u_t^*(z)$ be one particular branch of the solution, to be defined later.
	Further, for fixed $t>0$, define the domain
	\begin{align*}
		\mathcal D_t:=\{z\in\C \setminus B: G(z, u_t^i(z))\neq G(z, u_t^j(z))\ \forall i\neq j\}.
	\end{align*}

	We will provide more details in Section \ref{sec:Saddlepointequation} below, see in particular Lemma \ref{lem:domain} and Lemma \ref{lem:saddle}. Define the logarithmic potential of $\mu_t$ by
	\begin{align*}
		U_t(z)\coloneq \int_\C \,\log |z-\xi|\,d\mu_t(\xi),\quad z\in\C
	\end{align*}
	and its Stieltjes transform by
	\begin{align*}
		m_t(z)\coloneq  \int_\C\,\frac{\mu_t(dy)}{z-y}, \quad z\in\mathcal D_t.
	\end{align*}
	Recall the definition of the Wirtinger derivatives
	\begin{align*}
		\partial_z =\frac{1}{2}\Bigl(\frac{\partial}{\partial x}-i\frac{\partial}{\partial y} \Bigr)\quad\text{and}\quad
		\partial_{\bar z}=\frac{1}{2}\Bigl(\frac{\partial}{\partial x}+i\frac{\partial}{\partial y} \Bigr),
	\end{align*}
	for $z=x+iy$, $x,y\in\R$.
It is well known that $m_t(z)=2\partial_z U_t(z)$ in the sense of distributions and that $\mu_t$ can be represented via the distributional Laplacian $\Delta=4\partial_z\partial_{\bar z}$ as
	\begin{align*}
		\mu_t=\frac  1 {2\pi}\Delta U_t.
	\end{align*}
In particular, the logarithmic potential and the Stieltjes transform determine $\mu_t$ uniquely as follows.

\begin{theorem}\label{thm:description}
For any $t>0$, the limiting distribution $\mu_t$ satisfies the following.
\begin{enumerate}
\item Its Stieltjes transform $m_t$ is given by
\begin{align}\label{eq:mtut}
m_t(z)=\frac 1 t (z-u^*_t(z)), \quad z\in\mathcal D_t
\end{align}
and solves the equation
\begin{align}\label{eq:selfcons_eqn}
m_t(z)=\frac 1 \alpha \sum_{j=1}^d\frac{\alpha_j}{z-tm_t(z)-\lambda_j}, \quad z\in\mathcal D_t .
\end{align}
\item Its logarithmic potential satisfies
\begin{align}\label{eq:U_t=G}
U_t(z)=G(z,u_t^*(z)), \quad z\in\mathcal D_t .
\end{align}
\item It is compactly supported on finitely many smooth curves (plus endpoints) within
 $$\supp(\mu_t)\subseteq\mathcal D_t^c= \{z\in\C\setminus B: G(z,u_t^i(z))=G(z,u_t^j(z))\text{ for some }1\le i< j\le d+1\}\cup B.$$
\end{enumerate}
\end{theorem}

Let us provide some explanatory remarks.

\begin{remark}
	\label{rem:example_hermite}
In the most simple case $d=\alpha_1=1$ of $P^n(z)=(z-a)^n$ for some $a\in\C$, it holds that $P_t^n=(\tfrac{t}{n})^{n/2}\mathrm{He}_n(\sqrt n (z-a) /\sqrt t)$. Hence, the limiting distribution $\mu_t=\mathsf{sc}_t(\cdot-a)$ is the (shifted) semicircle distribution supported on a straight line $\R+a$ and centered at $a$. In this case, \eqref{eq:selfcons_eqn} turns into the well-known self-consistent equation of the semicircle law. See Section \ref{sec:example} for more details on this example.
\end{remark}

\begin{remark}
In terms of free probability theory, we may write that if $d\geq 1$ and all $\lambda_j\in\R$, then $\mu_t$ is the free convolution $\mu_0\boxplus \mathsf{sc}_t$, see \cite{ZakharLeeYang,VW22}. Indeed, plugging \eqref{eq:mtut} into \eqref{eq:selfcons_eqn} yields the so-called subordination representation $m_t(z)=m_0(u_t^*(z))$ of this free convolution, with subordination function $u_t^*(z)$ being the inverse of $z\mapsto z+tm_0(z)$ according to \eqref{eq:mtut}, see \cite[Proposition 2]{Biane}.

If $\lambda_j\notin \R$, then this free probability interpretation fails: Adding free semicircular non-commutative random variables to an element with Brown measure $\mu_0=\frac 12(\delta_0+\delta_i)$ corresponds to \cite[\S 10.2]{hallhobrown} which, in contrast to our measure $\mu_t$, has been proven to be absolutely continuous. This is consistent with \cite[Theorem 6.1]{heatflowrandompoly}, where it has been shown that the heat flow for complex limit distributions does not correspond merely to adding a free semicircular component, but also involves removing a semicircular component in the imaginary direction.
\end{remark}

\begin{remark}
The case of large $d\to\infty$ is formally consistent with results known in the literature. 
	Suppose that $\mu_{P^n}\Rightarrow\nu_0$ as $d,n\to \infty$ for some absolutely continuous $\nu_0$ on $\C$, which corresponds to the random polynomials setting of \cite{heatflowrandompoly}. Formally, we expect $u_t^*(z)$ to converge to some limiting object (corresponding to the inverse transport map $T_t^{-1}=u_t^*$ therein). Then, $m_t(z)=m_0(u_t^*(z))$ continues to hold for $d\to\infty$. It can be shown that this implies $\mu_{P_t^n}\Rightarrow\nu_t$ as $d,n\to\infty$ for the push-forward $\nu_t=\nu_0(u_t^*(\cdot))$ as conjectured in \cite{heatflowrandompoly}. In the notion of transport maps from that paper, the equation $U_t(T_t(w))=U_0(w)+\frac{t}{2}\mathrm{Re}(m_0(w)^2)$ corresponds to \eqref{eq:U_t=G}.
	However, the choice of the branch $u_t^*$ is based on topological arguments, which become increasingly difficult to track under growing $d$. We aim to return to this problem in a forthcoming work.
\end{remark}

Thus, we may imagine the distribution $\mu_t$ to be a free convolution of semicircle distributions supported on some of the curves in $\mathcal{D}_t^c$. Indeed, zooming into initial zeros $\lambda_j$ as $t\to 0$, we see semicircle laws appearing nearby. For any distribution $\mu$, we write $\mu(c\cdot)$ for the push-forward of $\mu$ under the dilation $z\mapsto z/c$.

\begin{theorem}\label{thm:small_t}
	Fix $j\in\{1,\dots , d\}$ and define $T^j: z\mapsto \frac{z-\lambda_j}{\sqrt{t\alpha_j/\alpha}}$. Then we obtain vague convergence of the push-forward $ T_{\#}^j\mu_t:=\mu_t(\lambda_j +\cdot \sqrt{{t\alpha_j}/{\alpha}})\to \frac {\alpha_j} {\alpha} \mathsf{sc}_1$ as $t\to 0$.
\end{theorem}

For arbitrary times, the geometry of the curves in $\mathcal{D}_t^c$ is difficult to analyze, see Figure \ref{fig:heat_flow}. For large values of $t$, they merge into a single horizontal line as follows. 

\begin{theorem}\label{thm:large_t}
Fix a closed disk $K\subseteq \C$ centered at $\frac 1  \alpha \sum_{j=1}^d\alpha_j\lambda_j$.  Then, for sufficiently large $t>0$,
\begin{align}\label{eq:U=max}
	U_t(z)=\max_{j=1,\dots,d+1}G\big(z,u_t^j(z)\big), \quad z\in K
	\end{align}
	and	$\supp(\mu_t)\cap K$
	consists of a single smooth curve that converges, as $t\to\infty$, in Hausdorff metric to a horizontal segment through the center of mass $\frac 1  \alpha \sum_{j=1}^d\alpha_j\lambda_j$.
\end{theorem}

It was proven in \cite{heatflowrandompoly} that for a different model of random polynomials with independent coefficients, the limiting distribution after the heat flow will always collapse to the semicircle distribution on the real line, after some finite threshold $t_{\mathrm{sing}}>0$. In our case, we do not expect such a behavior, since it follows from Theorem \ref{thm:large_t} that the curves in $\mathcal{D}_t^c$ supporting mass of $\mu_t$ won't even converge to the real line. Properly rescaled however, we identify a semicircle distribution in the ``macroscopic/global'' limit.

\begin{theorem}\label{thm:semicircle}
As $t\to\infty$, the rescaled distribution $\mu_t(\sqrt t\cdot )$ converges weakly to the standard semicircle distribution $\mathsf{sc}_1$.
\end{theorem}

Since $\mu_t,m_t,$ and $U_t$ are dynamic objects in the time parameter $t>0$, it is natural to ask for an evolution equation, which describes the hydrodynamic limit of the particles given by the zeros of $P_t^n$. The following theorem should be considered as a counterpart to Theorem 5.2 in \cite{heatflowrandompoly}.

\begin{theorem}
\label{thm:PDEs}
Let $\mathcal{D}\coloneq \{(z,t)\in\C^2: t=|t|e^{i\theta}\neq 0, e^{-i\theta/2} z \in \mathcal{D}_{|t|}\}$.
\begin{enumerate}
\item The logarithmic potential $U_t$ satisfies the Hamilton-Jacobi PDE
\begin{align}\label{eq:HJ-PDE}
\partial_{t}U_t(z)=-(\partial_z U_t(z))^2, \quad (z,t)\in\mathcal{D},
\end{align}
where $\partial_t$ is the Wirtinger derivative of $t\in\C$.
\item The Stieltjes transform $m_t$ satisfies the PDEs
\begin{align}\label{eq:Burgers-PDE}
\partial_t m_t(z)&=-\frac{1}{2}\partial_z(m_t^2(z))=-m_t(z)\partial_z m_t(z), \quad (z,t)\in\mathcal{D},\\
\partial_{\bar t}m_t(z)&=-\frac{1}{2}\partial_z\overline{(m_t^2(z))}=-\overline{m_t(z)}\partial_z \overline{m_t(z)}, \quad (z,t)\in\mathcal{D} .
\end{align}
\end{enumerate}
\end{theorem}

\begin{remark}
Part \emph{(2)} states that the Stieltjes transform of the \mbox{measure} $\mu_t$ fulfills a real and complex version of the (generalized) inviscid Burgers' equation. The inviscid Burgers' equation is known to develop shock waves, which are discontinuities that are precisely described by the curves within $\supp(\mu_t)\subseteq\mathcal D_t^c$. Note however, that the real part of \eqref{eq:HJ-PDE} seen as a PDE in real coordinates $(\Re(z),\Im(z))$ corresponds to a non-convex Hamiltonian due to the negative imaginary part in the  definition of the Wirtinger derivative $\partial_z$. Hence, standard results in the theory of Hamilton-Jacobi equations do not apply in our setting (like Hopf-Lax solutions, see \cite[Remark 4.4]{heatflowrandompoly}).
\end{remark}

\begin{remark}
	The emergence of rather horizontal lines of zeros under the heat flow can heuristically be understood from the point of view of particle systems as follows.
	Consider a heat-evolved polynomial $\eee^{-\frac t{2n} \partial_z^2} P_n(z)$ of degree $n$ and suppose that the roots of the polynomial $P_n$ are distinct (which is clearly not the case for our $P^n$, but which is true for $P^n_\varepsilon$ for any small $\varepsilon>0$, see Remark \ref{rem:distinct_zeros} below).
	For sufficiently small $t$, if we denote the zeros as locally analytic functions $z_1(t), \ldots, z_n(t)$ in $t$, it is well known (see e.g.~\cite{csordas_smith_varga,tao_blog1,tao_blog2,hallho,heatflowrandompoly}) that they satisfy the system of ordinary differential equations
	\begin{align}\label{eq:ODE_for_poly}
		\partial_t z_j(t) =  \frac 1 n\sum_{ k\neq j} \frac{1}{z_j(t) - z_k (t)},
		\qquad
		j=1,\ldots, n.
	\end{align}
	In particular, looking at the distance of two zeros, say $\Delta(t)=z_1(t)-z_2(t)$, it follows that
	\begin{align}
		\partial_t  \Delta(t)=& \frac{2\overline{\Delta(t)}}{n|\Delta(t)|^2}-\frac 1 n\sum_{k=3}^n\frac{\Delta(t)}{(z_1(t)-z_k(t))(z_2(t)-z_k(t))}.
	\end{align}
	Suppose now that $z_1(t)$ and $z_2(t)$ are much closer than they are to all other $z_k(t)$, then the first term dominates. Hence, $\partial _t\Re(\Delta(t))$ has the same sign as $\Re(\Delta(t))$, thus $\Re(z_1(t))$ and $\Re(z_2(t))$ repel each other. On the contrary,  $\partial_t\Im(\Delta(t))$ has opposite sign of $\Im(\Delta(t))$, implying attractive behavior between $\Im(z_1(t))$ and $\Im(z_2(t))$. This explains the tendency of $z_1(t)$ and $z_2(t)$ to locally align horizontally. However, the second term of the majority of roots $z_k(t)$ is often non-negligible and adds a global flow into the direction of the initial Stieltjes transform as explained in \cite{heatflowrandompoly}.
\end{remark}

\section{Existence of a limiting distribution}
In this section we will prove Theorem \ref{thm:measure}.

\subsection{Saddle point method}
We begin with a short reminder of the saddle point method and recall the necessary statement for later use as stated in \cite[Corollary 1.4]{O_Sullivan_2019}.

We make the following assumptions.
\begin{enumerate}[(i)]
\item Let $g:\C\to\C$ be holomorphic in a neighborhood $O$ of a simple curve $\gamma\subset \C$.
\item Let $u^*$ be a non-degenerate saddle point of $g$, i.e.  $\partial_u g(u)|_{u=u^*}=0$ and $\partial_u^2 g(u)|_{u=u^*}\neq 0$.
\item The chosen contour $\gamma$ passes smoothly through the saddle point $u^*$ of $g$ and $u^*$ is not an endpoint of $\gamma$.
\item $u^*$ uniquely maximizes $\Re\,g$ on $\gamma$, i.e.~$\Re\,g(u)<\Re\,g(u^*)$ for all $u \in\gamma$ such that $u \neq u^*$.
\end{enumerate}

\begin{theorem}[Saddle point method]

\label{thm:sp_method}
Under the assumptions above, we have
\begin{align*}
\lim_{N\to\infty}\frac 1 N\log\left\lvert\int_{\gamma}\,e^{Ng(u)}\,du\right\rvert=\Re g(u^*).
\end{align*}
\end{theorem}

In order to apply the method, we need an integral representation of the heat-evolved polynomials $P_t^n(z)$. Recall that the  heat flow on the real line can be expressed as a convolution of the initial datum with a Gaussian kernel. The following is an analogous representation  for the holomorphic   heat flow we are interested in.

\begin{lemma}
\label{lem:int}
Let $t >0$. Then,
\begin{align}\label{eq:int}
P_t^n(z)=e^{-\frac{t}{2\alpha n}\partial_z^2}P^n(z)=\frac{-i}{\sqrt{2\pi \frac{t}{\alpha n}}}\int_{i\mathbb{R}}\, \exp \left(\frac{(z-u)^2}{2\frac{t}{\alpha n}}\right)P^n(u)\,du \quad \forall z\in\C.
\end{align}
\end{lemma}

\begin{proof}
See \cite[Theorem 2.3]{GAF-paper}. 
\end{proof}

We rewrite the  integral appearing in Lemma \ref{lem:int} as
\begin{align*}
\int_{i\mathbb{R}}\, e^{\frac{(z-u)^2}{2\frac{t}{\alpha n}}}P^n(u)\,du=\int_{i\mathbb{R}}\, e^{\alpha n\,\bigl(\frac{1}{\alpha}\sum_{j=1}^d\alpha_j\log(u-\lambda_j)+\frac{(z-u)^2}{2t}\bigr)}\,du=\int_{i{\R}}\,e^{\alpha n\,g(z,u)}\,du,
\end{align*}
where the function $g$ is given by
\begin{align*}
g(z,u)&=  \frac{1}{\alpha}\sum_{j=1}^d\alpha_j\log(u-\lambda_j)+\frac{(z-u)^2}{2t},\\
\intertext{and its real part will be denoted by}
G(z,u)&= \Re\,g(z,u)=\frac{1}{\alpha}\sum_{j=1}^d\alpha_j\log |u-\lambda_j|+\Re\left(\frac{(z-u)^2}{2t}\right),
\end{align*}
as introduced in \eqref{eq:function g} and \eqref{eq:function G}.
Note that, for fixed $z$, the function $u\mapsto g(z,u)$ is multi-valued analytic on $\C\setminus \{\lambda_1,...,\lambda_d\}$, whereas $u\mapsto G(z,u)$ is well defined (single-valued) and harmonic on $\C\setminus \{\lambda_1,...,\lambda_d\}$.

\subsection{Analysis of the saddle point equation}\label{sec:Saddlepointequation}
In order to use the saddle point method, we need to identify the critical points of the function $g$, i.e.~solve the saddle point equation given by
\begin{align}
\partial_u g(z,u)=\frac{1}{\alpha }\sum_{j=1}^d\frac{\alpha_j}{u-\lambda_j}-\frac{z-u}{t}&=0 \notag \\
\iff u+\frac{t}{\alpha}\sum_{j=1}^d\frac{\alpha_j}{u-\lambda_j}&=z.\label{eq:sp_eqn_2}
\end{align}

Note that for fixed $z$, there are at most $d+1$ such saddle points, since the polynomial
\begin{align}
	\label{eq:polynomial Q}
	Q_{z,t}(u):=\frac t \alpha \sum_{j=1}^d \alpha_j\prod_{k\neq j}(u-\lambda_k)-(z-u)\prod_{j=1}^d(u-\lambda_j)
\end{align}
 has degree $d+1$ in $u$. Since we assumed all $\lambda_1,...,\lambda_d$ to be different, $u=\lambda_j$ cannot be a solution to \eqref{eq:polynomial Q} for all $j\in\{1,...,d\}$. Hence, \eqref{eq:sp_eqn_2} and $Q_{z,t}(u)=0$ are equivalent.

We define the branching locus $B:=\{z\in\C:Q_{z,t}(u)\text{ has a root of multiplicity} >1\}$ (which implicitly depends on $t$) and denote by $u_t^1(z),...,u_t^{d+1}(z)$, for $z\notin B$, the solutions to the saddle point equation (\ref{eq:sp_eqn_2}). Note that these are well defined even for complex $t\in\C$ and that their numbering is not canonical. For fixed $t\in\C$, we can define $u_t^1(z),...,u_t^{d+1}(z)$ as analytic functions in $z$ on any simply connected subdomain of $\C\setminus B$ and the same holds in both variables $(z,t)$. We denote the open ball by $B_r(z)= \{w\in\C : |w-z|<r \}$.

\begin{lemma}\label{lem:analyt}
The saddle points $(z,t)\mapsto u_t^j(z)$, $j=1,\dots,d+1$, are analytic within any simply connected domain of $\{(z,t)\in\C^2: z\notin B\}$. Moreover, fix $t\in\C$, $z_0\in B$, and let $V=B_r(z_0)\setminus[z_0,\infty)$ be some slit neighborhood of $z_0$ that contains no other branch points and on which, for some $k\le d+1$ and $\eps>0$, we have $|u_t^k(z)-u_t^j(z)|>\eps$ for all $j\neq k$. Then $u_t^k(z)$ can be analytically extended to $z_0$.
\end{lemma}
\begin{proof}
The function  $(t,z,u)\mapsto Q_{z,t}(u)$ is a polynomial in $t,z,u$ and its derivative in $u$ is invertible except at the branching points. By the implicit function theorem, for each $j\le d+1$, the function $u^j_t(z)$ is locally analytic in $(z,t)\in\C^2$, as long as $z\notin B$, and can be analytically extended along paths within every simply connected domain of $\{(z,t)\in\C^2: z\notin B\}$. 

For the second claim, assume that there is a non-trivial permutation of the branches around $z_0$ such that the branch $u_t^k(z)$ switches at the slit $[z_0,\infty)$. Then, $u_t^k(z)$ has different values as we approach $z\to\tilde z\in [z_0,\infty)$ from opposite sides, hence there exists some $j\neq k$ such that $u_t^j(z)$ and $u_t^k(z)$ coalesce at $\tilde z$. This contradicts the assumption that $|u_t^k(z)-u_t^j(z)|>\eps$, and therefore $u_t^k(z)$ can be analytically extended on the slit (there is no monodromy). Within $z\in B_r(z_0)\setminus\{z_0\}$, all solutions of \eqref{eq:sp_eqn_2} are bounded and so is $u_t^k(z)$, hence $z_0$ is a removable singularity and $u_t^k(z)$ can be analytically extended to $z_0$ by Riemann's Extension Theorem.
\end{proof}

\begin{lemma}\label{lem:branching}
For any $t\in\C$ we have $|B|\le 2d$.
\end{lemma}
These $z\in B$ are precisely those points such that $\partial_u g=\partial_u^2g=0$, i.e.~such that solutions $u(z)$ of \eqref{eq:sp_eqn_2} are degenerate saddle points. We expect the branch points to be the natural endpoints of the curves forming the support of $\mu_t$ and we will make this rigorous for sufficiently small $t>0$ in Corollary \ref{cor:small_t} below.

The claim also follows from the Riemann-Hurwitz formula \cite[\S 9.6]{Fischer}.\footnote{More precisely, viewing $u\mapsto \partial_u g(0,u)$ as a rational map of degree $d+1$ between Riemann spheres having genus $0$, the Riemann-Hurwitz formula yields a total branching order of  $-2-2(d+1)(-1)=2d$.} Nevertheless, let us provide a self-contained analytic proof that serves as a preliminary step for the subsequent arguments involving Rouch\'e's Theorem.

\begin{proof}
The set $B$ consists precisely of those $z\in\C$ such that $Q_{z,t}(u)$ has a root of multiplicity $>1$. 
Suppose $t\neq 0$, since for $t=0$ the claim follows immediately from the fact that $\lambda_1,\dots,\lambda_{d}$ are different.
By definition of the resultant $R(z):=\mathrm {res}_u(Q_{z,t},Q_{z,t}')$, $Q_{z,t}(u)$ can have a root of multiplicity $>1$ if and only if $R(z)=0$. The resultant $R$ is a homogeneous polynomial of degree $2d+1$ in the coefficients $c_k(z,t),\,k=0,\dots, d+1$, of $Q_{z,t}$, each of which is affine linear in $z$ and $t$, and is defined as the determinant of the Sylvester matrix with the first column only containing two non-zero entries, given by the leading coefficients of $Q_{z,t}$ and $Q_{z,t}'$. Now, since the leading coefficients of both $Q_{z,t}$ and $Q_{z,t}'$ do not depend on $z$ (and $t$), $R(z)$ is a polynomial of degree at most $2d$ in $z$ (and $t$). It follows that for each fixed $t\neq 0$ we have at most $2d$ branching points $z\in B$, unless $R(z)$ is the constant polynomial equal to $0$. However, we exclude this possibility by showing that for sufficiently large $|z|$, we have $d+1$ different zeros of $Q_{z,t}$, and hence $R(z)$ is not identically  $0$.

Using Rouch\'e's theorem, we will now verify that for large $|z|$ there are $d$ saddle points close to $\lambda_1,\dots,\lambda_d$ and one saddle point close to $z$.
Fix  $0<\varepsilon<\frac 1 2 \min_{i,j\le d}|\lambda_j-\lambda_i|$. Then, by Rouch\'e's theorem,  the functions $u\mapsto (u-z)$ and $u\mapsto (u-z)+\frac t \alpha \sum_i \frac{\alpha_i}{u-\lambda_i}$ have the same number of zeros in $B_\varepsilon(z)$ since
$$\left\lvert\frac {t}{\alpha}\sum_i \frac{\alpha_i}{u-\lambda_i}\right\rvert\sim \left\lvert\frac t u\right\rvert<\varepsilon =|u-z| ,\qquad u\in \partial B_\varepsilon (z)$$
as $|z|\to \infty$ (and hence $|u|\to\infty$). Consequently, there is one saddle point, denoted by  $u_t^{d+1}(z)$, at distance  $< \varepsilon$ to $z$. Let $1\leq j \leq d$ be fixed. Again  by Rouch\'e's theorem, the functions  $u\mapsto (z-u)\prod_i(u-\lambda_i)$ and $u\mapsto \frac t \alpha \sum_i\alpha_i\prod_{k\neq i}(u-\lambda_k)-(z-u)\prod_i(u-\lambda_i)$ have the same number of zeros in $B_\varepsilon(\lambda_j)$ since
$$\left|(z-u)\prod_i(u-\lambda_i)\right|>\left|\frac t \alpha \sum_i\alpha_i\prod_{k\neq i}(u-\lambda_k)\right| ,\qquad u\in \partial B_\varepsilon (\lambda_j)$$
for $|z|$ sufficiently large. (Indeed, the left-hand side diverges to $\infty$ as $|z|\to \infty$, while the right-hand side stays bounded).  Hence, for every $1\leq j \leq d$, there exists a saddle point  $u_t^{j}(z)$ at distance  $< \varepsilon$ to a corresponding $\lambda_j$. By the choice of $\eps$, all these saddle points are different.
\end{proof}

\begin{remark}\label{rem:distinct_zeros}
Further, note that $P_t^n$ has simple roots for all except finitely many $t>0$. Indeed, the resultant $t\mapsto \tilde R(t)=\mathrm{res}_z (P_t^n(z),\partial_z P_t^n(z))$ is a polynomial in $t$ (that is not identically $0$, because its leading coefficient is nonzero),
and hence its set of zeros is finite.
Note that the branching points are also algebraic functions of $t$, as they are solutions of the equation $\mathrm {res}_u(Q_{z,t},Q_{z,t}')=0$ in the variable $z$.
\end{remark}

In the following, we restrict ourselves to the domain on which all saddle points are of different heights:
\begin{align*}
\mathcal D_t=\{z\in\C\setminus B : G(z, u_t^i(z))\neq G(z, u_t^j(z))\ \forall i\neq j\}.
\end{align*}

Not only has $\mathcal D_t$ full Lebesgue measure, but it is a union of domains separated by analytic curves.

\begin{lemma}\label{lem:domain}
For any fixed $t\in\C$, the complement of the domain $\mathcal D_t$, given by
 $$\mathcal D_t^c= \{z\in\C\setminus B: G(z,u_t^i(z))=G(z,u_t^j(z))\text{ for some }1\le i< j\le d+1\}\cup B,$$  consists of countably many real analytic curves. Every closed disk $K$ intersects only finitely many curves. The curves may only terminate at branch points or the boundary $\partial K$.
\end{lemma}

\begin{remark}
	The curves in $\mathcal D_t^c$ introduced above also arise in the context of the Stokes phenomenon, where they are known as anti-Stokes curves. The Stokes phenomenon describes the sudden change of the so-called Stokes multiplier, which determines the dominance of a saddle point in the saddle point method, when crossing an anti-Stokes line, see \cite[Section 1]{berry1989}. In close analogy, we observe that switches of the dominant saddle point in our setting can occur only upon crossing a curve in $\mathcal D_t^c$. For further background on the Stokes phenomenon, we refer the reader to \cite{berry1989} and \cite{berry1988}.
\end{remark}

The local finiteness of the number of curves may be well known to readers interested in concepts of analytic geometry, but we shall include a proof for completeness. The uninterested reader may skip the proof at first reading.

\begin{proof}
	By Lemma \ref{lem:analyt} and the definition \eqref{eq:function g} of $g$, $z\mapsto g(z,u_t^j(z))$ is analytic on simply connected domains of $\C\setminus B$, hence $z\mapsto G(z,u_t^j(z))$ is harmonic for each $j\le d+1$, and we are interested in the nodal set of the harmonic function $G(z,u_t^i(z))-G(z,u_t^j(z))$.
	Note that the complex chain rule for $\partial_z$ applied to the function $g$ yields
	\begin{align*}
		\partial_z (g(z,u^j_t(z)))&=(\partial_u g)(z, u^j_t(z)) \partial_z u^j_t(z)+(\partial_z g)(z,u^j_t(z))\\
		&=(\partial_z g)(z,u^j_t(z))=\frac 1 t (z-u_t^j(z))
	\end{align*}
	for any saddle point $u_t^j(z)$.
	It follows that there is no critical point $z\in\C\setminus B$ of
	\begin{align*}
		0=\partial_z \Big(g(z,u_t^i(z))-g(z,u_t^j(z))\Big)=(\partial_x-i\partial_y) \Big(G(z,u_t^i(z))-G(z,u_t^j(z))\Big)
	\end{align*}
	(this set of critical points consists precisely of the branching points $z\in B$ such that $u_t^i(z)=u_t^j(z)$).
	Note that $G(z,u_t^i(z))-G(z,u_t^j(z))=0$ cannot hold for all $z\in\C\setminus B$. Indeed, if it would hold for all $z\in B_r(z_0)\subseteq\C\setminus B$ for some open ball $B_r(z_0)$ and some $i\neq j$, then differentiation as above would imply $u_t^i(z)=u_t^j(z)$, which contradicts $z\in\mathcal{D}_t$.
	
	Hence, we may again apply the implicit function theorem to $G(x+iy,u_t^i(x+iy))-G(x+iy,u_t^j(x+iy))$, since either $\partial_x$ or $\partial_y$ is invertible, and its zero set is locally given by a curve $y(x)$ (or $x(y)$) for each $i\neq j$. In particular, no such curve can have endpoints in $\C\setminus B$. Moreover, the nodal set does not intersect itself in $\C\setminus B$, since otherwise it would fail to be locally analytic at the intersection points.
	
	We shall now show that for fixed $i,j$ only finitely many curves within the nodal set of $G(z,u_t^i(z))-G(z,u_t^j(z))$ may start at a given branching point $z_0\in B$, and that they emerge as equiangular rays. The order (or, ramification index) $r\in\N$ of the branch point $z_0$ is given by the number of branches $u_t^k(z)$ which coalesce at $z_0$ to, say, $u^0=u_t^k(z_0)$ for $k=1,\dots,r$. This value $r$ is equal to the multiplicity of the root $u^0$ of $Q_{z_0,t}(u)$ hence $2\le r\le d+1$. Using Weierstra\ss\ factorization, we can write $Q_{z,t}(u)=\tilde Q(z-z_0,u-u^0) {\tilde Q}^0 (z-z_0,u-u^0)$, where the polynomial ${\tilde Q}^0$ is non-vanishing in some neighborhood of $(0,0)$. Following e.g.~\cite[\S 6]{Fischer}, $\tilde Q(\tilde z,\tilde u)$ is an irreducible Weierstra\ss\ polynomial in $\tilde u$ of degree $r$, i.e.~$\tilde Q(0,\tilde u)=\tilde u ^r$. By Puiseux's theorem \cite[Section 7.8]{Fischer}, there exists an analytic function $\varphi$ defined in some neighborhood of $0$ such that $\tilde Q(\tilde z^r,\varphi(\tilde z))=0$ and $\tilde Q (\tilde z^r,\tilde u)=\prod_{k=1}^{r}(\tilde u-\varphi(\zeta_r^{k-1}\tilde z))$ for $\zeta_r=\eee^{2\pi i /r}$, see \cite[Section 7.10]{Fischer}. Now, choose some branch of $z\mapsto (z-z_0)^{1/r}=\tilde z$ in a slit neighborhood $V$ of $z_0$, then
	\begin{align}\label{eq:Puiseux}
		u_t^k(z)=u^0+\varphi(\zeta^k (z-z_0)^{1/r})\approx u^0+ \varphi'(0)\zeta^k (z-z_0)^{1/r} , \quad z\in V, k=1,\dots,r,
	\end{align}
	where $\approx$ is due to a Taylor expansion of $\varphi$ omitting higher order terms $\mathcal O (z-z_0)^{2/r}$.
	This is the analytic expression of the formal Puiseux expansion around $z_0$.
	If $u^k_t(z)$, for $k=r+1,\dots,d+1$, is one of the branches which do not coalesce with any of the other branches at $z_0\in B$, then the expansion becomes a usual Taylor expansion $u_t^k(z)= u^k_t(z_0)+ \partial_z u^k_t(z_0)  (z-z_0)+\mathcal O(z-z_0)^2$ for $z\in V$.
	
	We are interested in $G(z,u_t^i(z))-G(z,u_t^j(z))=0$ in the neighborhood of $z_0\in B$, where both $u^i,u^j$ coalesce into $u^0$, i.e.~$i,j\le r$. Note that if they don't, then the nodal set is again an analytic curve (which can be interpreted as two equiangular rays emerging from $z_0$), which follows from a similar argument that we shall comment on promptly.
	A Taylor expansion of $g(z,u)$ in $(z,u)$ near $(z_0,u^0)$ involves $\partial_u^\ell g(z_0,u^0)=0$ for $\ell=1,\dots r$, $\partial_u\partial_zg\equiv -\frac 1 t$, and higher order terms, which will turn out to be negligible. The zeroth-order terms in the difference $g(z,u_t^i(z))-g(z,u_t^j(z))$ cancel so that inserting \eqref{eq:Puiseux} gives the expansion in $z\in V$
	\begin{align}
		&g(z,u_t^i(z))-g(z,u_t^j(z))\nonumber \\
		=& \frac {\partial_u^{r+1}g(z_0,u^0) } {(r+1)!}\big( (u_t^i(z)-u^0)^{r+1} -(u_t^j(z)-u^0)^{r+1}\big) -\frac 1 {t} (z-z_0)(u_t^i(z)-u_t^j(z)) +\mathcal O (z-z_0)^{\frac{r+2}r}\nonumber\\
		=&  c_{ij} (z-z_0)^{\frac{r+1}{r}} +\mathcal O (z-z_0)^{\frac{r+2}r}, \quad \text{ for }c_{ij}= \varphi'(0) (\zeta^i-\zeta^j)\Big(\frac {\partial_u^{r+1}g(z_0,u^0)} {(r+1)!}\varphi'(0)^{r}-\frac 1 {t}\Big),\label{eq:G_Expansion}
	\end{align}
	where we used $\zeta_r^r=1$. (Temporarily coming back to the case that $u^i_t(z_0)\neq u_t^j(z_0)$, we see that the leading order term in the second line is $z-z_0$, implying that the branches do not coalesce at $z_0$.)
	Choosing $z\in V$ such that $G(z,u_t^i(z))-G(z,u_t^j(z))=0$, we take the real part of the above, implying
	$$\Re \big(c_{ij}(z-z_0)^{\frac{r+1}{r}}\big)\approx 0 \Leftrightarrow \cos\big( \tfrac {r+1}r \arg(z-z_0)+\arg(c_{ij})\big)\approx 0
	$$ locally around $z_0$.
	Again, $\approx$ approximates up to lower order terms  $(z-z_0)^{(r+2)/r}$. Consequently, $z$ must lie on one of the $r+1$ rays starting in $z_0$ in the direction of
	$$\arg(z-z_0)= \pi \frac r {r+1}\ell +\frac{r}{r+1}\big(\frac\pi 2 -\arg (c_{ij})\big),\quad \text{ for }\ell=0,\dots, r.$$
	As claimed, these angles are equidistant and, for $\ell=r+1$, we obtain the same angle as for $\ell=0$ up to an integer multiple of $\pi$. Finally, since $r\le d+1$, we have at most $d+2$ rays starting from each of the at most $2d$ branching points and the nodal set contains at most $2d(d+2)$ curves. A union bound over all possible pairs $i,j$  contributes a factor of $d(d+1)/2$, yielding at most $d^2(d+1)(d+2)$ many curves within $\mathcal D_t^c$, each connecting a branching point either to another branching point or to infinity.
	
	Moreover, the maximum principle for the harmonic function $G(z,u_t^i(z))-G(z,u_t^j(z))$ implies that its nodal set cannot have self-loops in any simply connected subset of $\C\setminus B$. The same reasoning also excludes loops from $z_0\in B$ back to itself, two different curves between a pair of branch points, or any other collection of curves within the nodal set encircling any bounded area. Indeed, suppose there exists a curve $\gamma$ encircling a bounded domain $D$, then  $G(z,u_t^i(z))-G(z,u_t^j(z))$ is harmonic on $D \setminus \cup_{z_0\in B}\overline{B_\varepsilon (z_0)}$ for some small $\varepsilon>0$ and attains its maximum on the boundary. However, its value is zero on $\gamma$ and of small order $c\varepsilon$ on $\partial B_\varepsilon (z_0)$, for some $c>0$, by \eqref{eq:G_Expansion}. Hence, $\sup_{z\in D}|G(z,u_t^i(z))-G(z,u_t^j(z))|\le c\varepsilon$ for all $\varepsilon>0$, which yields a contradiction.
	
	Finally, let us take some compact disk $K\subseteq \C$ and fix $i$ and $j$ again. Then, by the Puiseux expansion, we only have finitely many arcs in $K\cap B_\varepsilon (z_0)$ for all $z_0\in B$. Within a distance from branching points, we take some neighborhood $V_z$ of each $z\in K\setminus \cup_{z_0\in B} {B_\varepsilon(z_0)}$ where $\Gamma_{i,j}^{V_z}$ consists of at most a single curve (by the implicit function theorem above). By compactness, we may choose a finite subcover of $K\setminus \cup_{z_0\in B} {B_\varepsilon(z_0)}\subseteq \cup_{z}V_z$. Thus, there are only finitely many curves within $\Gamma_{i,j}^D$ for any simply connected $D\subseteq K\setminus B$. Taking the union over all $1\le i<j\le d+1$ and $B_\varepsilon (z_0)$, $z_0\in B$, we obtain that also $\mathcal D_t^c\cap K$ is a finite union of real analytic curves.
\end{proof}

\begin{remark}
	The proof also reveals an explicit non-optimal bound $d^2(d+1)(d+2)$ on the number of curves ending in branching points.	We shall not need this number throughout the paper. In generic situations, the branch points are simple and the preceding argument shows that exactly three curves originate from each of them, see the first two simulations in Figure \ref{fig:ex} and Lemma \ref{lem:simple_B} below. The last simulation illustrates a non-generic situation (caused by the symmetry of the initial zeros) in which branch points coalesce and four curves emerge from them. 
	As apparent from the simulation, determining the precise arcs that support $\mu_{n,t}$ is a challenging task.
\end{remark}

\begin{figure}[t]
	\centering
	\includegraphics[width=0.8\linewidth]{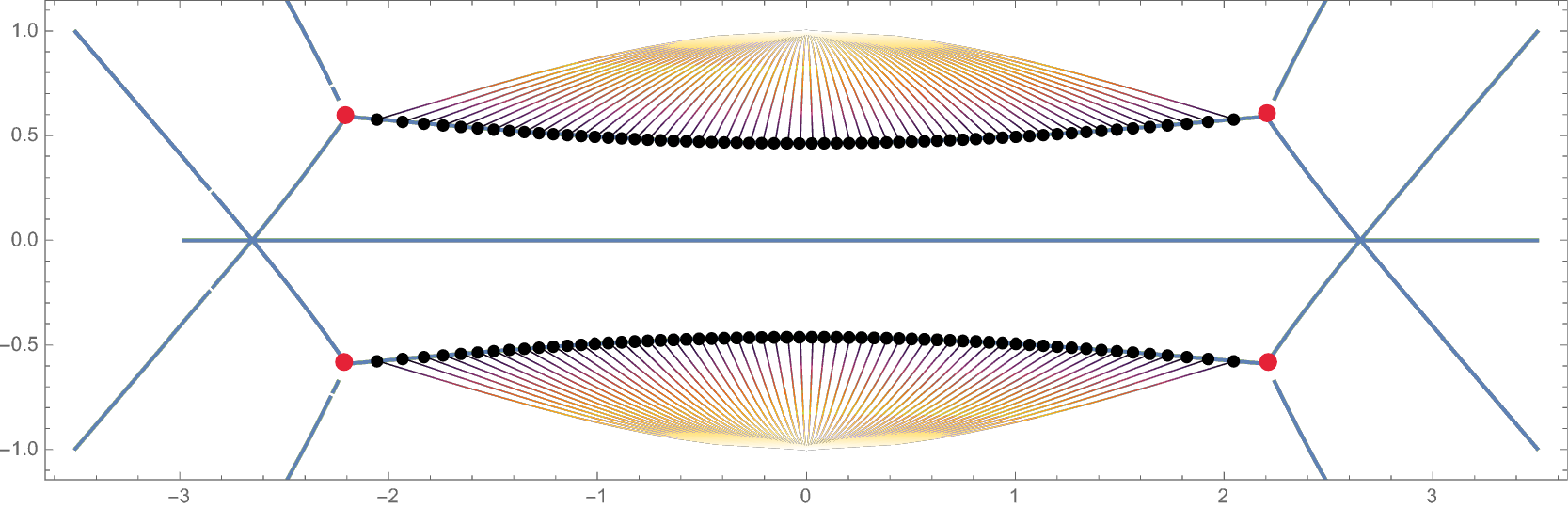}\\
	\includegraphics[width=0.8\linewidth]{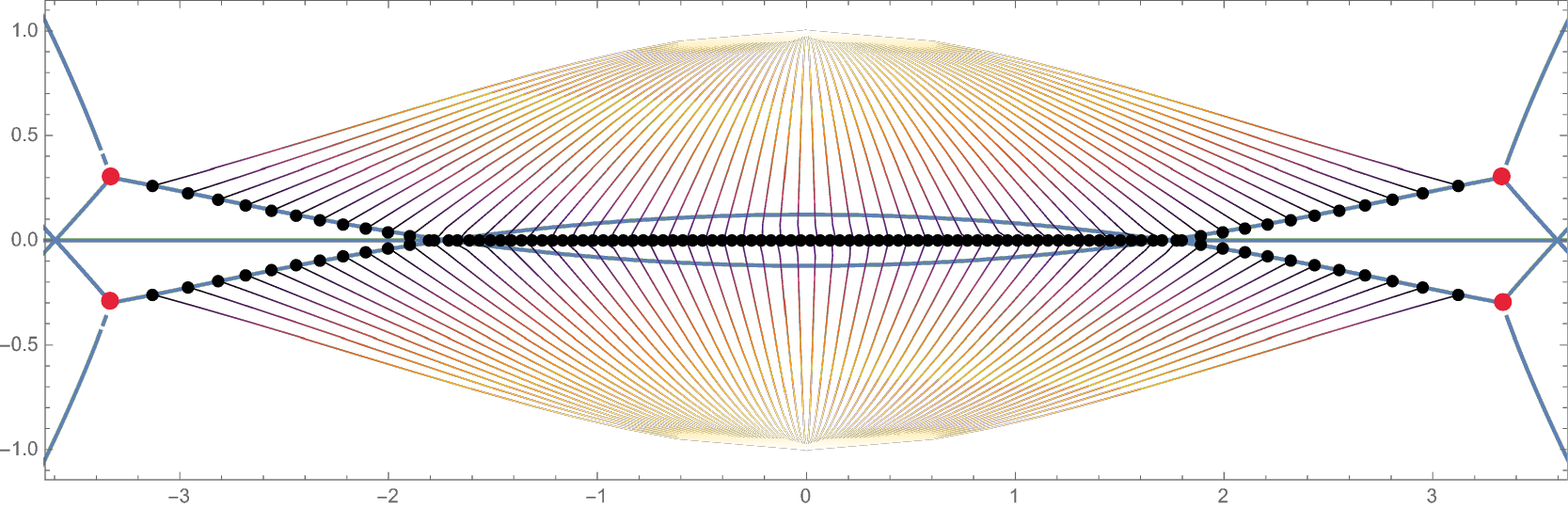}\\
	\includegraphics[width=1\linewidth]{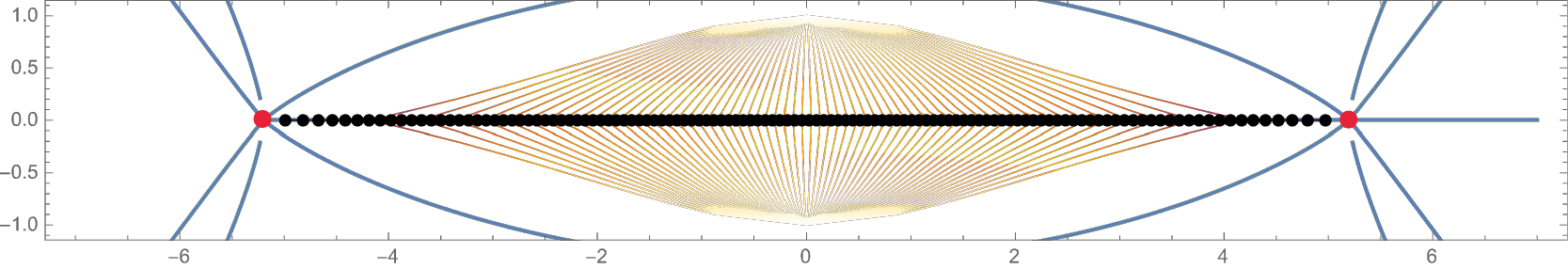}
	\caption{The curves in $\mathcal D_t^c$ in blue for $d=2$ with $\lambda_1=i,\lambda_2=-i$ and $t=2$, $t=4$ and $t=8$, including the finite $n=60$ approximation of $\mu_t$ of individual zeros (black), their trajectories (in orange) and the branching points (red).}
	\label{fig:ex}
\end{figure}

\subsection{Limiting logarithmic potential}

The crucial step in the description of the limit measure $\mu_t$ is the following convergence of the logarithmic potential. To begin with, observe that the  logarithmic potential of the measure $\mu_{n,t}$ is given by $$U_{n,t}(z)=\frac{1}{\alpha n}\log \left|e^{-\frac{t}{2\alpha n}\partial_z^2}[P^n(z)]\right|=\frac{1}{\alpha n}\log|P_t^n(z)|.$$

\begin{prop}
\label{prop:log}
Let $t>0$, $z\in \mathcal D_t$, and $u_t^*(z)$ be some specific solution of \eqref{eq:sp_eqn_2} to be chosen later, see Lemma~\ref{lem:saddle}. Then,
\begin{align*}
U_t(z):=\lim_{n\to\infty}\frac{1}{\alpha n}\log\left|\int_{i\mathbb{R}}\,e^{\alpha n\,g(z,u)}\,du\right|=G\big(z,u_t^*(z)\big).
\end{align*}
\end{prop}

Later, we will see that $U_t$ is a logarithmic potential of some probability measure $\mu_t$.

In the upcoming subsections, we will be occupied with the proof of Proposition~\ref{prop:log}. One essential step is to  study the behavior of the function $g$ in order to choose the right saddle point $u_t^*(z)$.
For fixed $t>0,\,z\in\mathcal D_t$, the asymptotics of the function $G$, introduced in \eqref{eq:function G}, are summarized as follows.

There are $d$ singularities at $u=\lambda_1,\dots,u=\lambda_d$   and $G(z,u)\to -\infty$ as $u\to \lambda_j$. Moreover,
\begin{align}\label{eq:G_asymptotics_Re}
	G(z,u)\sim \Re(u)^2/2t\to +\infty \quad \text{as } |\Re(u)|\to\infty,
\end{align}
 while $|\Im (u)|$ stays bounded and
\begin{align} \label{eq:G_asymptotics}
	G(z,u)\sim -\Im(u)^2/2t\to -\infty \quad \text{as } |\Im(u)|\to\infty,
\end{align}
while $|\Re (u)|$ stays bounded.

Here, and in the following, we denote asymptotic equivalence by $f\sim g$, that is $f(x)=g(x)(1+o(1))$ as $x\to\infty$ (or, $x\to 0$). We want to deform the integration contour in Lemma~\ref{lem:int} (originally, the imaginary axis $i\R$) in a way such that:
\begin{itemize}
\item[(i)]The chosen contour passes through a saddle point of the function $g$.
\item[(ii)]This saddle point is the unique critical point on the contour and maximizes $\Re\,g=G$.
\item[(iii)]$g$ is holomorphic in a neighborhood of the contour, i.e.~the contour does not contain the singularities $\lambda_1,\dots,\lambda_d$.
\end{itemize}
Before we choose such a path $\gamma$, we have to find a suitable saddle point $u_t^*$. To this end, it is necessary to understand the topology of the sublevel sets of the function $G$.

\subsection{Choice of Contour}
\label{subsec:contour}
Let $t>0$ and $z\in\mathcal D_t$ be fixed. We define the sublevel sets
\begin{align*}
	\Sigma(h)\coloneq  \{u\in \mathbb C : G(z,u)\le  h\},
\end{align*}
where $h\in\R$ can be interpreted as a height parameter. 

\begin{lemma}
	\label{lem:saddle}
	For each $t>0$,  $z \in \mathcal D_t$, there exists a unique saddle point $u_t^*(z)$ such that:
	\begin{enumerate}
		\item There exists a path $\gamma=(-i\infty, -i C]\cup\gamma_C\cup [+iC, +i\infty)$ from $-i\infty$ to $+i\infty$, for some $C>0$ and some simple curve $\gamma_C$ from $-iC$ to $+iC$, that contains $u_t^*(z)$ and avoids the singularities $\lambda_1,\dots,\lambda_d$.
		\item For all points $u\in\gamma\setminus \{u_t^*(z)\}$, we have $G(z,u_t^*(z))>G(z,u)$.
		\item 
		The choice $u_t^*(z)$ is the unique saddle point $u_t^j(z)$ such that for each $\eps>0$ there exists no path in $\Sigma(G(z,u_t^j(z))-\eps)$ connecting $-i\infty$ and $+i\infty$, whereas there exists such a path in $\Sigma(G(z,u_t^j(z))+\eps)$.
	\end{enumerate}
\end{lemma}

\begin{definition}
	Adopting the terminology of \cite{BHS24}, we call $u^*_t(z)$ the \emph{maximally relevant} saddle point.  A saddle point $u_t^j(z)$ is called \emph{irrelevant} if $G(z,u_t^j(z)) > G(z,u_t^*(z))$. 
\end{definition}

We prove Lemma \ref{lem:saddle} by analyzing the connected components of $\Sigma(h)$ in dependence of $h$, whose homotopy type changes exactly if $h$ passes a critical value of $G$ according to well-known results from Morse theory. 
For an illustration,  see Figure \ref{fig:G}. 

\begin{proof}[Proof of Lemma~\ref{lem:saddle}]
	We fix $z\in \mathcal D_t$ such that all heights $G(z,u_t^j(z))$, $j=1,\dots,d+1$ are different.

	\begin{claim}
	We claim that $\Sigma(h)$ consists of $d+2$ simply connected components for $h\ll 0$ and is simply connected for $h\gg 0$. 
	\end{claim}
	
	First, note that in any bounded simply connected region of $\C\setminus \{\lambda_1,...,\lambda_d\}$ there can be no isolated simply connected component as this would contradict the maximum principle of the harmonic function $u\mapsto G(z,u)$. Naturally, each logarithmic singularity at $\lambda_1,...,\lambda_d$ of $u\mapsto G(z,u)$ gives rise to one component of $\Sigma(h)$, which we shall denote by $\Sigma^1(h),\dots,\Sigma^d(h)$ each containing $\lambda_j\in\Sigma^j(h)$ by definition. Again, by the maximum principle, these are simply connected. For sufficiently small $h\ll 0$, they are disjoint, but the components will merge as $h$ grows, which will become important soon. 
	
	Two further components of $\Sigma(h)$ are located far out the imaginary axis: The asymptotic \eqref{eq:G_asymptotics} implies that for each fixed $z$ and $h\in\R$ there exist $C_+, C_- >0$ such that for all $y>C_+$ holds $G(z,iy)<G(z,iC_+)<h$ and, analogously, $G(z,iy)<G(z,-iC_-)<h$ for all  $y<-C_-$. We choose $h\ll 0$ and corresponding $C:=\max\{C_+, C_-\}$ such that the components $\Sigma^+(h)$ and $\Sigma^-(h)$, containing the intervals $(+iC, +i\infty)$ and $(-i\infty, -iC)$ respectively, are disjoint from $\Sigma^1(h),\dots,\Sigma^d(h)$. Moreover, we claim these components to be simply connected. Assume for the moment that $x\mapsto G(z,x\pm iy)$ is convex for $|y|>c_{\lambda,t}:=2\max\{\sqrt t, |\lambda_1|,\dots, |\lambda_d|\}$. Then, choosing $h$ so small that $C>c_{\lambda,t}$, we have $(+iC,i\infty)\subseteq\Sigma^+(h)$ and by the claimed convexity, each line $\R+iy$ with $|y|>C$ may intersect $\Sigma^+(h)$ only in an interval containing $iy$, hence $\Sigma^+(h)\cap (\R+i[C,\infty))$ is simply connected. On the other hand, $\Sigma^+(h)\setminus (\R+i[C,\infty))$ is also simply connected: In a bounded area $\Sigma^+(h)\cap ([-C,C]+i[-C,C])$ is simply connected by the maximum principle and the remaining area is empty by \eqref{eq:G_asymptotics_Re}, i.e.~$\Sigma(h)\cap \big(((-\infty,-C]\cup[C,\infty))+i[-C,C]\big)=\emptyset$. Thus, $\Sigma^+(h)$ is simply connected and the analogous statement holds for $\Sigma^-(h)$. It remains to check the claimed convexity. Since $\partial_u^2 g(z,u)=\frac 1 t-\frac 1 \alpha \sum_{j=1}^d \frac{\alpha_j}{(u-\lambda_j)^2}$, and $\frac{\partial^2}{\partial x^2}G(z, x+iy)=\Re( \partial_u^2 g(z,u))\vert_{u=x+iy}$, we have for $|y|>2\max\{\sqrt t, \Im\lambda_1,\dots, \Im\lambda_d\}$ that $|y-\Im \lambda_j|>|y|/2$, hence
	\begin{align}
		\label{eq:convexity}
		\frac{\partial^2}{\partial x^2}G(z,x+iy)
		=\Re\left(\frac 1 t-\frac 1 \alpha \sum_{j=1}^d \frac{\alpha_j}{(x+iy -\lambda_j)^2}\right)  \ge\frac 1 t -\frac 1 \alpha \sum_{j=1}^d \frac{\alpha_j}{|x+iy-\lambda_j|^2}  >\frac 1 t-\frac 4 {|y|^2} >0.
	\end{align}
	Therefore, we have shown that there is a sufficiently small $h\in\R$ such that $\Sigma(h)$ splits into the $d+2$ simply connected components $\Sigma^+(h),\Sigma^-(h),\Sigma^1(h),\dots,\Sigma^d(h)$, as claimed in the beginning.

Again by the maximum principle for harmonic functions, $u\mapsto G(z,u)$ does not have any local maxima or minima, such that every critical point is a saddle point. 
Moreover, since $z\in\mathcal D_t$, all saddle points are non-degenerate. Standard results from Morse theory imply that, as $h$ increases, the number of connected components of $\Sigma (h)$ decreases by one for each passed critical value of a non-degenerate saddle point. We refer to \cite{milnor} for further background on the theory. 

To be more precise, a classical statement holds for smooth functions $f:M\to\R$ on a compact manifold $M$ with a non-degenerate critical point $p\in M$ of certain index $k$ and states that sublevel sets $f^{-1}(-\infty, f(p)+\eps]$ have the same homotopy type as $f^{-1}(-\infty, f(p)-\eps]$ with a $k$-cell attached, see \cite[Theorem 3.2]{milnor} for $M$ without boundary. In our case, we choose $f(u)=G(z,u)$, and $p=u_t^j(z)$ having index $k=1$, since these are non-degenerate saddle points. Hence, attaching $1$-cells translates to merging connected components of $\Sigma(h)$ at heights $h=G(z,u_t^j(z))$ and positions $u_t^j(z)$. In order to make the domain of $f$ compact, we need to choose a compact manifold $M$ with boundary $\partial M$ like the above-mentioned box after smoothing the corners and removing the singularities. To be very precise one may choose 
\begin{align}\label{eq:M}
M=\big(([-2R,2R]+i[-R,R])\cup \overline{B_R(2R)}\cup \overline {B_R(-2R)}\big)\setminus \bigcup_{j=1}^d B_{1/R} (\lambda_j)
\end{align}
for some large $R>C>0$. Naturally, the homotopy type of $f^{-1}(-\infty, h]=\Sigma(h)\cap M$ may also change for heights $h$ that are critical values at the boundary, as is discussed in \cite[\S 1]{laudenbach} and \cite[p147]{braess}. However, our interest lies in heights $h$ in bounded intervals $I$ containing the critical values $G(z,u_t^j(z))$ and we may choose the box-size $R>0$ large enough such that no critical values of $f|_{\partial M}$ lie within $I$. Note that there is at most one critical point within $[-R,R]\pm iR$ by the above convexity whose value can be made sufficiently small. As $\Sigma(h)$ is a filtration (i.e.~$\Sigma(h_1)\subseteq\Sigma(h_2)$ for all $h_1\le h_2$), the whole connectivity structure of the sublevel sets only changes at critical values (which is more than bare homotopy type), which can be described via the so-called Reeb graph, see \cite{Reeb} for more information. In particular, this implies that once connected components remain connected as the height $h$ increases and they even remain simply connected since loops cannot appear due to the maximum principle.

In summary, we have: For all $h_1<h_2$ such that $[h_1,h_2]$ contains no critical value of $u\mapsto G(z,u)$, the inclusions $\Sigma(h_1)\hookrightarrow \Sigma(h_2)$ induce homotopy equivalences on each connected component $\Sigma^+(h),\Sigma^-(h),\Sigma^1(h),\dots,\Sigma^d(h)$. Moreover, if $h= G(z,u_t^j(z))$ is a critical value, then for sufficiently small $\varepsilon>0$ exactly two connected components of $\Sigma(h-\varepsilon)$ merge into a single connected component of $\Sigma(h+\varepsilon)$, while all other components are canonically identified via the inclusion maps. In particular, the total number of connected components decreases by exactly one.
  Since we have $d+1$ such saddle points, only one connected component of $\Sigma (h)$ remains for sufficiently large $h$. In other words, for $h\gg 0$, $\Sigma (h)$ is simply connected.

By the preceding claim, there must exist a minimal value $h_0$ such that the two components $\Sigma^+(h_0)$ and $\Sigma^-(h_0)$ meet, occurring precisely at a saddle point $u_t^j(z)$. We let $u_t^*(z)$ be this first saddle point that connects the components $\Sigma^+(h_0)$ and $\Sigma^-(h_0)$ at level $h_0\in\R$. 
In particular, this implies existence of a path $\gamma=(-i\infty, -iC]\cup\gamma_C\cup[+iC,+i\infty)$ such that $G(z,u_t^*(z))>G(z,u)$ for all other points on the contour. Regarding uniqueness, we will argue as follows.
\begin{claim}
There exists no other saddle point such that the newly formed component $\Sigma^+(h)=\Sigma^-(h)$ containing $\pm i\infty$ meets itself at a higher level $h>h_0$. In particular, for any saddle point $v^*$ with $G(z,v^*)>G(z,u^*_t(z))$, there exists no curve $\gamma$ from $-i\infty$ to $+i\infty$ and containing $v^*$ such that $G(z,v^*)>G(z,u)$ for all $u$ on $\gamma$.
\end{claim}

By definition, $u_t^*(z)$ is the first saddle point that connects the components $\Sigma^+(h)$ and $\Sigma^-(h)$ at level $h=h_0$. Assume now, there exists another saddle point $v^*$ such that $\Sigma (h_1)$ meets itself at $h_1>h_0$. (Recall that we assumed all saddle points to be of different heights.) Then we can form a closed curve $\tilde\gamma$ that lies in $\Sigma (h_1)$, encircling a bounded region $\tilde C$ disjoint from $\Sigma (h_1)$. But $G$ is, by definition, bounded in $\tilde C$ from above hence must contain a local maximum in this area. This yields a contradiction to the maximum principle, which proves our claim and finishes the proof of Lemma~\ref{lem:saddle}.
\end{proof}

\begin{figure}[t]
	\centering
	\includegraphics[width=0.45\linewidth]{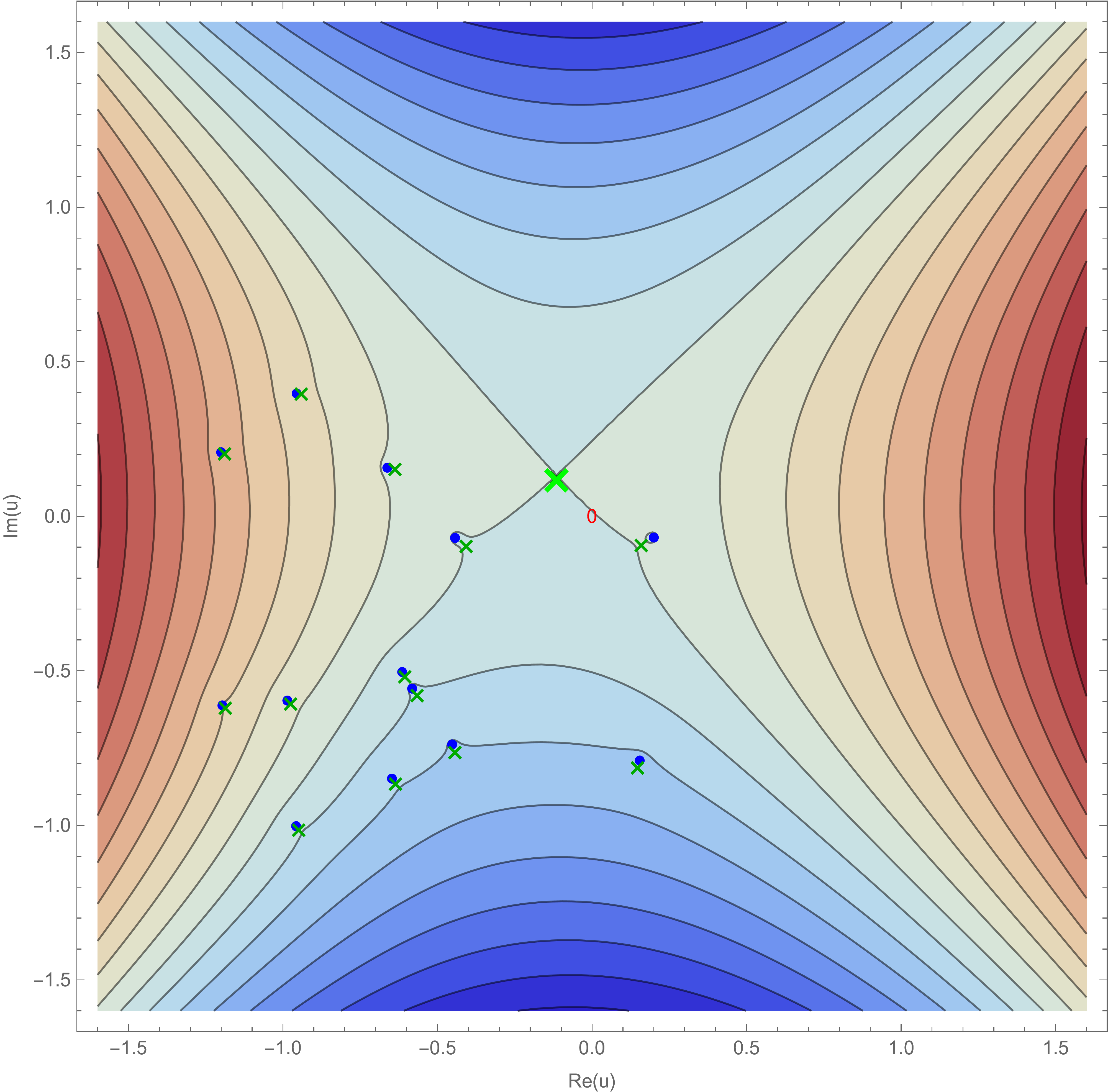}\quad
	\includegraphics[width=0.45\linewidth]{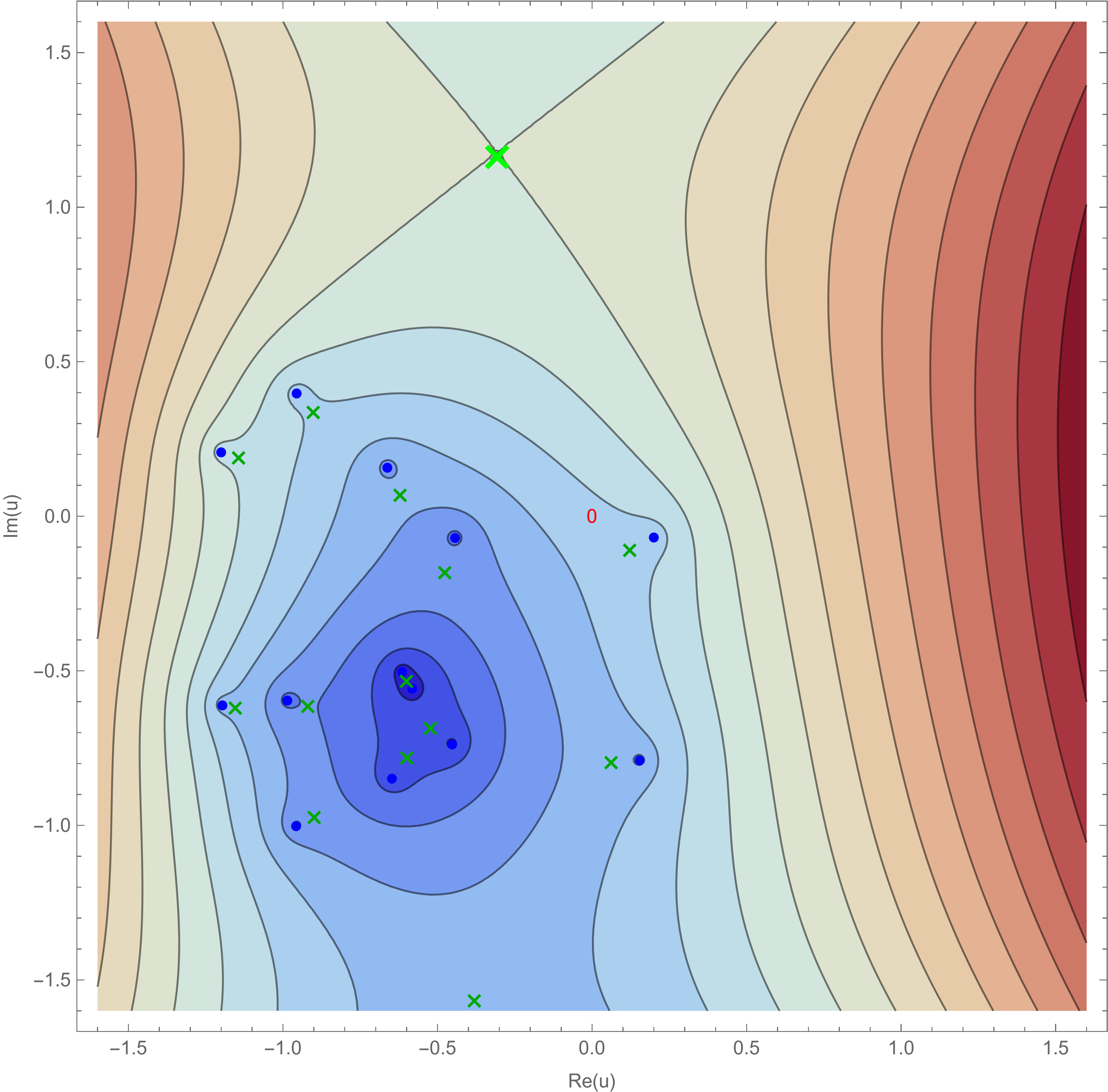}
\caption 
{Heatmaps of the (sub)level sets of the function $u\mapsto G(z,u)$, where low values are depicted in cold colors and high values in warm colors. We fix $z=0$ (shown in red) and consider $d=13$ initial roots $\lambda_j$, indicated by blue dots at times $t=0.2$ (left) and $t=2$ (right). The $d+1$ saddle points $u_t^j(0)$ are depicted as green crosses with the maximally relevant saddle point $u_t^*(0)$ highlighted. Notably, the contour $\gamma$ can be placed within the sublevel set of height $h=G(0,u_t^*(0))$.\\
For smaller $t=0.2$, the behavior of $G$ is dominated by its component $\Re(u^2)$, hence the maximally relevant \emph{saddle} point $u_t^*(0)$ is near $z=0$, right at the back of the Pegasus. This illustrates our findings of Lemma \ref{lem:small_t} (in particular \eqref{eq:sp_dist}). Further, note that several (irrelevant) saddle points can be seen very close to the initial $\lambda_j$, as we will verify in Lemma \ref{lem:u_t_small_t}.\\
For larger $t=2$, the function $G$ flattens out. No saddle point is irrelevant and the maximally relevant saddle $u_t^*(0)$ has large imaginary part, while there is another saddle point having particularly small imaginary part. We will explain and prove this observation in Lemma \ref{lem:large_t} and Corollary \ref{cor:large_t_rescaled} below.
}
\label{fig:G}
\end{figure}

\subsection{Application of the saddle point method}
\label{sec:sp_method}
Let $t>0,\,z\in\mathcal D_t$ be fixed.
We have shown in Section \ref{subsec:contour} that there exist a suitable saddle point $u_t^*$ and contour $\gamma$ such that we can deform our original path $i\R$ and integrate over $\gamma=(-i\infty,-iC]\,\cup\, \gamma_C\, \cup\, [iC,i\infty)$ instead without changing the value of the integral. First, we consider the integral over $\gamma_C$, where we can apply the saddle point method.
Letting $g(u)=g(z,u)$, as introduced in \eqref{eq:function g}, in Theorem \ref{thm:sp_method} yields
\begin{align*}
	\lim_{n\to\infty}\frac{1}{\alpha n}\log\left|\int_{\gamma_C}\,e^{\alpha n\,g(z,u)}\,du\right|= \Re\,g(z,u_t^*(z))= G(z,u_t^*(z)), \qquad z\in \mathcal D_t.
\end{align*}
For the remaining parts of the integral, note that by definition of the (large) constant $C>0$ in the proof of Lemma \ref{lem:saddle}, it holds $G(z,u_t^*(z))-G(z,u)\ge G(z,u_t^*(z))-G(z,iC)=\epsilon>0$ for all $u\in [iC,i\infty)$.
Due to \eqref{eq:G_asymptotics}, we can bound the term 
\begin{align*}
	\left|\int_{iC}^{i\infty}e^{\alpha ng(z,u)}du\right|&\leq \int_{iC}^{in}e^{\alpha n G(z,u)}du+ \int_{in}^{i\infty}e^{\alpha n G(z,u)}du\\
	&\leq n e^{\alpha n (G(z,u_t^*(z))-\epsilon)}+\int_{in}^{i\infty}e^{-\frac{\alpha n}{4t}\Im(u)^2}du
\end{align*}
as $n\to\infty$. Thus, dividing by the dominant part $\int_{\gamma_C}\,e^{\alpha n\,g(z,u)}\,du$, the contribution of the integration over $[iC,i\infty)$ becomes asymptotically negligible. The same holds for $\int_{-i\infty}^{-iC}e^{\alpha n\,g(z,u)}du$. 
Hence, we have managed to prove Proposition \ref{prop:log} via the saddle point method and can continue with proving the existence of a limit of the empirical root measure $\mu_{n,t}$.

\subsection{Limit of the empirical root measure}

To prove weak convergence of the sequence $(\mu_{n,t})_{n\in\N}$, we need to verify its tightness, which requires bounding the support of the measures $\mu_{n,t}$. In the repeated differentiation setting of~\cite{BHS24}, the tightness was a simple consequence of  the Gauss--Lucas theorem. In our case, a different argument is required.

\begin{lemma}
	\label{lem:P_t^n roots}
For any fixed $t>0$, the zeros of $P_t^n$ remain within a distance of $2\sqrt t+o(1)$ from the initial zeros $\lambda_1,\dots,\lambda_d$, more precisely, for all $n\in \N$,
\begin{align} \supp(\mu_{n,t})\subset \cup_{j=1}^d B_{2\sqrt t \sqrt{1+\frac{1}{2n\alpha}}} (\lambda_j).\label{eq:support_mu_nt}
\end{align}
Consequently, the family $(\mu_{n,t})_{n\in\N}$ forms a tight sequence of probability distributions.
Moreover, if $t>0$ is sufficiently small so that \eqref{eq:support_mu_nt} is a disjoint union, then $$\mu_{n,t}\left(B_{2\sqrt{t}\sqrt{1+\frac{1}{2n\alpha}}}(\lambda_j)\right)=\frac{\alpha_j}{\alpha}, \qquad j=1,\ldots, d.$$
\end{lemma}

\begin{proof}
We rely on the results of Section 5.3 in \cite{rahman_schmeisser}. Let
	\begin{align*}
		 \left(\frac{t}{n\alpha}\right)^{n\alpha/2}\mathrm{He}_{n\alpha}\left(\sqrt{\frac{n\alpha}{t}}z\right)=\sum_{m=0}^{n\alpha}a_m\frac{z^m}{m!}\quad\text{and}\quad P^n(z)=\sum_{m=0}^{n\alpha} b_m\frac{z^m}{m!}
	\end{align*}
	such that $\left(\frac{t}{n\alpha}\right)^{n\alpha/2}\mathrm{He}_{n\alpha}\left(\sqrt{\frac{n\alpha}{t}}z\right)=f(z)$ and $P^n(z)=g(z)$ in the setting of \cite[Equation (5.3.2)]{rahman_schmeisser}. Consistent with their notation, we then have
	\begin{align*}
		h(z)=\sum_{m=0}^{n\alpha}b_m f^{(n\alpha-m)}(z)
		=(n\alpha)!\sum_{m=0}^{n\alpha}\frac{b_m}{m!}\left(\frac{t}{n\alpha}\right)^{m/2}\mathrm{He}_{m}\left(\sqrt{\frac{n\alpha}{t}}z\right)
		=(n\alpha)!P_t^n(z),
	\end{align*}
	where we used that $\mathrm{He}_{\alpha n}^{(\ell)}(x)= \frac{(\alpha n)!}{(\alpha n - \ell)!} \mathrm{He}_{\alpha n-\ell}(x)$ and
	\begin{align}\label{eq:Poly_finite_free}
		P_t^n(z)=e^{-\frac s 2 \partial_z^2}P^n(z)=\sum_{m=0}^{\alpha n} \frac{b_m}{m!} e^{-\frac s 2 \partial_z^2}z^m=\sum_{m=0}^{\alpha n} \frac{b_m}{m!} \mathrm{He}_m\left(\frac{z}{\sqrt s}\right)s^{m/2},
	\end{align}
	with $s=\frac{t}{n\alpha}$ in our case.
	By \cite[\S 6.31, p.129]{szego75}, the zeros of the (physicist) Hermite polynomials $\mathrm H_n$ are contained in $[-\sqrt {2n+1},\sqrt {2n+1}]$ and  consequently, the zeros of $\mathrm{He}_{\alpha n}(z/\sqrt s)= 2^{-\alpha  n/2} \mathrm H_{\alpha n}(z/\sqrt{2 s})$ are bounded by $2\sqrt t\sqrt{1+1/(2n\alpha)}$ in absolute value.
In particular, by \cite[Theorem 5.3.1]{rahman_schmeisser}, every zero of the heat-evolved polynomials $P_t^n$ is of the form $\lambda_j +\xi$ for a suitable $j\in \{1,\ldots, d\}$ and some $\xi\in B_{2\sqrt{t}\sqrt{1+\frac{1}{2n\alpha}}}(0)$. This proves~\eqref{eq:support_mu_nt}. Now, let $t>0$ be sufficiently small such that the union in \eqref{eq:support_mu_nt} is disjoint. Then, the  supplement of \cite[Theorem 5.3.1]{rahman_schmeisser} yields that the number of zeros of $P_t^n$ in each disk $B_{2\sqrt t \sqrt{1+\frac{1}{2n\alpha}}} (\lambda_j)$ is exactly $\alpha_j n$ and it follows that the measure $\mu_{n,t}$ assigns the value $\alpha_j /\alpha$ to this disk.
	\end{proof}
	
\begin{remark} 
Let us provide some additional comments.
\begin{enumerate}
\item In the language of finite free probability (see \cite{MSS15} for more information) we rephrased the heat-evolved polynomial in \eqref{eq:Poly_finite_free} as the finite free convolution
$$ P_t^n(z)= P^n(z)\boxplus_{\alpha n} s^{\alpha n/2}\mathrm{He}_{\alpha n} (z/\sqrt s).$$
Usually, this bilinear operation for polynomials is only studied  for \emph{real-rooted} polynomials $P$, for which the first statement of Lemma \ref{lem:P_t^n roots} can be found in \cite[Theorem 1.3]{MSS15}. Interestingly, real-rootedness is not required in Lemma~\ref{lem:P_t^n roots}.
\item  In Corollary \ref{cor:small_t} below, we will see that the bound on the support of $\mu_{n,t}$ is not optimal. We expect that the measure $\mu_t$ is supported inside the union of balls  of radii  $2\sqrt {t\alpha_j/\alpha}$ centered at $\lambda_j$.
\item Lemma \ref{lem:P_t^n roots} shows tightness of the sequence $(\mu_{n,t})_{n\in \N}$ with quantitative bounds on the support. One can show that mere qualitative tightness also follows from the following upper bound on the logarithmic potential combined with a Markov inequality. We omit the details of this implication.
\end{enumerate}
\end{remark}
	
\begin{lemma}\label{lem:upperbound}
For any $r,t>0$, there is a constant $c>0$ such that
\begin{align*}
U_{n,t}(z)=\frac {1}  {\alpha n} \log \left\lvert \eee^{-\frac{t}{2\alpha n}\partial_z^2}P^n(z)\right\rvert <c
\end{align*}
uniformly in $n\in\N$ and $|z|\le r\sqrt t$. The bound $c$ depends on $t$ as $c=\frac {1} {2} \max(0,\log t)+ \tilde c$ for some $\tilde c>0$, which is independent of $t$ and only depends on the data $(\lambda_j)_j, (\alpha_j)_j$ of the polynomial $P$.
\end{lemma}
In the following, we shall apply this lemma for fixed $t>0$, however, we will also need the large-$t$ case later.

\begin{proof}
For $|z|\le r\sqrt t$ and any $k\le \alpha n$, we use the elementary bound of Hermite polynomials (see for instance \cite[Equation (1.2)]{HermiteBound})
\begin{align*}
\left| \exp\big(-\tfrac{t}{2\alpha n} \partial_z^2\big) z^k\right| =  \left|\tfrac t{\alpha n} \right|^{k/2} \left|\mathrm{He}_k\big(\sqrt{\tfrac {\alpha n} t }z\big)\right|\le \left|\tfrac {t} {\alpha n} \right|^{k/2}\sqrt{k!}\exp\big( \sqrt{\tfrac {\alpha nk} {t}} |z|\big)\le e^{c_0 n} \max(t,1)^{\alpha n/2}
\end{align*}
for some $c_0>0$.
Expanding $P^n(z)$, denoting $k=\sum_{j=1}^d i_j$ and using Stirling's approximation for the binomial coefficient, we obtain the claimed bound
\begin{align*}
\left| e^{-\frac t {2\alpha n}  \partial_z^2}  P^n(z)\right|
& =\left| \sum_{i_1=0}^{n\alpha_1}\dots \sum_{i_d=0}^{n\alpha_d} \exp\big(-\tfrac{t}{2\alpha n} \partial_z^2\big) z^k\prod_{j=1}^d \binom{n\alpha_j}{i_j}(-\lambda_j)^{n\alpha_j-i_j}\right|\\
 &\le \sum_{i_1=0}^{n\alpha_1}\dots \sum_{i_d=0}^{n\alpha_d} e^{c_0 n}\max(t,1)^{\alpha n/2} e^{c_1 n}\le e^{c n}
\end{align*}
for some $c_0,c_1>0$ depending on $(\lambda_j)_j, (\alpha_j)_j$  as claimed.
\end{proof}
We are now ready for the Proof of Theorem \ref{thm:measure}.
\begin{proof}[Proof of Theorem \ref{thm:measure}]
First, let $t>0$.
By tightness as seen in Lemma \ref{lem:P_t^n roots} and Prokhorov's Theorem, it suffices to show vague convergence of $(\mu_{n,t})_{n\in \N}$.  Let $\varphi\in\mathcal C_c^\infty$ be a test function. We will argue that each subsequence of $\int U_{n,t}(z) \varphi(z) dz $ contains a further subsequence converging to $\int U_t(z) \varphi(z) dz$, which implies convergence of the whole sequence $\int U_{n,t}(z) \varphi(z) dz $. Let $U_{n_k,t}$ be an arbitrary subsequence.
By \cite[Theorem 4.1.9]{hormander1}, any locally upper-bounded sequence of subharmonic functions is either converging to $-\infty$ a.s.\  or contains a further subsequence converging in $L^1_{\mathrm{loc}}$ (thus in the distributional sense). The former is excluded by pointwise convergence (Proposition~\ref{prop:log}), hence $U_{n_k,t}$ contains a further subsequence $U_{\tilde n_k,t}$ converging in $L^1_{\mathrm{loc}}$ and
\begin{align}\label{eq:L1loc}
\int U_{\tilde n_k,t}(z) \varphi(z) dz \to \int U_t(z) \varphi(z) dz .
\end{align}
Therefore, since $\mathcal D_t$ has full Lebesgue measure, we conclude that the whole sequence  $\int U_{n,t}(z) \varphi(z) dz$ converges to  $\int U_t(z) \varphi(z) dz$. Replacing $\varphi$ by $\Delta \varphi$ shows that $\int U_{n,t}(z) \Delta\varphi(z) dz$ converges to  $\int U_t(z) \Delta\varphi(z) dz$. The definition  of the Laplacian $\Delta$ on the space of distributions  implies
\begin{align*}
\int_{\C} \varphi d\mu_{n,t}
=
\int_{\C} \Delta \varphi(z) \frac {U_{n,t}(z)} {2\pi }  dz \ton
\int_{\C}\Delta\varphi (z) \frac {U_{t}(z)} {2\pi} dz
=
\int_{\C} \varphi d\mu_{t}
\end{align*}
as $n\to\infty$. So, $\mu_{n,t}$ converges vaguely (hence by tightness also weakly) to $\mu_t \coloneq \frac 1 {2\pi} \Delta U_t$, which has compact support by Lemma \ref{lem:P_t^n roots}.

Finally, let $t=|t|e^{i\theta}\in\C$ and define $\tilde P^n(z)=P^n(e^{i\theta/2}z)$, the rotated polynomial. It follows from the chain rule that $t\partial_z^2=|t|\partial_{\tilde z}^2$ for $\tilde z=e^{-i\theta/2}z$ and then from \eqref{eq:def_heat_flow} that $P^n_t(z)=\big(e^{-\frac{|t|}{2\alpha n}\partial_z^2}\tilde P^n\big)(e^{-i\theta/2}z)$. Denoting the maximally relevant solution to the saddle point equation \eqref{eq:sp_eqn_2} for $\tilde P^n$ by $\tilde u_{|t|}^*$, we obtain
\begin{align*}
\tilde U_{|t|}(e^{-i\theta/2}z)&=\frac 1 {\alpha n}\log |\tilde P^n (\tilde u_{|t|}^*(e^{-i\theta/2}z))|+\Re\left(\frac{(e^{-i\theta/2}z-\tilde u_{|t|}^*(e^{-i\theta/2}z))^2}{2|t|}\right)\\
&=\frac{1}{\alpha n}\log |P^n(u_t^*(z))|+\Re\left(\frac{(z-u_t^*(z))^2}{2t}\right)\\
&=U_t(z),
\end{align*}
from Proposition \ref{prop:log}, where we now finally defined $u_t^*(z):=e^{i\theta/2}\tilde u_{|t|}^*(e^{-i\theta/2} z)$ for $t\in\C$. This saddle point $u_t^*(z)$ already illustrates how rotation of $t$, $z$, and $\lambda_j$ interact and, consistently, it solves
\begin{align}\label{eq:sp_eqn_complex}
\frac{t}{\alpha}\sum_{j=1}^d\frac{\alpha_j}{u-\lambda_j}=z-u, \qquad t,z\in\C.
\end{align}
Moreover, by Lemma \ref{lem:analyt} it depends analytically on $t\in\C$.
 The claimed push-forward follows from applying the Laplacian
$$ \mu_t=\frac 1 {2\pi}\Delta U_t=\frac{1}{2\pi}\Delta \tilde U_{|t|}(e^{-i\theta/2}\cdot)=\tilde\mu_{|t|}(e^{-i\theta/2}\cdot).$$
\end{proof}

\section{Description of the limiting distribution}
In this section, we will present the proofs of all results describing the properties of $\mu_t$. In the following, we will need to identify regions where branches of maximally relevant saddle points switch. We begin with the following auxiliary lemma.

\begin{lemma}\label{lem:sp_nbh}
	Suppose that $D\subset\C\setminus B$ is a simply connected domain  -- then the saddle points $u_t^j(z)$, $j=1,\dots,d+1$, are well-defined analytic functions on $D$ -- and let $z_0\in D\cap\mathcal D_t$, where $u_t^k(z_0)=u_t^*(z_0)$ is the unique maximally relevant saddle point corresponding to $z_0$. Further, assume that $G(z,u_t^k(z))\neq G(z,u_t^j(z))$ for all $z\in D$ and $j\neq k$.

Then, for every $z\in D$, $u_t^k(z)=u_t^*(z)$ is the unique maximally relevant saddle point.
\end{lemma}

Note that we could also assume the stronger condition $G(z,u_t^i(z))\neq G(z,u_t^j(z))$ for all $i\neq j$ as was done in \cite[Lemma 4.7]{BHS24}. However, the weaker assumption of Lemma \ref{lem:sp_nbh} allows for domains $D$ such that  $\Gamma_{i,j}^D=\{z\in D: G(z,u_t^i(z))=G(z,u_t^j(z))\}\neq \emptyset$  corresponding to two saddle points $u_t^i,u_t^j$ of lower height than $u_t^k$ (or, higher irrelevant ones).

\begin{proof}
	Note that the saddle point $u_t^k(z)$ is continuous in $z\in D$, and $G(z,u)$ is continuous in $z\in D$ and $u\in \C\setminus \{\lambda_1,\dots,\lambda_d\}$.
	Take two points $z_0,z_1\in D$, where $u_t^k(z_0)=u_t^*(z_0)$. Due to simple connectedness, there exists a path $\beta$ from $z_0$ to $z_1$ and a neighborhood of $\beta$ on which  $u_t^k$ is analytic. By continuity, and our assumption that $G(z,u_t^k(z))\neq G(z,u_t^j(z))$ for all $z\in D$ and $j\neq k$, there exists $\varepsilon>0$ such that $|G(z,u_t^k(z))-G(z,u_t^j(z))|>\varepsilon$ for all $z\in \beta$. Let $z\in\beta$ such that $u_t^k(z)$ is maximally relevant. By Lemma \ref{lem:saddle} (3), there is no path $\gamma$ from $-i\infty$ to $+i\infty$ through $\Sigma( h_0 -\varepsilon/2)$ for $h_0=G(z,u_t^k(z))$, while there exists such a path through $\Sigma(h_0 +\varepsilon/2)$. Moreover, as explained in the proof of Lemma \ref{lem:saddle}, no other connected components of sublevel sets $\Sigma(h)$ merge within the range of heights $h\in [h_0-\varepsilon/2 , h_0 +\varepsilon/2]$. We claim that by continuity in $z$, there is a neighborhood $B_\delta (z)$ such that the same holds for all $w\in B_\delta (z)$ and maximally relevance is stable under perturbation of $z$.
	 
	Let $M\subseteq \C$ be compact as defined in \eqref{eq:M}, such that the sublevel-set-mapping $w \mapsto M\cap \Sigma_w(h)$, for $\Sigma_w(h)\coloneq  \{u\in \mathbb C : G(w,u) \le h\}$, has compact values (actually, it is a Hausdorff continuous set-valued mapping, see \cite[Theorem 5.7]{RockaWets} and \cite[Corollary 5.21]{RockaWets}). By uniform continuity of $(w,u)\mapsto G(w,u)$ on the compact set $\beta \times M$,  there exists some $\delta>0$ such that for all $w\in B_\delta(\tilde w)$, for some $\tilde w \in \beta$, we have $\sup_{u\in M} |G(w,u)-G(\tilde w,u)|<\varepsilon/4$. Hence, $$\Sigma_{\tilde w}(h-\varepsilon/4)\cap M\subseteq \Sigma_w(h)\cap M\subseteq \Sigma_{\tilde w}(h+\varepsilon/4)\cap M$$
	 for all $h\in\R$.
	In particular, there is no path $\gamma$ through sublevel sets $\Sigma_w(h_0-\varepsilon/4)$, while there is one through $\Sigma_w(h_0+\varepsilon/4)$. Thus, if we cover the path $\beta$ with $\delta$-balls of this kind, Lemma \ref{lem:saddle} (3) implies that if $u_t^k(z)$ is maximally relevant for some $z\in \beta$, then $u_t^k(w)$ stays maximally relevant for all $w\in \beta$. Hence, being maximally relevant is preserved along $z\in \beta$ as claimed.
\end{proof}

\subsection{Proof of Theorem \ref{thm:description}}
\begin{proof}[Proof of Theorem \ref{thm:description}]

\emph{(2)} The formula $U_t(z)=G(z,u_t^*(z))$, $z\in \mathcal D_t$, was proven in Section \ref{sec:sp_method} via the saddle point method.

\emph{(1)} Let $z\in \mathcal{D}_t$. Recall that, in Lemma \ref{lem:analyt}, we used the implicit function theorem to show  that $u_t^j(z)$, $j=1,\dots,d+1$, are analytic in a small disk around $z$. By Lemma \ref{lem:sp_nbh}, the branch of $u_t^*(z)$ does not switch locally around $z$, i.e.\ $u_t^*(z) = u_t^k(z)$ for some fixed $k$. Therefore, we can take the derivative of $u_t^*(z)$ in a small disk around $z$. By definition of the Stieltjes transform, we have $m_t(z)=2\partial_z U_t(z)$.
The complex chain rule for $\partial_z$ applied to the function $g$ yields
\begin{align*}
\partial_z (g(z,u^*_t(z)))=(\partial_u g)(z, u^*_t(z)) \partial_z u^*_t(z)+(\partial_z g)(z,u^*_t(z))=(\partial_z g)(z,u^*_t(z))
\end{align*}
since $u_t^*(z)$ is a saddle point of $g$. Recalling the definition of $g$ given in \eqref{eq:function g}, we obtain
\begin{align*}
m_t(z)=2\partial_z U_t(z) =\partial_z \Bigl(g(z,u^*_t(z)) + \overline{g(z,u^*_t(z))}\Bigr) =\partial_z \Bigl(g(z,u^*_t(z))\Bigr)=(\partial_z g)(z, u^*_t(z))=\frac{z-u^*_t(z)}{t}.
\end{align*}
Since $u=u_t^*(z)$ solves \eqref{eq:sp_eqn_2}, i.e.\
$\frac{t}{\alpha}\sum_{j=1}^d\frac{\alpha_j}{u-\lambda_j}=z-u$, we have
\begin{align*}
u_t^*(z)=z-t m_t(z)\quad\text{and}\quad t m_t(z)=\frac{t}{\alpha}\sum_{j=1}^d\frac{\alpha_j}{u_t^*(z)-\lambda_j},
\end{align*}
and, so,
\begin{align*}
 \alpha m_t(z)=(\partial_z\log P)(z-tm_t(z))=\sum_{j=1}^d\,\frac{\alpha_j}{z-tm_t(z)-\lambda_j}.
\end{align*}

\emph{(3)} Recall that the logarithmic potential of $\mu_t$ is given by $U_t(z)=\Re g(z,u_t^*(z))$, $z\in \mathcal D_t$, and that $g(z,u_t^*(z))$ is locally analytic on $\mathcal D_t$. It therefore follows that $U_t(z)$ is harmonic  on $\mathcal D_t$ and so, $\mu_t=\frac  1 {2\pi}\Delta U_t=0$ on $\mathcal{D}_t$. 
This implies that $\mu_t$ is supported on $\mathcal{D}_t^c$ and compactly supported by Lemma \ref{lem:P_t^n roots}. Therefore, by Lemma \ref{lem:domain}, $\supp(\mu_t)$ is contained in the union of finitely many smooth curves as claimed.
\end{proof}

With the following corollary, we can show that the measure $\mu_t$ is supported only on those curves along which the branch of the maximally relevant saddle point switches. Further, we compute the density of $\mu_t$ w.r.t.\ the length measure on these curves.

\begin{corollary}\label{cor:density}
	Let $D\subset \C\setminus B$ be a simply connected domain such that $\Gamma_{i,j}^D:=\{z\in D:G(z,u_t^i(z))=G(z,u_t^j(z))\}\neq \emptyset$ for exactly one pair $1\leq i <j \leq d+1$  and $\Gamma_{i,j}^D$  is a real analytic curve. For simplicity, let us assume that $\Gamma_{i,j}^D$ separates $D=D_i\cup D_j\cup \Gamma_{i,j}^D$ into two simply connected open domains $D_i,D_j$. Then, $\mu_t(\Gamma_{i,j}^D)>0$ if and only if the branch of the maximally relevant saddle point switches at $\Gamma_{i,j}^D$, which means that  $u_t^*(z) = u_t^i(z)$ for $z\in D_i$ and $u_t^*(z) = u_t^j(z)$ for $z\in D_j$. In this case, the density $\varrho:\Gamma_{i,j}^D\to\R$ of $\mu_t$ (restricted to $\Gamma_{i,j}^D$) with respect to the length measure on $\Gamma_{i,j}^D$ is given by
	$$\varrho(z)=\frac {|u_t^i(z)-u_t^j(z)|} {2\pi t}.$$ 
\end{corollary}

For every $z\in\Gamma_{i,j}^D$, we have $G(z,u_t^i(z))=G(z,u_t^j(z))$ by definition. It follows that we can extend $U_t(z)$ continuously to $\Gamma_{i,j}^D$ by defining $U_t(z) = G(z,u_t^i(z))=G(z,u_t^j(z))$ for $z\in \Gamma_{i,j}^D$.

\begin{proof}
If the branch of $u_t^*(z)$ does not switch at $\Gamma_{i,j}^D$, then $u_t^*(z)=u_t^k(z)$, $z\in D$, for some $1\leq k\le d+1$. Hence, $u_t^*$ is holomorphic in $D$ and $G(z,u_t^*(z))$ is harmonic on $D$, i.e. $\mu_t(D)=0$.

On the other hand, suppose that the maximally relevant saddle point switches at $\Gamma_{i,j}^D$. Note that $\pi \mu_t=\partial_{\bar z}m_t\ge 0$ in the sense of distributions (we know $\mu_t$ to be non-negative).
By the Sokhotski-Plemelj theorem as stated in \cite[Lemma 2]{borcea}, the density $\varrho:\Gamma_{i,j}^D\to\R$ of $\mu_t$ (restricted to $\Gamma_{i,j}^D$) with respect to the length measure on $\Gamma_{i,j}^D$ is given by
$$
\varrho (z) = \frac 1 {2\pi} \Bigl|\lim_{w\to z, w\in D_i } m_t(w) - \lim_{w\to z, w\in D_j} m_t(w)\Bigr|, \qquad z\in \Gamma_{i,j}^D.
$$
	By Theorem \ref{thm:description} (1), we have $m_t(w)=\frac 1 t (w-u_t^*(w))$, $w\in D_i \cup D_j$, hence $m_t(w) = \frac 1 t (w-u_t^i(w))$ for $w\in D_i$ and $m_t(w) = \frac 1 t (w-u_t^j(w))$ for $w\in D_j$. This yields the claimed formula for $\varrho(z)$.
\end{proof}

\subsection{Large $t$ asymptotics}

Let us now analyze the saddle points more explicitly for large values of $t$.
See Figure \ref{fig:G} for an illustration of the analysis that follows.

We denote by $\C_+\coloneq \{z\in\C:\Im(z)>0\}$ the upper half-plane, and by $\C_-\coloneq \{z\in\C:\Im(z)<0\}$ the lower half-plane, respectively. Recall that the saddle point equation \eqref{eq:sp_eqn_2} is equivalent to
\begin{align}\label{eq:sp_eqn_3}
	\frac{u-z}{t}\prod_{j=1}^d(u-\lambda_j)+\frac 1 \alpha \sum_{j=1}^d \alpha_j\prod_{k\neq j} (u-\lambda_k)=0
\end{align}
and has exactly $d+1$ solutions for $z\not\in B$. In order to motivate Definition \ref{def:u_pm} and Lemma \ref{lem:large_t} below, let us begin with a heuristic argument. As $t\to \infty$,  the first term of \eqref{eq:sp_eqn_3} vanishes and we can predict that $d-1$ saddle points $u_t^1,\dots, u_t^{d-1}$ will be close to  the zeros $\lambda_1',\ldots, \lambda_{d-1}'$ of the polynomial $u\mapsto  \frac 1 \alpha \sum_{j} \alpha_j\prod_{k\neq j} (u-\lambda_k)$. As we will see, the remaining two saddle points $u_t^d, u_t^{d+1}$ become unbounded and one of them becomes maximally relevant. Indeed, taking $u\to\infty$ in the initial saddle point equation \eqref{eq:sp_eqn_2}, where  $\frac 1 \alpha \sum_{j=1}^d \frac{\alpha_j}{u-\lambda_j}\sim \frac 1 u$, leads to the equation $\frac u {\sqrt t} +\frac{\sqrt t} u =\frac z{\sqrt t}$. Our aim is to show that $u_t^d, u_t^{d+1}$   are close to the two solutions of this equation whose left-hand side is essentially the Joukowsky function.
\begin{definition}
	\label{def:u_pm}
	We define the two solutions to $\frac u {\sqrt t} +\frac{\sqrt t} u =\frac z{\sqrt t}$ by
	\begin{align}
		\label{eq:u_pm}
		u_{t}^{\pm}(z):=\frac{1}2(z\pm\sqrt{z^2-4t}), \quad z\in\C\setminus \big( (-\infty,-2\sqrt t]\cup [2\sqrt t,\infty) \big),
	\end{align}
	where we choose the branch of the root such that $u_t^+(z)\sim z$ as $|z|\to\infty$ with $z\in \C_+$, and $u_t^-(z)\sim z$ as $|z|\to\infty$ with $z\in \C_-$, respectively.
	If $t=1$ we will denote the solutions by $u^{\pm}(z):=u_1^\pm(z)$.
\end{definition}

\begin{remark}
	\label{rem:square root}
The above corresponds to the choice $\sqrt{z^2-4t}\to 2i\sqrt t$ as $|z|\to 0$ and $\sqrt\cdot$ being the branch of the square root defined on $\C\setminus [0,\infty)$. Choosing this branch of the square root corresponds to $\arg(w)\in(0,2\pi)$ for $w\in\C$. Thus, for $z\in\C_+$, $\arg(z^2)\in (0,2\pi)$ and $\sqrt {z^2}=z$, while for $z\in\C_-$, we obtain $\sqrt{z^2}=-z$. In particular, this implies the above stated asymptotics $u_t^+(z)\sim z$ as $|z|\to\infty$ with $z\in \C_+$, and $u_t^-(z)\sim z$ as $|z|\to\infty$ with $z\in \C_-$.

Further, note that $u_t^\pm(z)\in\C_\pm$ for all $z\in \C$, where they are defined. Indeed, the above definition of $u_t^\pm(z)$ and $u^+_t(z)u_t^-(z)=t$ imply that $|u_t^+(z)|\ge\sqrt t$ for all $z\in \C_+$ and $|u_t^-(z)|\ge\sqrt t$ for all $z\in \C_-$. In particular,
$$\Im(z)=\Im\Big(u_t^\pm(z)+\frac{t}{u_t^\pm(z)}\Big)=\Im\big(u_t^\pm(z)\big)\Big(1-\frac{t}{|u_t^\pm(z)|^2}\Big),$$ so the claim $u_t^\pm(z)\in\C_\pm$ follows from comparing signs.

This choice reflects our conception that we switch holomorphic branches of saddle points on the curves where $\mu_t$ is supported, see Corollary \ref{cor:density}. Accordingly, we will see in \eqref{eq:Psi} below that the semicircle law arises on $[-2\sqrt t,2\sqrt t]$ when switching from $u^+_t$ to $u^-_t$ along that interval.
Alternatively, one may define $u_t^\pm$ with a branch cut at $[-2\sqrt t,2\sqrt t]$ as was done in \cite{heatflowrandompoly}, which would lead to technical obstacles in the upcoming harmonicity arguments.
\end{remark}

Let us now make the heuristic $u_t^d(z)\approx u_t^-(z)$, $u_t^{d+1}(z)\approx u_t^+(z)$ as $t\to\infty$ rigorous.

\begin{lemma}\label{lem:large_t}
For any polynomial $P$ and any $C>0$, there exists a (large) $t_0>0$ such that for all $t>t_0$ the following holds. For all $z\in D:=\{z\in\mathcal D_t: |z|<C\sqrt t,|z\pm 2\sqrt t|>2t^{1/3}\}$, the maximally relevant saddle point $u_t^*(z)$ is well defined and satisfies either
\begin{align*}
\Bigg\lvert\frac{u_t^*(z)}{u_t^+(z)}-1\Bigg\rvert<t^{-1/6}\qquad \text{or}\qquad \Bigg\lvert\frac{u_t^*(z)}{u_t^-(z)}-1\Bigg\rvert<t^{-1/6}.
\end{align*}

\end{lemma}

Observe that this indicates again the connection $u^*_t(z)\approx T_t^{-1}z=\frac{1}2(z\pm\sqrt{z^2-4t})$ to the inverse transport map of \cite{heatflowrandompoly}.

\begin{proof}[Proof of Lemma \ref{lem:large_t}]

In order to quantitatively locate the saddle points as forecast above, we are going to repeatedly apply Rouch\'e's theorem as in the proof of Lemma \ref{lem:domain}. Define 
$$\tilde P(u):=\frac 1 \alpha \sum_{i=1}^d \alpha_i \prod_{k\neq i} (u-\lambda_k)=\prod_{i=1}^{d-1}(u-\lambda_i')$$
 and note that $|\tilde P(u)|>\eps_t^{d-1}$ for all $u\in \partial B_{\eps_t}(\lambda_j')$. Hence, taking $\eps_t=t^{-1/4d}$, for sufficiently large $t$, yields $$\left|\frac{u-z}{t}\prod_{i=1}^d(u-\lambda_i)\right|<\frac{C}{t}<\frac{1}{\sqrt t}<t^{-1/4}<|\tilde P(u)|\quad\text{for all }u\in\partial B_{t^{-1/4d}}(\lambda_j'),$$
and some constant $C$.

By Rouch\'e's theorem, it follows that the saddle point equation $\frac{u-z}{t}\prod_i(u-\lambda_i)+\frac 1 \alpha \sum_i \alpha_i\prod_{k\neq i} (u-\lambda_k)=0$ and $\frac 1 \alpha \sum_i \alpha_i \prod_{k\neq i} (u-\lambda_k)=0$ have the same number of solutions in $B_{t^{-1/4d}}(\lambda_j')$. We record these solutions as saddle points $u_t^1(z),\dots,u^{d-1}_t(z)$, each of which converges to $\lambda_1',\dots,\lambda_{d-1}'$ (not necessarily different) as $t\to\infty$, respectively.

In particular, we obtain a uniform bound on these saddle points and their height
\begin{align*}
G(z,u_t^j(z))=\frac 1 \alpha \sum_{i=1}^d \alpha_i \log |u_t^j(z)-\lambda_i|+\frac 1 {2t}\Re(z-u_t^j(z))^2 \le K <\infty,\quad j\le d-1
\end{align*}
for some constant $K$.
The remaining two saddle points will be close to the two solutions $u_t^{\pm}(z)=\frac{1}2(z\pm\sqrt{z^2-4t})$ of $\frac u {\sqrt t} +\frac{\sqrt t} u =\frac z{\sqrt t}$. First, note that $|z|<C\sqrt t$ implies $$|u_t^\pm(z)|\le \frac 1 2 (|z|+\sqrt {|z|^2+4t})\le |z|+\sqrt t<(C+1)\sqrt t$$ and it follows that $|u_t^\pm(z)|>\sqrt t/(C+1)$ from $u_t^+(z)u_t^-(z)=t$.
Observe that our assumption $|z\pm 2\sqrt t|>2t^{1/3}$ implies $|u_t^+(z)-u_t^-(z)|=|\sqrt{z^2-4t}|>2t^{1/3}$, hence
\begin{align}\label{eq:Jouk_lower}
\left| \frac u {\sqrt t} +\frac{\sqrt t} u -\frac z {\sqrt t}\right|=\left| \frac {(u-u_t^+)(u-u_t^-)} {u\sqrt t}\right| > \frac {t^{2/3}} {4|u|\sqrt t}>\frac{1}{4(C+2)t^{1/3}},\quad u\in\partial B_{ t^{1/3}}(u_t^\pm(z))
\end{align}
since $|u|<|u_t^\pm(z)|+{t^{1/3}}< (C+2)\sqrt t$ for $t>1$.
On the other hand, for $t$ sufficiently large (only depending on $\lambda_{\max}:=\max_i\{|\lambda_i|\}$) and using $|u_t^\pm(z)|>\sqrt t/(C+1)$, it holds 
$$\left|\frac{\sqrt t}{\alpha}\sum_{i=1}^d\frac{\alpha_i}{u-\lambda_i}-\frac {\sqrt t} u\right|=\frac {\sqrt t}{\alpha}\left|\sum_{i=1}^d\frac{\lambda_i \alpha_i}{(u-\lambda_i)u}\right|\le \frac{2(C+1)^2\lambda_{\max}}{\sqrt t},\quad u\in\partial B_{ t^{1/3}}(u_t^\pm(z)),$$
which can be made smaller than \eqref{eq:Jouk_lower}.
Therefore, by Rouch\'e's theorem, $\frac u {\sqrt t} +\frac{\sqrt t} u -\frac z {\sqrt t}$ and $\frac{\sqrt t}{\alpha}\sum_{i=1}^d\frac{\alpha_i}{u-\lambda_i}+\frac{u-z}{\sqrt t}$ have the same number of zeros in both $B_{t^{1/3}}(u_t^\pm(z))$, which is one each. Locally, we denote them by $u_t^d(z)$ and $u_t^{d+1}(z)$ such that  $|u_t^{d}(z)-u_t^-(z)|<t^{1/3}$ and $|u_t^{d+1}(z)-u_t^+(z)|<t^{1/3}$. We claim that either $u^{d+1}_t(z)$ or $u^{d}_t(z)$ is the maximally relevant saddle point.
Their height for such $j=d,d+1$ is given by
\begin{align*}
G(z,u_t^j(z))=\frac 1 \alpha \sum_i \alpha_i\log |u_t^j(z)-\lambda_i|+\frac 1 {2t}\Re(z-u_t^j(z))^2 = \frac 1 2 \log t +\mathcal O (1),
\end{align*}
eventually surpassing the bounded height of $j\le d-1$ for sufficiently large $t$.
It remains to verify that at least one of the saddle points $u_t^{d},u_t^{d+1}$ is not irrelevant. Inspecting the proof of Lemma \ref{lem:saddle}, observe that at height $h_\varepsilon=\min(G(z,u_t^{d}(z)),G(z,u_t^{d+1}(z)))-\varepsilon$, before the lower one of $u_t^d,u_t^{d+1}$ is attained, the saddle points $u_t^1,\dots,u^{d-1}_t$ have already been crossed due to their lower height. Consequently, $\Sigma (h_\eps)$ has  three connected components, with $\Sigma^+(h_\varepsilon)$ and $\Sigma^-(h_\varepsilon)$ (containing $\pm i\infty$) not connected to any saddle point $u_t^j,j<d$. Indeed, for $j<d$, $ |u|=t^{1/3}$ and $t$ sufficiently large, we have
$$G(z,u_t^j(z))<G(z,u)\sim \frac 1 3\log t<h_\varepsilon$$
and this circular barrier prevents $u_t^j(z)$ to be connected to $\pm i \infty$ at heights $\frac 1 3 \log t$.
\end{proof}

Let us now turn to the proof of Theorem \ref{thm:large_t}.

\begin{proof}[Proof of Theorem \ref{thm:large_t}]
	Let $K$ be a fixed disk. First, we argue why we may choose $t$ sufficiently large such that $K\subset D$, where $D$ is defined in Lemma \ref{lem:large_t}. Recall that during the proof of Lemma \ref{lem:large_t} we chose the different branches  $|u_t^{d}(z)-u_t^-(z)|<t^{1/3}$ and $|u_t^{d+1}(z)-u_t^+(z)|<t^{1/3}$, for $z\in K$, and all other branches $u_t^j$ for $j<d$ to be the ones  close to $\lambda_j'$. In particular, $u_t^d(z)$ and $u_t^{d+1}(z)$ will not coalesce with any other branch, and hence are well defined analytic functions on $K$ according to Lemma \ref{lem:analyt}.
	In particular, $K\subset D$ is simply connected and 
	\begin{align*}
		\Gamma_{d, d+1}^K=\{z\in K:G(z, u_t^d(z))=G(z,u_t^{d+1}(z))\}
	\end{align*}
	is well defined.
	
We turn to the geometric claim of Theorem \ref{thm:large_t} that the curve $\supp(\mu_t)\cap K=\Gamma_{d, d+1}^K $ converges as $t\to\infty$ in Hausdorff metric to a horizontal line through the center of mass $\frac 1  \alpha \sum_{j=1}^d\alpha_j\lambda_j$. For fixed or bounded $z\in\C$, we saw that $u_t^*(z)\sim \pm i \sqrt t$ with $G(z,u_t^*(z))\sim \frac 1 2 \log t$. Thus, no saddle point $u_t^j(z)$ for $j<d$ becomes irrelevant if $t$ is sufficiently large and only $u_t^d,u_t^{d+1}$ can be maximally relevant. 
	By Lemma \ref{lem:large_t}, the saddle points $u_t^{d}, u_t^{d+1}$ will be close to
	\begin{align*}
		u_t^\pm(z) = \frac{1}{2} \left( z \pm \sqrt{z^2 - 4t} \right)
		= \frac{z}{2} \pm i \sqrt{t} \sqrt{1 + \frac{z^2}{4t}} = \pm i \sqrt{t} + \frac{z}{2} + \mathcal O\left(t^{-1/2}\right)
	\end{align*}
	by Taylor's approximation. In particular, the branching points where $u_t^d,u_t^{d+1}$ coalesce are close to those of $u_t^\pm$, which are located at $\pm 2 \sqrt t\not \in K$ for $t>0$ sufficiently large. Hence, for some $w \in \C$ to be determined we write
	\begin{align*}
		u_t^{d+1}(z) &= u_t^+(z)  + \frac w 2 = i \sqrt{t} + \frac{z+w}{2}  + \mathcal O\left(t^{-1/2}\right)\\
		u_t^{d}(z) &= u_t^-(z)  + \frac w 2 = -i \sqrt{t} + \frac{z+w}{2}   + \mathcal O\left(t^{-1/2}\right)\
	\end{align*}
	for $t \to \infty$ uniformly in $z\in K$, by Lemma \ref{lem:large_t}  and the Taylor approximation.
	For $u=u_t^d(z),u_t^{d+1}(z)$, the saddle point equation reads
	\begin{align*}
		0 &= u - z + \frac{t}{\alpha} \sum_j \frac{\alpha_j}{u - \lambda_j}\\
		&=  \pm i \sqrt{t} + \frac{z+w}{2}   -z + \frac{t}{\alpha} \sum_j \frac{\alpha_j}{\pm i \sqrt{t} + \frac{z+w}{2}  - \lambda_j} + \mathcal O\left(t^{-1/2}\right)\\
		&=  \pm i \sqrt{t} -\frac{z+w}{2}  + \frac{t}{\alpha} \sum_j \left( \frac{\alpha_j}{\pm i \sqrt{t}} - \frac{\alpha_j(z/2 + w/2 - \lambda_j)}{(\pm i \sqrt{t})^2} \right) + \mathcal O\left(t^{-1/2}\right)\\
		&=  \frac{w-z}{2}  + \frac{1}{\alpha} \sum_j \alpha_j \left( \frac{z+w}{2} +  - \lambda_j \right) + \ \mathcal O\left(t^{-1/2}\right) \\
		&=w - \frac{1}{\alpha} \sum_j \alpha_j \lambda_j +  \mathcal O\left(t^{-1/2}\right),
	\end{align*}
where we used $1/(a+b)=1/a-b/a^2+b^2/(a^2(a+b))$. Thus, $
w = \frac{1}{\alpha} \sum_j \alpha_j \lambda_j$ is the center of mass and we obtain uniformly in $z\in K$
	\begin{align*}
		u_t^k(z) = \pm i \sqrt{t} + \frac{z}{2} + \frac{1}{2\alpha} \sum_{j=1}^d \alpha_j \lambda_j +  \mathcal O\left(t^{-1/2}\right), \quad \text{for } k = d, d+1.
	\end{align*}
	Now, consider $z\in K$ with $G(z, u_t^d(z)) = G(z, u_t^{d+1}(z))$, then
	\begin{align*}
		0 &=\frac 1 \alpha \sum_{j=1}^d \alpha_j \log \left\lvert \frac{u_t^d(z) - \lambda_j}{u_t^{d+1}(z) - \lambda_j} \right\rvert
		+ \frac{1}{2t} \Re \left[ \left(u_t^d(z) - z \right)^2 - \left(u_t^{d+1}(z) - z \right)^2 \right] \\
		&= \frac{1}{\alpha} \sum_{j=1}^d \alpha_j \log \left\lvert \frac{-i \sqrt{t} + \frac{z+w}{2}  - \lambda_j}{i \sqrt{t} + \frac{z+w}{2}  - \lambda_j} +\mathcal O\left(t^{-1}\right)\right\rvert  + \frac{1}{2t} \Re \left[ -2i \sqrt{t} \left( \frac{-z+w}{2}  \right) \cdot 2 \right] + \mathcal O\left(t^{-1}\right) \\
		&= \frac{1}{\alpha} \sum_{j=1}^d \alpha_j \log \left| 1 - \frac{z + w - 2\lambda_j }{i \sqrt{t}} +\mathcal O\left(t^{-1}\right)\right|
		+ \frac{1}{\sqrt t} \Re(i (z - w))+\mathcal O\left(t^{-1}\right)\\
		&= \frac{1}{\alpha} \sum_{j=1}^d \alpha_j \Re \left( \frac{i \left( z + w - 2 \lambda_j \right)}{\sqrt{t}} \right)
		+ \frac{1}{\sqrt t} \Im( w-z) +\mathcal O\left(t^{-1}\right)\\
		&= \frac{1}{\sqrt t} \Im\left( -z - w  + w - z + \frac{2}{\alpha} \sum_{j=1}^d \alpha_j \lambda_j \right) + \mathcal O\left(t^{-1}\right)\\
		&=\frac{2}{\sqrt t}\Im\left( -z  + \frac{1}{\alpha} \sum_{j=1}^d \alpha_j \lambda_j\right) + \mathcal O\left(t^{-1}\right),
	\end{align*}
	where we used the above decomposition of $1/(a+b)$ again, as well as a Taylor approximation of $\log|1+x|$ in the fourth line. Again, all the errors $\mathcal O (t^{-1})$ are uniform in $z\in K$.
	We conclude the local uniform convergence of the harmonic functions
	$$\sqrt t \big(G(z, u_t^d(z))-G(z, u_t^{d+1}(z))\big) \to \Im(\frac{2}{\alpha} \sum_{j=1}^d \alpha_j \lambda_j -2z).$$
	Since $z\mapsto \Im(z)$ is strictly increasing in the imaginary direction and its nodal set is the real axis, it follows that also the nodal set of $\Im(\frac{2}{\alpha} \sum_{j=1}^d \alpha_j \lambda_j -2z)$ is a single straight line, given by $\Im(z) = \Im\left( \frac{1}{\alpha} \sum_{j=1}^d \alpha_j \lambda_j \right)$. Local uniform convergence of the derivatives follows from e.g.~\cite[Theorem 2.6]{Axler}, such that $\partial_z \sqrt t \big(G(z, u_t^d(z))-G(z, u_t^{d+1}(z))\big)\to i$, and hence
	\begin{align}
		\label{eq:G_derivative}
		\frac{\partial}{\partial y} \sqrt t \big(G(x+iy, u_t^d(x+iy))-G(x+iy, u_t^{d+1}(x+iy))\big)\to -2.
	\end{align}
	Therefore, the nodal set $\Gamma_{d,d+1}^K$ is a single curve for sufficiently large $t$. Further, any point $z\in \Gamma_{d,d+1}^K$ satisfies
	\begin{align*}
		\Im(z) = \Im\left( \frac{1}{\alpha} \sum_{j=1}^d \alpha_j \lambda_j \right)  +\mathcal O\left(t^{-1/2}\right),
	\end{align*}
in other words, $\supp(\mu_t)\cap\Gamma_{d,d+1}^K$ converges as $t \to \infty$ to the horizontal line through the center of mass of $\lambda_j$.
	
It remains to verify the claim \eqref{eq:U=max}. We have seen in the proof of Lemma \ref{lem:large_t} that the saddle points  satisfy $G(z, u_t^j(z))< G(z, u_t^*(z))$, for all  $j=1,\ldots, d-1$ and $z\in K$. It remains to show that neither $u_t^d(z)$ nor $u_t^{d+1}(z)$ are irrelevant.

As we have just shown, the nodal set $\Gamma_{d, d+1}^K=\{z\in K:G(z, u_t^d(z))=G(z,u_t^{d+1}(z))\}$ consists of a single curve  for sufficiently large $t$. By \eqref{eq:G_derivative}, we have $G(z,u_t^{d+1}(z))>G(z,u_t^d(z))$ for all $z$ lying above $\Gamma_{d, d+1}^K$ and $G(z,u_t^{d+1}(z))<G(z,u_t^d(z))$ for all $z$ lying below. Hence, we aim to show that $u_t^{d+1}(z)$ is the maximally relevant saddle point in the domain above $\Gamma_{d, d+1}^K$, and, analogously, that $u_t^d(z)$ is maximally relevant in the domain below. By Lemma \ref{lem:sp_nbh}, it suffices to determine the maximally relevant saddle point for one value of $z$ in the domains separated by $\Gamma_{d, d+1}^K$.

To this end, fix some $z\in K$ that lies above the curve $\Gamma_{d, d+1}^K$. 
Suppose the function $x\mapsto G(z,u_t^{d+1}(z)+x)$ is convex on $\R$, which we will verify below. Since we already know that $G(z,u_t^{d+1}(z))>G(z,u_t^d(z))$, this implies $G(z,u_t^{d+1}(z)+x)>G(z,u_t^d(z))$ for all $x\in \R$. Hence, $u_t^{d+1}(z)+\R$ is a 'barrier' that prevents the components $\Sigma^+(h)$ and $\Sigma^-(h)$ to be connected at height $h<G(z,u_t^{d+1}(z))$, and we can conclude that $u_t^{d+1}(z)$ must be maximally relevant. 

In order to check the claimed convexity, we follow the lines of \eqref{eq:convexity} providing 
\begin{align}
	\label{eq:convexity2}
	\frac{\partial^2}{\partial x^2}G(z,u_t^{d+1}(z)+x)
	=\Re\left(\frac 1 t-\frac 1 \alpha \sum_{j=1}^d \frac{\alpha_j}{(u_t^{d+1}(z)+x-\lambda_j)^2}\right) .
\end{align}
Recall that $u_t^{d+1}(z)= i\sqrt t+\mathcal O (1)$, such that for $|x|>\sqrt{t}/2 $ we obtain 
\begin{align*}
	\frac{\partial^2}{\partial x^2}G(z,u_t^{d+1}(z)+x) \ge \frac 1 t -\frac 1 \alpha \sum_{j=1}^d \frac{\alpha_j}{|u_t^{d+1}(z)+x-\lambda_j|^2}\ge \frac 1 t -\frac 1 \alpha \sum_{j=1}^d \frac{\alpha_j}{t/4+t+\mathcal O(1)}\ge \frac 1 {7t}>0
\end{align*}
for sufficiently large $t$. On the other hand, for $|x|\le {\sqrt t}/{2}$ it holds $$(x+\mathcal O(1))^2-(\Im u_t^{d+1}(z)+\mathcal O(1))^2<0$$ for sufficiently large $t$, which also implies
\begin{align*} 
	\frac{\partial^2}{\partial x^2}G(z,u_t^{d+1}(z)+x) =\frac 1 t-\frac 1 \alpha \sum_{j=1}^d \alpha_j\frac{ (x+\Re u_t^{d+1}(z)-\Re \lambda_j)^2-(\Im u_t^{d+1}(z)-\Im \lambda_j)^2 }{| u_t^{d+1}(z)+x-\lambda_j|^4}>\frac 1 t >0.
\end{align*}

This settles the claimed convexity of $x\mapsto G(z,u_t^{d+1}(z)+x)$. 
The same line of argument choosing some $z$ below $\Gamma_{d, d+1}^K$ proves that $u_t^{d}(z)$ is the maximally relevant saddle point for $z\in K$ lying below the curve $\Gamma_{d, d+1}^K$. So, Lemma \ref{lem:sp_nbh} implies that $u_t^{d+1}(z)$ is maximally relevant for all $z\in K$ located above $\Gamma_{d, d+1}^K$ and $u_t^{d}(z)$ is maximally relevant for all $z\in K$ located below. In particular, both $u_t^{d}(z)$ and $u_t^{d+1}(z)$ are not irrelevant and \eqref{eq:U=max} holds.
\end{proof}

\subsection{Proof of Theorem \ref{thm:semicircle}}

In the upcoming section, we will use the following rescaling of $z$. For fixed $\tilde z$ in $\C$ we denote by $u_t^\pm(\sqrt t \tilde z)$ the solutions of $\frac{u}{\sqrt t}+\frac{\sqrt t}{u}=\tilde z$

\begin{proof}[Proof of Theorem \ref{thm:semicircle}]
We prove the weak convergence of $\tilde \mu_t=\mu_t(\sqrt t\cdot )$ to the standard semicircle distribution $\mathsf{sc}_1$ by considering the respective logarithmic potentials. To this end, we show that for large $t$, $\tilde U_t(\tilde z)=U_t(\sqrt t \tilde z)- \frac 1 2 \log |t| \sim \Psi(\tilde z)$, where
\begin{align}
	\label{eq:Psi}
	\Psi(\tilde z)=
	\begin{cases}
		\log\left|\frac{\tilde z+\sqrt{\tilde z^2-4}}{2}\right|+\frac 1 4 \Re(\tilde z^2-\tilde z\sqrt{\tilde z^2-4})-\frac 1 2\quad \text{for } \tilde z \in \C_+\cup(-2\sqrt t, +2\sqrt t)\\
		\log\left|\frac{\tilde z-\sqrt{\tilde z^2-4}}{2}\right|+\frac 1 4 \Re(\tilde z^2+\tilde z\sqrt{\tilde z^2-4})-\frac 1 2 \quad \text{for } \tilde z \in \C_-\cup(-2\sqrt t, +2\sqrt t)
	\end{cases}
\end{align}
 is the logarithmic potential of $\mathsf{sc}_1$, see for instance \cite[\S 2.4]{heatflowrandompoly}. Note that therein, the authors chose a different branch of the square root from the one used here, hence the deviating formula. Recall that $\mathsf{sc}_1$ is a probability measure concentrated in $\{|z|\le 2\}$. By Lemma \ref{lem:P_t^n roots},  $(\tilde \mu_t)_{t>0}$ is a tight family of probability measures and it therefore suffices to prove vague convergence. We shall see that it suffices to show vague convergence in the disk $D\coloneq \overline{B_{2-\eps}(0)}$, since $\mathsf{sc}_1(D)>1- \varepsilon$ for all $\varepsilon>0$.  

Fix $\tilde z\in D\setminus \R$. By Lemma \ref{lem:large_t}, the maximally relevant saddle point $u_t^*(\sqrt t\tilde z)$ is either close to $u_t^+(\sqrt t\tilde z)$ or $u_t^-(\sqrt t \tilde z)$ for large values of $t$ so,
\begin{align*}
U_t(\sqrt t \tilde z)&=G(\sqrt t \tilde z, u_t^*(\sqrt t \tilde z))\sim\frac 1 \alpha \sum_j\alpha_j \log|u_t^\pm(\sqrt t \tilde z)-\lambda_j| + \frac{1}{2t} \Re(\sqrt t \tilde z-u_t^\pm(\sqrt t \tilde z))^2\\
&=\frac 1 \alpha \log \left|\prod_j (u_t^\pm(\sqrt t \tilde z)-\lambda_j)^{\alpha_j}\right|+\frac {1} {2t} \Re(u_t^\mp(\sqrt t \tilde z))^2\\
&\sim \log \left|\frac{\sqrt {t}}{2} (\tilde z \pm \sqrt {\tilde z^2 -4})\right|+\frac {1} {2t} \Re(u_t^\mp(\sqrt t \tilde z))^2\\
&=\frac 1 2 \log|t|+\log\left|\frac{\tilde z\pm\sqrt{\tilde z^2-4}}{2}\right|+\frac 1 4 \Re(\tilde z^2\mp\tilde z\sqrt{\tilde z^2-4})-\frac 1 2.
\end{align*}
In particular, $\tilde U_t(\tilde z)$ converges pointwise to one of the two solutions, which we claim to be \eqref{eq:Psi} as will be shown below.
Then, from Lemma \ref{lem:upperbound}, we have the upper bound $\tilde U_t(\tilde z)<c$ uniformly in $|\tilde z|<2-\varepsilon$. Hence, taking the same route as in the proof of Theorem \ref{thm:measure}, see \eqref{eq:L1loc}, we obtain $L^1_{\mathrm{loc}}$-convergence of the logarithmic potential on $D$, i.e.~vague convergence of the sequence $\tilde \mu_t$ on $D$. 

It remains to determine the maximally relevant saddle point and the limit of $\tilde U_t(\tilde z)$. By Lemma \ref{lem:sp_nbh}, it suffices to consider some $\tilde z$ in each of the components separated by the curves in the nodal set $\{G(\sqrt t \tilde z,u^d(\sqrt t \tilde z))=G(\sqrt t \tilde z, u^{d+1}(\sqrt t \tilde z))\}$, which we now aim to locate more explicitly.
For sufficiently large $t$, we see that
\begin{align}
\label{eq:function_H}
G(\sqrt t \tilde z, u_t^+(\sqrt t \tilde z))-G(\sqrt t \tilde z, u_t^-(\sqrt t \tilde z))&\sim \log \left |\frac{u_t^+(\sqrt t \tilde z)}{u_t^-(\sqrt t \tilde z)}\right |+\frac {1}{2t} \Re(u_t^-(\sqrt t \tilde z)^2-u_t^+(\sqrt t \tilde z)^2) \notag \\
&= \log \left|\frac 1 2 (\tilde z+\sqrt{\tilde z^2-4})\right|^2+\frac 1 2 \Re(-\tilde z\sqrt{\tilde z^2-4}) \notag \\
&=:H(\tilde z),
\end{align}
where we used that $u_t^+(\sqrt t \tilde z)u_t^-(\sqrt t \tilde z)=t$ in the last equality. Recall that we chose the branch of $\sqrt{\tilde z^2-4}$ in Definition \ref{def:u_pm} such that $H(\tilde z)$ is a harmonic function in $D$. Now taking the derivative in $\tilde z$ yields
\begin{align*}
\partial_{\tilde z} H(\tilde z) &= \frac{1}{\tilde z+\sqrt{\tilde z^2-4}}\left(1+ \frac{\tilde z}{\sqrt{\tilde z^2-4}}\right)-\frac 1 4\left(\sqrt{\tilde z^2-4}+\frac{\tilde z^2}{\sqrt{\tilde z^2-4}}\right) \\
&=\frac{1}{\sqrt{\tilde z^2-4}}\left(1-\frac 1 2 \tilde z^2 +1\right)=-\frac 1 2\sqrt{\tilde z^2-4}.
\end{align*}
Here, we used that $\partial_z \log|z|^2=\frac 1 z$ is the Wirtinger derivative in the complex variable $z$.
In particular,
\begin{align}
\label{eq:function_H_derivative}
\frac{\partial}{\partial \tilde y}\frac{H(\tilde x +i\tilde y)}{2}=-\Im \partial_{\tilde z} H(\tilde z)=\frac 1 2 \Im(\sqrt{(\tilde x+i\tilde y)^2-4})>0,
\end{align}
by our choice of the branch of the square root, see Remark \ref{rem:square root}.
Since $H(\tilde x)=0$ for all $\tilde x \in(-2 ,2)$, this implies $H(\tilde z)>0$ for $\tilde z\in\C_+$ and $H(\tilde z)<0$ for $\tilde z\in\C_-$.	In particular, for any $\epsilon>0$ there exists a $\delta>0$ such that $H(\tilde z)>\epsilon$ for $\tilde z\in D$ with $\Im(\tilde z)>\delta$ and $H(\tilde z)<-\epsilon$ for $\tilde z\in D$ with $\Im(\tilde z)<-\delta$. By convergence \eqref{eq:function_H}, we have $G(\sqrt t\tilde z,u_t^+(\sqrt t \tilde z))-G(\sqrt t \tilde z,u_t^-(\sqrt t \tilde z))>\epsilon/2$ for all $\tilde z\in D$ with $\Im(\tilde z)>\delta/2$, and analogously negative in the lower half-plane. Again, by convergence of harmonic functions (and their derivatives), we obtain that the nodal set of $G(\sqrt t \tilde z,u_t^d(\sqrt t \tilde z))-G(\sqrt t \tilde z,u_t^{d+1}(\sqrt t \tilde z))$ is a single curve $\gamma_d$ within the strip $(-2+\varepsilon,2-\varepsilon)+i(-\delta/2,\delta/2)$.

It remains to determine the maximally relevant saddle point in the regions separated by this curve.  
We repeat the argument we have presented in the proof of Theorem \ref{thm:large_t}, except that we now consider some explicit $\tilde z\in D$ that lies above $\gamma_d$, namely $\tilde z =i$, and verify that $u_t^{d+1}(\sqrt t\tilde z)$ is indeed the maximally relevant saddle point. The same line of argument choosing $\tilde z=-i$ lying below $\gamma_d$ proves that $u_t^{d}(\sqrt t\tilde z)$ is the maximally relevant saddle point below the curve $\gamma_d$.
Once again, we use that we have seen $G(\sqrt t\tilde z, u_t^{d+1}(\sqrt t \tilde z))>G(\sqrt t\tilde z, u_t^{d}(\sqrt t\tilde z ))$ for our chosen $\tilde z=i$. It can be shown analogously to \eqref{eq:convexity2} and below, that $x\mapsto G(\sqrt t\tilde z, u_t^{d+1}(\sqrt t \tilde z)+x)$ is convex on $\R$: We have that $u_t^{d+1}(i\sqrt t)=ic\sqrt t+\mathcal O(1)$, where $c>0$ by our choice of the square root, see Remark \ref{rem:square root}. Then, the same calculation as already presented in \eqref{eq:convexity2} yields a barrier through the saddle point $u_t^{d+1}$ that prevents $\Sigma^+(h)$ to be connected to $\Sigma^-(h)$ at heights $h<G(i\sqrt t,u_t^{d+1}(i\sqrt t))$.

So, Lemma \ref{lem:sp_nbh} implies that $u_t^{d+1}(\sqrt t\tilde z)$ is maximally relevant for all $\tilde z\in D$ located above the curve $\gamma_d$ and $u_t^{d}(\sqrt t\tilde z)$ is maximally relevant for all $\tilde z\in D$ located below. Hence, $\tilde U_t(\tilde z)\to\Psi(\tilde z)$ as $t\to\infty$ in $L^1_{\mathrm{loc}}(D)$, i.e.~$\tilde \mu_t\big|_D \to \mathsf{sc}_1\big|_D$ vaguely.

It remains to remove the $\varepsilon$-constraint on $D$. In particular, $\tilde\mu_t(D)\to \mathsf{sc}_1(D)\geqslant 1-\frac \eps 2$ and it follows that $\tilde \mu_t(D)>1-\eps$ for sufficiently large $t$. Finally, let $\varphi\in\mathcal C_c^\infty$ be a test function, then we have
\begin{align*}
\left|\int_{\C} \varphi d\tilde\mu_t-\int_{\C} \varphi d\mathsf{sc}_1\right|\le \left|\int_{D} \varphi d\tilde\mu_t-\int_{D} \varphi d\mathsf{sc}_1\right|+\left|\int_{D^c} \varphi d\tilde\mu_t-\int_{D^c} \varphi d\mathsf{sc}_1\right|\le \eps+ \|\varphi\|_\infty 2\eps .
\end{align*}
Since we know that the vague limit is the probability distribution $\mathsf{sc}_1$, we have proven weak convergence of $\tilde \mu_t$.
\end{proof}

Recall from the proof of Lemma \ref{lem:large_t} that the $d-1$ saddle points $u_t^1(z),\dots,u_t^{d-1}(z)$ converge to the zeros $\lambda_1',\dots,\lambda_{d-1}'$ of $\tilde P(u)=\frac 1 \alpha \sum_{i=1}^d \alpha_i \prod_{k\neq i} (u-\lambda_k)$ as $t\to\infty$, and, for large $t$, the solutions $u_t^{d}(z)$, $u_t^{d+1}(z)$ are close to $u_t^-(z)$, $u_t^+(z)$ as defined in \eqref{eq:u_pm}. Consistent with Corollary \ref{cor:density}, we have the following.

\begin{corollary}\label{cor:large_t_rescaled}
For any $\tilde z\in\C_+$ and sufficiently large $t$, $u_t^{d+1}(\sqrt t \tilde z)$ is the maximally relevant saddle point, and for any $\tilde z\in\C_-$ and sufficiently large $t$, so is $u_t^{d}(\sqrt t \tilde z)$.
\end{corollary}

\subsection{Small $t$ asymptotics}
\label{subsec:small_t}
The proof of Theorem \ref{thm:small_t} follows a similar strategy as we have used in the large-$t$ case above. We begin by identifying the saddle points that could potentially become maximally relevant, analogous to Lemma \ref{lem:large_t}.

\begin{lemma}
\label{lem:small_t}
For any $C_{\max}>0$ there exist some (large) $c>0$ and some (small) $t_0>0$ such that $D\coloneq \{z\in\C\setminus B:\,|z|<C_{\max}, \min_{j\le d}|z-\lambda_j|>c\sqrt t\} $ satisfies $\mu_t(D)=0$ for all $t<t_0$. In particular, $U_t$ is locally harmonic in $D$, and $u_t^*(z)$ satisfies $|u_t^*(z)-z|<\frac c 2 \sqrt t$ for all $z\in D$.
\end{lemma}

\begin{proof}

We claim that for small enough values of $t$, $d$ saddle points lie close to the roots $\lambda_1,\dots,\lambda_d$ of $P^n$ and the maximally relevant saddle point close to $z$. Indeed, assume that $|z|<C_{\max}$ and $|z-\lambda_j|>c_1\sqrt t$ for all $j=1,\dots ,d$ and a constant $c_1>2$. We define the radius $r_1\coloneq \frac{c_1\sqrt t}{2}$, consider the saddle point equation $u-z+\frac t \alpha \sum_{i=1}^d \frac{\alpha_i}{u-\lambda_i}=0$, and observe that $$\left|\frac t \alpha \sum_i\frac{\alpha_i}{u-\lambda_i} \right|<\frac t \alpha \sum_i\frac{\alpha_i}{r_1}=\frac{t}{r_1}<r_1=\frac{c_1\sqrt t}{2}=|u-z|, \quad  u\in\partial B_{r_1}(z)$$ if $c_1>2$, since $|u-\lambda_j|>\frac{c_1\sqrt t}{2}$ if $|z-\lambda_j|>c_1\sqrt t$.
 Hence, Rouché's theorem yields that $u-z+\frac t \alpha \sum_{i=1}^d \frac{\alpha_i}{u-\lambda_i}=0$ has the same number of solutions in $\partial B_{r_1}(z)$ as $u-z=0$, that is one solution. We denote this saddle point by $u_t^{d+1}(z)$. In particular, for all $z\in D_1\coloneq\{z\in\C:|z|<C_{\max}, \min_{j\le d}|z-\lambda_j|>c_1\sqrt t\}\setminus B$, we have $|u_t^{d+1}(z)-z|<\frac{c_1\sqrt t}{2}$.

We use a similar argument to locate the other $d$ saddle points. Assume now that $|z-\lambda_j|>c_2\sqrt t$ such that $c_2$ fulfills $c_2>K_2(\delta(\lambda), \lambda_{\max}, \alpha_{\max})$ for a constant $K_2$, yet to be defined, that only depends on $\lambda_{\max}=\max_i\{|\lambda_i|\}$, $\alpha_{\max}=\max_i\{|\alpha_i|\}$  and $\delta(\lambda)\coloneq \min_{i\neq j}|\lambda_i-\lambda_j|$, the minimal distance between the roots. We now consider the equation $\frac t \alpha \sum_{i=1}^d\alpha_i\prod_{k\neq i}(u-\lambda_k)-(z-u)\prod_{i=1}^d(u-\lambda_i)=0$. In this case, we define the radius $r_2\coloneq \frac{c_2\sqrt t}{2}$ and observe that for $t<t_1:=\frac{4}{c_2^2}$, $$\left|\frac t \alpha \sum_{i=1}^d\alpha_i\prod_{k\neq i}(u-\lambda_k)\right|<\frac {td} {\alpha} \left(|\alpha_{\max}|(1+2|\lambda_{\max}|)^{d-1}\right)\quad\text{and}\quad |u-z|>\frac 1 2 |z-\lambda_j|,\quad u\in\partial B_{r_2}(\lambda_j).$$

Further, for $t<t_2:=\frac{\delta(\lambda)^2}{c_2^2}$, $$\left|(z-u)\prod_{i=1}^d (u-\lambda_i)\right|=r_2|z-u|\prod_{k\neq j}|u-\lambda_k|>r_2\frac{c_2\sqrt t}{2}\left(\frac{\delta(\lambda)}{2}\right)^{d-1},\quad u\in\partial B_{r_2}(\lambda_j).$$
(Note that, for $t<t_2$, $r_2<\frac 1 2 \delta(\lambda)$ such that every ball $B_{r_2}(\lambda_j)$ only contains one zero for all $j=1,\dots,d$.)

So,
 $$\left|\frac t \alpha \sum_{i=1}^d\alpha_i\prod_{k\neq i} (u-\lambda_k)\right|<\frac{c_2^2t}{4}\left(\frac{\delta(\lambda)}{2}\right)^{d-1}<r_2|z-u|\prod_{k\neq j}|u-\lambda_k|=\left|(z-u)\prod_{i=1}^d (u-\lambda_i)\right|,\quad u\in\partial B_{r_2}(\lambda_j)$$ if $c_2>\sqrt{\frac{4d(|\alpha_{\max}|(1+2|\lambda_{\max}|))^{d-1}}{(\delta(\lambda)/2)^{d-1}}}$.
Again, Rouché's theorem implies that $\frac t \alpha \sum_{i=1}^d\alpha_i\prod_{k\neq i}(u-\lambda_k)-(z-u)\prod_{i=1}^d(u-\lambda_i)=0$ has the same number of solutions in $\partial B_{r_2}(\lambda_j)$ as $(z-u)\prod_{i=1}^d(u-\lambda_i)=0$, i.e.~one solution, which we record by $u_t^j(z)$ for each $j=1,\dots,d$, respectively. Hence, for all $z\in D_2:=\{z\in\C:|z|<C_{\max}, \min_{j\le d}|z-\lambda_j|>c_2\sqrt t\}\setminus B$, $|u_t^j(z)-\lambda_j|<\frac{c_2\sqrt t}{2}$. Combining the results above, we obtain that for all $z$ within simply connected domains of $D_1\cap D_2$
\begin{align}
\label{eq:sp_dist}
|u_t^{d+1}(z)-z|<\frac{c_3\sqrt t}{2}\quad \text{and}\quad |u_t^j(z)-\lambda_j|<\frac{c_3\sqrt t}{2},
\end{align}
where $c_3:=\max(c_1, c_2)$ and $t<t_0:= \min(t_1, t_2)$.
It remains to show that $u_t^{d+1}(z)$ is, in fact, the maximally relevant saddle point.

First, we shall argue why  $u_t^{d+1}(z)$ is not irrelevant.   As we have done in the proof of Theorem \ref{thm:large_t}, we shall find a barrier that prevents the components $\Sigma^-(h)$ and $\Sigma^+(h)$ to be connected at heights $h<G(z,u_t^{d+1}(z))$. Once again, we would like to use convexity of the function $x\mapsto G(z,u_t^{d+1}(z)+x)$, but in this case the line $u_t^{d+1}(z)+\R$ may run through singularities $\lambda_j$. Enlarging $c_3$ to our final constant $c:=10 c_3$, we first restrict ourselves to $x \in\R$ such that $|u_t^{d+1}(z)+x-z|< \frac{c }{2}\sqrt t$, which implies $|u_t^{d+1}(z)+x-\lambda_j|> \frac {c} 2\sqrt t$ since $|z-\lambda_j|>c \sqrt t$ by assumption. Analogously to \eqref{eq:convexity2}, we obtain
 
\begin{align}\label{eq:convexity_small_t}
	\frac{\partial^2}{\partial x^2}G(z,u_t^{d+1}(z)+x) 
	\ge \frac 1 t -\frac 1 \alpha \sum_{j=1}^d \frac{\alpha_j}{|u_t^{d+1}(z)+x-\lambda_j|^2}
	\ge \frac 1 t -\frac 1 \alpha \sum_{j=1}^d \frac{\alpha_j}{{c}^2/4 t}
	= \frac 1 t-\frac {4} {{c}^2t}>0,
\end{align}
since $c>2$.
On the other hand, we claim that this barrier can be extended through all of 
\begin{align}\label{eq:V_1}
	 V_1\coloneq \{u\in\C: |\Re(z-u)| >2|\Im(z-u)| \}\setminus (\bigcup_j B_{c\sqrt t/2}(\lambda_j)\cup B_{c\sqrt t/2}(z)),
	\end{align}
	wherein the heights of $G(z,u)$ turn out to be large.
First, note that the excluded balls are disjoint and that $u\in V_1$ implies 
 \begin{align}\label{eq:Re25}
 \Re((z-u)^2)= (\Re(z-u))^2-(\Im (z-u))^2> \frac 1 2 (\Re(z-u))^2 >\frac 2 5 |z-u|^2.
 \end{align}

Applying \eqref{eq:sp_dist}, \eqref{eq:V_1}, and \eqref{eq:Re25}, it follows
\begin{align}
G(z,u)-G(z,u_t^{d+1}(z))&= -\sum_{j=1}^d \frac{\alpha_j}{\alpha}\log\left|\frac {u_t^{d+1}(z)-\lambda_j}{u-\lambda_j}\right| + \Re\left(\frac{(z-u)^2}{2t} \right)- \Re\left(\frac{(z-u_t^{d+1}(z))^2}{2t} \right)\nonumber \\
&>- \sum_{j=1}^d\frac{\alpha_j}{\alpha} \log\left|1+\frac{u_t^{d+1}(z)-u}{u-\lambda_j}\right| + \frac{|z-u|^2}{5t}  -\frac {c_3^2} 8 \nonumber\\
&\ge - \sum_{j=1}^d\frac{\alpha_j}{\alpha} \left( \left|\frac{u_t^{d+1}(z)-z}{u-\lambda_j}\right|+\left|\frac{z-u}{u-\lambda_j}\right| \right)+ \frac{c |z-u| }{10\sqrt t}   -\frac {c^2} {800} \nonumber\\
&\ge  - 1 +\frac{|z-u|}{\sqrt t}\left(-\frac{2}{c}+\frac{c}{10} \right)  -\frac {c^2} {800}>-2+\frac {c^2}{20}-\frac {c^2}{800}>0, \label{eq:V_1_barrier}
\end{align}
where we used that $\left(-\frac{2}{c}+\frac{c}{10} \right)>0$ in the last line.
In conclusion, we can construct a barrier through $u=u_t^{d+1}(z)$ within $V_1\cup B_{c\sqrt t /2}(z)$ whose height is larger than $G(z,u_t^{d+1}(z))$, and hence $u_t^{d+1}(z)$ cannot be irrelevant.

It remains to check that at all heights $h>G(z,u_t^{d+1}(z))$, there exists a path $\gamma$ connecting $\Sigma^+(h)$ and $\Sigma^-(h)$. The arguments follow the same lines as above, with the roles of real and imaginary directions essentially reversed. First, we can take the same route as done in \eqref{eq:convexity_small_t} to prove concavity along $y\mapsto G(z,u_t^{d+1}(z)+iy)$ for $|u_t^{d+1}(z)+iy-z|< \frac c 2 \sqrt t$. Indeed, since $\frac{\partial^2}{\partial y^2}\Re(g) =-\Re( \partial_u^2 g)$, we have 
\begin{align*}
	\frac{\partial^2}{\partial y^2}G(z,u_t^{d+1}(z)+iy) 
\le -\frac 1 t +\frac 1 \alpha \sum_{j=1}^d \frac{\alpha_j}{|u_t^{d+1}(z)+iy-\lambda_j|^2}
\le -\frac 1 t +\frac 1 \alpha \sum_{j=1}^d \frac{\alpha_j}{{c}^2/4t}
= -\frac 1 t+\frac {4} {{c}^2t}<0.
\end{align*}
On the other hand, $G(z,u)<G(z,u_t^{d+1}(z))$ within
$$V_2\coloneq \{u\in\C: |\Im(z-u)|> 2 |\Re(z-u)|\}\setminus B_{c\sqrt t /2}(z)$$

since here we have $(\Re(z-u)^2)<-\frac 1 2 (\Im(z-u))^2\le -\frac 2 5 |z-u|^2$ as in \eqref{eq:Re25} and therefore
\begin{align*}
	G(z,u)-G(z,u_t^{d+1}(z))&= \sum_{j=1}^d \frac{\alpha_j}{\alpha}\log\left|\frac {u-\lambda_j}{u_t^{d+1}(z)-\lambda_j}\right| + \Re\left(\frac{(z-u)^2}{2t} \right)- \Re\left(\frac{(z-u_t^{d+1}(z))^2}{2t} \right) \\
	&<  \sum_{j=1}^d\frac{\alpha_j}{\alpha} \log\left|1+\frac{u-u_t^{d+1}(z)}{u_t^{d+1}(z)-\lambda_j}\right| - \frac{|z-u|^2}{5t}  +\frac {c_3^2} 8 \nonumber\\
	&\le  \sum_{j=1}^d\frac{\alpha_j}{\alpha} \left( \left|\frac{u_t^{d+1}(z)-z}{u_t^{d+1}(z)-\lambda_j}\right|+\left|\frac{z-u}{u_t^{d+1}(z)-\lambda_j}\right| \right)- \frac{c |z-u| }{10\sqrt t} +\frac {c^2} {800} \nonumber\\
	&\le  1 +\frac{|z-u|}{\sqrt t}\left( \frac{2}{c}-\frac{c}{10} \right) +\frac {c^2} {800}< 0 \nonumber,
\end{align*}
since $\left( \frac{2}{c}-\frac{c}{10} \right)<0$.
Hence, there is a path $\gamma$ within $V_2\cup B_{c\sqrt t /2}$ and connecting $\Sigma^+(h)$ and $\Sigma^-(h)$ as claimed.

In particular, we have shown that $u_t^{d+1}(z)$ must be maximally relevant. Further, it is locally analytic by the implicit function theorem, as mentioned above, as long as we keep a distance from the branching points, which we excluded by assumption. As the real part of a locally analytic function, $G(z,u_t^{d+1}(z))$ is locally harmonic, as is $U_t$ on $D$, which finishes the proof.
\end{proof}

Additionally, we need the following lemma, which quantitatively locates the saddle points as multivalued functions.
\begin{lemma}\label{lem:u_t_small_t}
	For any fixed $z\in\C\setminus B$ and $t\to 0$, there exists one saddle point that converges to $\lambda_j$ for each initial zero $\lambda_1,\dots,\lambda_d$ of $P(z)$, and one saddle point that converges to $z$. For fixed $\tilde z\in\C$ and $j\le d$, set $z=\lambda_j+\tilde z \sqrt{\frac{t\alpha_j}{\alpha}}$ and assume $z\in\C\setminus B$ for all sufficiently small $t$. Then there exist branches of the saddle points satisfying
	\begin{align}
		u_t^{d+1}(z)&= \lambda_j + \sqrt{\frac{t\alpha_j}{\alpha}}
		u^+(\tilde z)+o(\sqrt t)\label{eq:ut_quant_d+1}, \\
		u_t^j(z)&= \lambda_j + \sqrt{\frac{t\alpha_j}{\alpha}}
		u^-(\tilde z)+o(\sqrt t) \label{eq:ut_quant_j}, \\
		u_t^k(z)&=\lambda_k+\frac{\alpha_k t}{\alpha (\lambda_j-\lambda_k)}+o(t)\quad \text{for }k\neq j, d+1,\label{eq:ut_quant_notj}
	\end{align}
	where $u^{\pm}(\tilde z)=\frac 1 2(\tilde z \pm \sqrt{\tilde z^2 -4})$ are defined in \eqref{eq:u_pm}.
\end{lemma}

\begin{proof}
For $z\in  D= \{z\in\C\setminus B:\,|z|<C_{\max}, \min_{j\le d+1}|z-\lambda_j|>c\sqrt t\}$, the first claim follows from \eqref{eq:sp_dist} as defined in Lemma \ref{lem:small_t}. On the other hand, if $z\in B_{c\sqrt t }(\lambda_j)\setminus B$, then $w:=\lambda_j+ 2c \sqrt t\in D$ is $\mathcal O (\sqrt t )$-close to $z$, and there exists a small simply connected domain $D_w \subseteq \C\setminus B$ such that $w\in D_w$. Locally, in $D_w$, we can fix a numbering of the branches and continuity ensures that $u_t^k(z)=u_t^k(w)+\mathcal O (\sqrt t)\to \lambda_k$ for $k\le d$ and $u_t^{d+1}(z)=u_t^{d+1}(w)+\mathcal O (\sqrt t)\to z$. This proves the qualitative statement.

Now, fix $j\le d$,  let us write $z=\lambda_j+\tilde z \sqrt{\frac{t\alpha_j}{\alpha}}$ for some $\tilde z\in\C$ and  $
\tilde u(\tilde z)\coloneq \sqrt{\frac {\alpha}{t\alpha_j}}(u_t^k(z)-\lambda_j)$ for $u_t^k(z)$ being one of the branches converging to $\lambda_j$. Inserting $\tilde u(\tilde z)$ into the saddle point equation $u-z+\frac t \alpha \sum_i\frac{\alpha_i}{u-\lambda_i}=0$ yields
\begin{align*}
	&0=u_t^k(z)-\tilde z \sqrt{\frac {t\alpha_j}{\alpha}}-\lambda_j+\frac {t}{\alpha}\sum_i \frac{\alpha_i}{u_t^k(z)-\lambda_i} \\
	\iff &0=\tilde u (\tilde z)-\tilde z+\frac{1}{\tilde u(\tilde z)}+\frac{\sqrt t}{\sqrt{\alpha\alpha_j}} \sum_{i\neq j}\frac{\alpha_i}{u_t^k(z)-\lambda_i}
	=\tilde u (\tilde z)-\tilde z+\frac{1}{\tilde u(\tilde z)}+\mathcal O(\sqrt t)
\end{align*}
as $t\to 0$. Hence, we recover the Joukowsky transform $\tilde z\approx \tilde u (\tilde z)+\frac{1}{\tilde u(\tilde z)}$, which has exactly two solutions in $\tilde u$, namely $u^\pm(\tilde z)=\frac 1 2(\tilde z \pm \sqrt{\tilde z^2 -4})$ as defined in \eqref{eq:u_pm}. In particular, we obtain the asymptotics $\tilde u(\tilde z)=u^\pm(\tilde z)+o(1)$ and both equations, \eqref{eq:ut_quant_d+1} and \eqref{eq:ut_quant_j} follow. On the other hand, for the branches $u^k_t$ that do not converge to the fixed $\lambda_j$, it holds
$$u_t^k(z)=\tilde z \sqrt{\frac {t\alpha_j}{\alpha}}+\lambda_j-\frac {t}{\alpha}\sum_i \frac{\alpha_i}{u_t^k(z)-\lambda_i}\to \lambda_k. $$
This can only be true if one summand $i=k$ cancels the constant term $\lambda_j$ with the leading asymptotic as claimed in \eqref{eq:ut_quant_notj}.
\end{proof}

\begin{lemma}\label{lem:simple_B}
	Fix $j\le d$, let $r >2\sqrt{\alpha_j/\alpha}$, and $t>0$ be sufficiently small. Then we have exactly two simple branching points in $B\cap B_{ {r \sqrt t}}(\lambda_j)$.

\end{lemma}
\begin{proof}
	Let us first locate the branching points, which are the zeros of the (signless) discriminant
	\begin{align*}
		R(z):=\prod_{1\le i< k\le d+1}(u_t^i(z)-u_t^k(z) )^2,
	\end{align*}
	see also the proof of Lemma \ref{lem:branching}. Note that this (commutative) product and the subsequent formulas do not depend on the actual choice of numbering the branches $u_t^1,\dots,u_t^{d+1}$. Define the polynomial $$V(z):=\prod_{1\leq i<k\leq d}(\lambda_i-\lambda_k) \cdot \prod _{k=1}^d(z-\lambda_k),$$
	which has our fixed roots $\lambda_k$ and is rescaled by a Vandermonde factor.
	By Lemma \ref{lem:u_t_small_t}, we have $R(z)\to V(z)^2$ as $t\to 0$. More precisely, we split the product of $R$ into the cases $i<k\le d$ and $i\neq j, k=d+1$ and $i=j, k=d+1$, then the quantitative estimates \eqref{eq:ut_quant_d+1}, \eqref{eq:ut_quant_j}, and \eqref{eq:ut_quant_notj} imply for $z=\lambda_j + \tilde z \sqrt{\frac {t\alpha_j} {\alpha}}$
	\begin{align*}
		R(z)&=\prod_{i<k\le d}(\lambda_i-\lambda_k+\mathcal O (\sqrt t))^2\cdot \prod_{i\neq j}(\lambda_i-\lambda_j+\mathcal O (\sqrt t))^2\cdot \Big( \sqrt{\frac{t\alpha_j}{\alpha}}(u^+(\tilde z)-u^-(\tilde z))\Big)^2\\
		&=\prod_{i<k}(\lambda_i-\lambda_k)^2 \cdot \prod _{i\neq j}^d(\lambda_j-\lambda_i)^2\cdot \frac{t\alpha_j}{\alpha}\cdot \Big(\sqrt{\tilde z^2-4}\Big)^2 +\mathcal O(t^{3/2}).
	\end{align*}
	Similarly,
	\begin{align*}
		V(z)^2=\prod_{i<k}(\lambda_i-\lambda_k)^2\cdot \prod _{k\neq j}^d(\lambda_j-\lambda_k)^2\cdot \frac{t\alpha_j}{\alpha} \tilde z ^2 +\mathcal O(t^{3/2}).
	\end{align*}
	If $|\tilde z|^2>4$, it follows
	\begin{align*}
		|V(z)^2-R(z)|= \prod_{i<k}|\lambda_i-\lambda_k|^2 \prod _{i\neq j}^d|\lambda_j-\lambda_i|^2\cdot \frac{t\alpha_j}{\alpha} |\tilde z ^2-\tilde z ^2 +4| +\mathcal O(t^{3/2})<|V(z)|^2
	\end{align*}
	for sufficiently small $t$ and all $z=\lambda_j + \tilde z \sqrt{\frac {t\alpha_j} {\alpha}}\in\partial B_{r\sqrt t}(\lambda_j)$. By Rouch\'e's theorem, $R$ has the same number of zeros in $B_{r\sqrt t}(\lambda_j)$ as $V^2$, that is two.
	
	Again by Lemma \ref{lem:u_t_small_t}, we know that precisely two (and not more) branches meet at each of the at most two branching points $ z_0\in B\cap B_{r\sqrt t}(\lambda_j)$. Hence, both these branching points are simple (of order two) and the monodromy action induced by analytically continuing around such a branching point corresponds to a transposition in the monodromy group, exchanging the two coalescing branches.
\end{proof}

With the asymptotics of Lemma \ref{lem:u_t_small_t} at hand, we can now identify the maximally relevant saddle point close to $\lambda_j$.

\begin{lemma}\label{lem:ut*_small_t}
	For $c>0$ as in Lemma \ref{lem:small_t}, and $z\in B_{{c\sqrt t}}(\lambda_j)\setminus B$, we have $u_t^*(z)\to\lambda_j$ as $t\to 0$.
\end{lemma}
\begin{proof}
From Lemma \ref{lem:u_t_small_t}, we know that all branches converge either to $\lambda_j$ or to some other $\lambda_k$ for some $k\le d$. We now show that such a branch $u_t^k(z)=\lambda_k+\mathcal O (t)$ as in \eqref{eq:ut_quant_notj} cannot be the maximally relevant saddle point $u_t^*(z)$ for $z\in B_{{c\sqrt t}}(\lambda_j)\setminus B$. By Lemma \ref{lem:saddle}, the maximally relevant saddle point $u_t^*(z)$ is characterized as the branch whose height determines the level at which the sublevel set components containing $-i\infty$ and $+i\infty$ merge. However, it turns out that the branches $u_t^k(z)=\lambda_k+\mathcal O (t)$ may only merge the component containing $\lambda_k$ with a single other component. Thus, even if this other component contains $-i\infty$ or $+i\infty$, the two components containing $-i\infty$ and $+i\infty$ cannot be connected through $u_t^k(z)$. This follows once we have established the following circular barrier around $\lambda_k$, passing through $u_t^k(z)$. We claim that for sufficiently small $t>0$ it holds
\begin{align}\label{eq:circularbarrier}
F(\theta):=G(z,\lambda_k+\eee^{i\theta}(u_t^k(z)-\lambda_k))\ge F(0), \quad \theta\in[-\pi,\pi],
\end{align}
uniformly in $z\in B_{{c\sqrt t}}(\lambda_j)\setminus B$, with equality iff $\theta=0$. Indeed, suppose that this holds. Then the component of the sublevel set $\Sigma(G(z,u_t^k(z))-\varepsilon)$ containing $\lambda_k$ cannot cross the circle for any $\varepsilon>0$, whereas at $\eps=0$ it will merge with some other component through the saddle point $u_t^k(z)$.

In order to prove \eqref{eq:circularbarrier}, we show that
\begin{align}\label{eq:F'}
	F'(\theta)=\frac{\alpha_k}\alpha \sin\theta +|\eee^{i\theta}-1|\mathcal O (t),
\end{align}
uniformly in $z$ and $\theta$. Assuming this for the moment, we integrate in $\theta$ and obtain 
$F(\theta)-F(0)=(\frac{\alpha_k}\alpha +\mathcal O (t))(1-\cos\theta) \geq 0,$
implying the claim \eqref{eq:circularbarrier}.

It remains to show \eqref{eq:F'}. We calculate
\begin{align*}
F'(\theta)&=\Re\left[ i\eee^{i\theta}(u_t^k(z)-\lambda_k) \left(\frac{\lambda_k+\eee^{i\theta}(u_t^k(z)-\lambda_k)-z}{t} +\frac{1}{\alpha }\sum_{\ell=1}^d\frac{\alpha_\ell}{\lambda_k+\eee^{i\theta}(u_t^k(z)-\lambda_k)-\lambda_\ell}\right)\right]\\
&=\Re\left[ i\eee^{i\theta}(\eee^{i\theta}-1)(u_t^k(z)-\lambda_k)^2 \left(\frac 1 t - \frac{1}{\alpha }\sum_{\ell=1}^d\frac{\alpha_\ell}{(\lambda_k+\eee^{i\theta}(u_t^k(z)-\lambda_k)-\lambda_\ell)(u_t^k(z)-\lambda_\ell)}\right)\right],
\end{align*}
where we subtracted $F'(0)=0$ in the second line. Now observe that the summand corresponding to $\ell=k$ leads to $\Re[-\frac{\alpha_k}{\alpha} i(\eee^{i\theta}-1)]=\frac{\alpha_k}{\alpha}\sin \theta$ as claimed in \eqref{eq:F'}. For the remaining terms, note that $u_t^k(z)-\lambda_k=\mathcal O (t)$ as stated in \eqref{eq:ut_quant_notj} holds uniformly in $z\in B_{{c\sqrt t}}(\lambda_j)\setminus B$ since $z-\lambda_k$ is uniformly bounded away from $0$,  implying that their contribution is given by $|\eee^{i\theta}-1|\mathcal O (t)$.
\end{proof}

Finally, let us turn to the Proof of Theorem \ref{thm:small_t}.
\begin{proof}[Proof of Theorem \ref{thm:small_t}]
Let $z=\lambda_j + \tilde z \sqrt{\frac {t\alpha_j} {\alpha}}$ for fixed $\tilde z$, then $u_t^*(z)\sim \lambda_j + \sqrt{\frac {t\alpha_j} {\alpha}}u^\pm(\tilde z)$ by Lemma \ref{lem:u_t_small_t} and Lemma \ref{lem:ut*_small_t}. We claim that either $u^+(\tilde z)$ or $u^-(\tilde z)$ is the right choice of sign for the maximally relevant saddle point, depending on the location of $\tilde z$. As previously done in \eqref{eq:function_H} in the proof of Theorem \ref{thm:semicircle}, we consider the disk $D=\overline{B_{2-\eps}(0)}$ for $\eps>0$. It holds
\begin{align*}
G\left(\lambda_j+\tilde z\sqrt{\frac {t\alpha_j}{\alpha}},\lambda_j+\sqrt{\frac {t\alpha_j}{\alpha}} u^+(\tilde z)\right)-G\left(\lambda_j+\tilde z\sqrt{\frac {t\alpha_j}{\alpha}},\lambda_j+\sqrt{\frac {t\alpha_j}{\alpha}} u^-(\tilde z)\right)\\
\sim \frac {\alpha_j}{\alpha} \left(\log \left|\frac 1 2 (\tilde z+\sqrt{\tilde z^2-4})\right|^2+\frac 1 2 \Re(-\tilde z\sqrt{\tilde z^2-4}) \right)=\frac {\alpha_j}{\alpha} H(\tilde z)
\end{align*}
and, as shown in \eqref{eq:function_H_derivative},
$$\frac{\partial}{\partial \tilde y}\frac{H(\tilde x +i\tilde y)}{2}=-\Im \partial_{\tilde z} H(\tilde z)=\frac 1 2 \Im(\sqrt{(\tilde x+i\tilde y)^2-4})>0.$$

As already mentioned above, $H(\tilde x)=0$ for all $\tilde x\in (-2,2)$, and by the same arguments presented in the proof of Theorem \ref{thm:semicircle}, the nodal set
$$\{\tilde z\in D:\,G(\lambda_j+\tilde z\sqrt{{t \alpha_j}/{\alpha}}, u_t^{d+1}(\lambda_j+\tilde z\sqrt{{t \alpha_j}/{\alpha}}))=G(\lambda_j+\tilde z\sqrt{{t \alpha_j}/{\alpha}}, u_t^j(\lambda_j+\tilde z\sqrt{{t \alpha_j}/{\alpha}}))\}$$
 is a single curve $\gamma_j$ within the strip $(-2+\varepsilon,2-\varepsilon)+i(-\delta/2,\delta/2)$.
 
 Also analogously to the proof of Theorem \ref{thm:semicircle}, we 
 determine the maximally relevant saddle point above and below the curve $\gamma_j$, for instance for $z=\lambda_j + \tilde z \sqrt{\frac {t\alpha_j} {\alpha}}$ with $\tilde z =i$.
We know that 
$G(z, u_t^{d+1}(z))>G(z, u_t^j(z))$
for those $\tilde z$ lying above $\gamma_j$. To show that $u_t^{d+1}(z)$ is not irrelevant, we claim that a path $\gamma$ connecting $-i\infty$ and $+i\infty$ at heights $h<G(z,u_t^{d+1}(z))$ is prevented by a barrier consisting of two parts. For the first part, one may take the same approach as in \eqref{eq:convexity2} and below to show convexity of $x\mapsto G(z,u_t^{d+1}(z)+x)$, which holds as long as $|x|<c\sqrt t$, with $c>0$ defined in Lemma \ref{lem:small_t}.

 Indeed, for $\sqrt t/2 <|x|<c\sqrt t$ we have
\begin{align*}
	\frac{\partial^2}{\partial x^2}G(z,u_t^{d+1}(z)+x) 
	\ge \frac 1 t -\frac 1 \alpha \sum_{k=1}^d \frac{\alpha_k}{|u_t^{d+1}(z)+x-\lambda_k|^2}
	=\frac 1 t -\frac {1} \alpha  \frac{\alpha_j}{t\frac{\alpha_j}\alpha +x^2 +o(t)} +\mathcal O(1)
	\ge \frac 1 {7t} >0
\end{align*}
where we used \eqref{eq:ut_quant_d+1} and $u^+(\tilde z)=u^+(i)\in i\R_+$. For $|x|<\sqrt t/2$, we have
\begin{align*} 
	\frac{\partial^2}{\partial x^2}G(z,u_t^{d+1}(z)+x) =\frac 1 t-\frac 1 \alpha \sum_{k=1}^d \alpha_k\frac{ (x+\Re u_t^{d+1}(z)-\Re \lambda_k)^2-(\Im u_t^{d+1}(z)-\Im \lambda_k)^2 }{| u_t^{d+1}(z)+x-\lambda_k|^4}>\frac 1 {2t} >0
\end{align*}
for sufficiently small $t$, since the $k=j$ summand is negative for $|x|<\sqrt t/2$ and small $t$, whereas all remaining summands are bounded.
On the other hand, this barrier $u_t^{d+1}(z)+[-c\sqrt t,c\sqrt t]$ can be extended to the left and right to $V_1$,
precisely as we have done in \eqref{eq:V_1} and \eqref{eq:V_1_barrier}. 
 
Hence, $u_t^{d+1}(\lambda_j+\tilde z\sqrt{{t \alpha_j}/{\alpha}})$ must be maximally relevant for all $\tilde z$ located above the curve $\gamma_j$. Again, the analogous argument choosing $\tilde z=-i$ proves that $u_t^j(\lambda_j+\tilde z\sqrt{{t \alpha_j}/{\alpha}})$ is maximally relevant for all $\tilde z$ that lie below $\gamma_j$.

Ultimately, we obtain
\begin{align*}
U_t\left(\lambda_j + \tilde z \sqrt{\frac {t\alpha_j} {\alpha}}\right)
&\sim \frac {\alpha_j} {\alpha} \left(\frac 1 \alpha_j \log\left|\prod_{k\neq j}(\lambda_k -\lambda_j)^{\alpha_k}\right|+\frac 1 2 \log \left |\frac {t\alpha_j}{\alpha} \right| +\Psi(\tilde z) \right),
\end{align*}
where $\Psi(\tilde z)$ is the logarithmic potential of the standard semicircle law as defined in \eqref{eq:Psi}.
Hence, $$\tilde U_t^j(\tilde z):=\int_\C \log|\tilde z-w|d\tilde \mu_t^j(w)=U_t\left(\lambda_j + \tilde z \sqrt{\frac {t\alpha_j} {\alpha}}\right)-\frac 1 2 \log\left|\frac {t\alpha_j} {\alpha} \right|$$ and $$\tilde U_t^j(\tilde z)+\frac{\alpha-\alpha_j}{2\alpha}\log\left|\frac {t\alpha_j} {\alpha} \right|\to \frac {\alpha_j}{\alpha} \left(\Psi(\tilde z)+\frac{1}{\alpha_j}\log\left|\prod_{k\neq j}(\lambda_k -\lambda_j)^{\alpha_k}\right|\right)\quad\text{as } t\to 0.$$ Recall that, if $t>0$ is sufficiently small, we have $\mu_{n,t}\left(B_{2\sqrt{t}\sqrt{1+\frac{1}{2n\alpha}}}(\lambda_j)\right)=\frac{\alpha_j}{\alpha}$ by Lemma \refeq{lem:P_t^n roots}. Using the same reasoning as in the proof of Theorem \ref{thm:measure} once more, we obtain vague convergence of the sequence $\tilde \mu_t^j$ as stated. Note that the term $\frac {1} {\alpha} \log|\prod_{k\neq j}(\lambda_k -\lambda_j)^{\alpha_k}|$ does not change the result, as it disappears when applying the Laplacian.
\end{proof}

\begin{corollary}\label{cor:small_t}
	As in Lemma \ref{lem:simple_B}, fix $j\le d$, let $r>2\sqrt{\alpha_j/\alpha}$, and $t>0$ be sufficiently small. Then, the maximally relevant saddle point $u_t^*$ is univalued on $\C\setminus \cup_j \gamma_j$, where $\gamma_j$ are curves between the two branching points in $B\cap B_{r\sqrt t}(\lambda_j)$. In particular, the limiting measure $\mu_t$ is supported on these curves $\gamma_j$ with mass $\mu_t(B_{ {r\sqrt t} }(\lambda_j))=\frac{\alpha_j}{\alpha}$.
\end{corollary}
Note that $2\sqrt {\alpha_j/\alpha}$ is precisely the width of the semicircle law as in Theorem \ref{thm:small_t} and improves upon the non-optimal bound $\lesssim 2$ given in Lemma \ref{lem:P_t^n roots}.
\begin{proof}
Using \eqref{eq:ut_quant_d+1} and \eqref{eq:ut_quant_j}, and quantifying the asymptotics of height differences as in the proof of Theorem \ref{thm:large_t}, we show that one part $\gamma_j\subseteq \mathcal{D}_t^c$ of the curve tends to be horizontal between the branching points. 
Let $z=\lambda_j+\tilde z\sqrt{\frac{t\alpha_j }{\alpha}}\not\in B$ within some simply connected domain $D\subseteq \C\setminus B$, e.g.~$D=B_{c\sqrt t}(\lambda_j)\setminus \gamma_B$ where $\gamma_B$ are two straight branch cuts from the two branching points in $B\cap B_{c\sqrt t } (\lambda_j)$ to infinity. Note that these are not the branch cuts supporting the measure $\mu_t$, but we introduce them solely to fix a numbering of the saddle points.
Inserting the asymptotics \eqref{eq:ut_quant_d+1} and \eqref{eq:ut_quant_j}, we obtain again locally uniform convergence of harmonic functions
\begin{align*}
G(z,u_t^{d+1}(z))-G(z,u_t^{j}(z)) &=\frac{1}{\alpha}\sum_{k\neq j}\alpha_k \log\left|1+\sqrt{\frac{t\alpha_j}{\alpha}}\frac{u^+(\tilde z)-u^-(\tilde z)}{\lambda_j-\lambda_k}+o(\sqrt t)\right|
	 +\frac{\alpha_j}{\alpha}\log\left|u^+(\tilde z)^2+o(1)\right|\\
	 &+\frac{\alpha_j}{2\alpha}\Re(u^-(\tilde z)^2-u^+(\tilde z)^2)+o(1)\\
	 &=\frac{\alpha_j}{\alpha}H(\tilde z)+o(1), \quad \text{ as }t\to 0,
\end{align*}
where the function $H$ was defined in \eqref{eq:function_H}. It holds that $H(\tilde z)=0$ if $\tilde z\in (-2,2)$ and
	\begin{align*}
		\begin{cases}
			\lim_{\tilde y\downarrow 0}H(\tilde x+i\tilde y)<0\quad \text{for }|\tilde x|>2\\
			\lim_{\tilde y\uparrow 0}H(\tilde x+i\tilde y)>0\quad \text{for }|\tilde x|>2.\\
		\end{cases}
	\end{align*}
Together with the fact that $\frac{\partial}{\partial \tilde y}H(\tilde x +i\tilde y)>0$, as we have seen in \eqref{eq:function_H_derivative}, this yields that the function $\tilde y\mapsto H(\tilde x+i\tilde y)$, for $ |\tilde x|>2$, can only have one zero in the positive (and negative) $\tilde y$-direction. This implies that each vertical line $\{\tilde x+i\tilde y: \tilde y\in\R \}$ inside $B_{c\sqrt t}(\lambda_j)$ is crossed at most twice by some curve in $\Gamma_{j,d+1}^V$ for a neighborhood $V=(\tilde x -\varepsilon ,\tilde x +\varepsilon)+i\R$. Further, these curves do not support mass outside the ball $B_{c\sqrt t}(\lambda_j)$ for $c>r>0$ by Lemma \ref{lem:small_t}, and hence the same must hold in the interior of $B_{c\sqrt t}(\lambda_j)$. Indeed, by Lemma \ref{lem:sp_nbh}, the branch of the maximally relevant saddle point does not switch along this curve outside $B_{c\sqrt t}(\lambda_j)$. Now if we take a simply connected domain $W$ in $\C\setminus B$ such that $W\cap B_{c\sqrt t}(\lambda_j)\neq \emptyset$ and $W\cap B_{c\sqrt t}(\lambda_j)^c\neq \emptyset$, then $\Gamma_{j,d+1}^W$ is well defined and, again by Lemma \ref{lem:sp_nbh}, the branch of maximally relevant saddle point cannot switch anywhere along $\Gamma_{j,d+1}^W$.
In particular, the mass is concentrated on the curve between the two branching points, which we call $\gamma_j$, and, after rescaling by $z\mapsto \tilde z$, it converges in Hausdorff metric to the straight line $[-2,2]$ being the nodal set of $H$.
Excluding these branch cuts, one may globally define single-valued branches $u_t^k(z)$ for any $k=1,\dots,d+1$ and hence $u_t^*(z)$ for sufficiently small $t$. The last claim $\mu_t(B_{ {r\sqrt t} }(\lambda_j))=\frac{\alpha_j}{\alpha}$ follows from Lemma \ref{lem:P_t^n roots} and the fact that $\mu_t(B_{c\sqrt t}(\lambda_j)\setminus B_{ {r\sqrt t} }(\lambda_j))=0$ as argued above.
\end{proof}

\begin{remark}
	Tracking all the constants above, one may explicitly compute the threshold $t_0$, at which two lines, formed by the zeros of the heat-evolved polynomial, coalesce for the first time. For instance, in Lemma \ref{lem:small_t}, we found this constant to be of the form $ (c\delta(\lambda)/(d\alpha_ {\max}\lambda_{\max}) )^d$, which certainly is not optimal. This expression may be tracked throughout the subsequent steps as well, but we omit this technicality. In the regime $t\to 0$, we explicitly determine the rate at which the zeros of the heat-evolved polynomial spread around an initial zero $\lambda_j$, namely $2\sqrt {t\alpha_j /\alpha}\le 2\sqrt t$, where $2\sqrt t$ corresponds to the bound obtained from $t\to\infty$. Extrapolating a steady spread at speed $2\sqrt t$, two lines can touch earliest when $\delta(\lambda)=4\sqrt t$, hence we conjecture the optimal threshold to be of the form $|t|<{\delta(\lambda)^2}/{16}$.
\end{remark}

\subsection{Proof of Theorem \ref{thm:PDEs}}

In our final proof, we show the PDEs for the logarithmic potential and Stieltjes transform of $\mu_t$.

\begin{proof}[Proof of Theorem \ref{thm:PDEs}]

\emph{(1)}
Recall that for $(z,t)\in\mathcal{D}\coloneq \{(z,t)\in\C^2: t=|t|e^{i\theta}\neq 0, e^{-i\theta/2} z \in \mathcal{D}_{|t|}\}$, the maximally relevant saddle point $u_t^*(z)$ is uniquely given by $u_t^j(z)$ for some $j\in\{1,\dots,d+1\}$ as shown in Lemma \ref{lem:saddle} and \eqref{eq:sp_eqn_complex}, and that it depends locally analytically on $z$ and $t$.
The complex chain rule for $\partial_z$ implies
\begin{align*}
\partial_z (g(z,u^j_t(z)))=(\partial_u g)(z, u^j_t(z)) \partial_z u^j_t(z)+(\partial_z g)(z,u^j_t(z))=(\partial_z g)(z,u^j_t(z)),
\end{align*}
since $u_t^j(z)$ is a saddle point of $g(z,u)$. Hence, the right-hand side of the PDE is given by the negative square of
\begin{align} \label{eq:z_deriv_U}
\partial_z U_t(z)=\frac{1}{2}\partial_z \Bigl(g(z,u^j_t(z))\Bigr)=\frac{1}{2}(\partial_z g)(z, u^j_t(z))=\frac{z-u^j_t(z)}{2t}.
\end{align}
Similarly, the complex chain rule for $\partial_t$ yields that
$$
\partial_t(g(z,u^j_t(z)))=(\partial_u g)(z, u^j_t(z)) \partial_tu^j_t(z)+(\partial_t g)(z,u^j_t(z))=(\partial_t g)(z,u^j_t(z))$$
and the claim follows from
\begin{align} \label{eq:t_deriv_U}
\partial_t U_t(z)=\partial_t\big(\Re g(z,u^j_t(z))\big)=\frac 1 2(\partial_t g)(z,u^j_t(z))=-\frac{(z-u^j_t(z))^2}{4t^2}=-\left(\frac{z-u^j_t(z)}{2t}\right)^2.
\end{align}
Note that here we used the fact that $g$ is analytic in $t$ for $(z,t)\in\mathcal D$ such that $\partial_t \bar g =0$.

\emph{(2)} The formula we have shown in part \emph{(1)} and the identity
$m_t(z)=2\partial_z U_t(z)$ in $(z,t)\in\mathcal D$ for the Stieltjes transform yield
\begin{align*}
\partial_t m_t(z)&=2\partial_t\partial_z U_t(z)=2\partial_z\partial_t U_t(z)=-2\partial_z(\partial_zU_t(z))^2=-2\partial_z\Bigl(\frac{1}{2}m_t(z)\Bigr)^2\\
&=-\frac{1}{2}\partial_z(m_t(z))^2=-m_t(z)\partial_z m_t(z).
\end{align*}
We show the second statement of \emph{(2)} in a similar manner. Here, we note that the logarithmic potential $U_t(z)$ takes real values and hence
$\partial_{\bar t} U_t(z)=\overline{\partial_t U_t(z)}$.
Again, we have
\begin{align*}
\partial_{\bar t} m_t(z)&=2\partial_{\bar t}\partial_z U_t(z)=2\partial_z\partial_{\bar t} U_t(z)=2\partial_z(\overline{\partial_t U_t(z)})\\
&=-2\partial_z(\overline{\partial_zU_t(z)})^2=-2\partial_z\Bigl(\frac{1}{2}\overline{m_t(z)}\Bigr)^2\\
&=-\frac{1}{2}\partial_z(\overline{m_t(z)})^2=-\overline{m_t(z)}\partial_z \overline{m_t(z)},
\end{align*}
which finishes the proof.
\end{proof}
\begin{corollary}
	Let us consider the setting of Corollary \ref{cor:density}, i.e.~let $D\subset \C\setminus B$ be a fixed simply connected domain such that $\Gamma_{i,j}^D(t)=\{z\in D:G(z,u_t^i(z))=G(z,u_t^j(z))\}\neq \emptyset$ for exactly one pair $1\leq i <j \leq d+1$, where $\Gamma_{i,j}^D(t)$  is a real analytic curve that separates $D=D_i\cup D_j\cup \Gamma_{i,j}^D(t)$ into two simply connected open domains $D_i$ and $D_j$ at time $t=t_0$. Further, for $z\in D$ we denote the two Stieltjes branches as $m_t^i(z)=\frac{z-u_t^i(z)}{t}$ and $m_t^j(z)=\frac{z-u_t^j(z)}{t}$ respectively. Then, for each $t_0$ such that $z(t_0)\in\Gamma_{i,j}^D(t_0)$ there exists $\varepsilon>0$ such that  \begin{align}\label{eq:curve_evolution}
		\partial_t z(t)=\frac{m_t^i(z(t))+m_t^j(z(t))}2,\quad t\in[t_0-\varepsilon,t_0+\varepsilon],
		\end{align} 
with initial datum $z(t_0)$, has a unique solution that satisfies $z(t)\in\Gamma_{i,j}^D(t)$.
\end{corollary}
In particular, $\Gamma_{i,j}^D(t)$ can be locally parametrized in $t$ via solutions $z(t)$ of \eqref{eq:curve_evolution}.
Equation \eqref{eq:curve_evolution} describes the time evolution of a particle $z(t)$ on the curves within $\supp \mu_t$, which is transported along the average field of two Stieltjes branches and equals the shock speed of the inviscid Burgers' equation \eqref{eq:Burgers-PDE}. 

\begin{proof}
Local existence and uniqueness of the solution $z(t)$ to  \eqref{eq:curve_evolution} follow directly from the Picard-Lindelöf theorem, and it remains to verify that $z(t)\in\Gamma_{i,j}^D(t)$.	Define the functions $h_{i,j}(z,t):=g(z,u_t^i(z))-g(z,u_t^j(z))$ on $D$, where the function $g(z,u)$ is given by $g(z,u)=\frac 1 \alpha \sum_{i=1}^d \alpha_i \log (u-\lambda_i) +\frac{(z-u)^2}{2t}$ as introduced in \eqref{eq:function g}. Taking the derivative in $z$ and $t$ of $g(z,u)$, as we have done in \eqref{eq:z_deriv_U} and \eqref{eq:t_deriv_U}, yields that $h_{i,j}(z,t)$ fulfills the transport equation
	\begin{align} \label{eq:transport}
		\partial_t h_{i,j}(z,t)+\frac{m_t^i(z)+m_t^j(z)}{2}\partial_z h_{i,j}(z,t)=0
	\end{align}
	on $\{(z,t)\in D\times \C: t=|t|e^{i\theta}\neq 0, e^{-i\theta/2} z \in \mathcal{D}_{|t|}\}$. Now the chain rule $$\partial_t( h_{i,j}(z(t),t))=(\partial_z h_{i,j})(z(t),t)\partial_t z(t)+(\partial_t h_{i,j})(z(t),t)$$ implies that solutions to \eqref{eq:transport} are constant along characteristic curves \eqref{eq:curve_evolution}. Since $z_0=z(t_0)\in\Gamma_{i,j}^D(t_0)$, it follows $\Re(h_{i,j}(z(t_0),t_0))=\Re(h_{i,j}(z(t),t))=0$, and therefore $z(t)\in\Gamma_{i,j}^D(t)$.
\end{proof}

\section{Simple example}\label{sec:example}
We conclude our work with an example. In the simple case where we have $d=1$, $\alpha_1=\alpha=1$ and one root, say at $\lambda_1=\lambda=a\in\C$, the polynomial power $P^n$ becomes $P^n(z)=(z-a)^n$, and it is well known that $P_t^n(z)=(\tfrac{t}{n})^{n/2}\mathrm{He}_n(\sqrt n (z-a) /\sqrt t)$. In particular, $\mu_t=\mathsf{sc}_t(\cdot-a)$ is the semicircle distribution as already mentioned in Remark \ref{rem:example_hermite}. Nevertheless, we want to consider this example in order to make the objects introduced in this paper more tangible.
The function $G(z,u)$ becomes $G(z,u)=\log|u-a|+\Re\left(\frac{(z-u)^2}{2t}\right)$ and the saddle point equation \eqref{eq:sp_eqn_2} reads $u-z+\frac{t}{u-a}=0$, hence we have two saddle points in $$u_t^{1,2}(z)=\frac 1 2 (z+a\pm\sqrt{(z-a)^2-4t}),$$ which are globally defined for all $z\in \C\setminus \big((-\infty,a-2\sqrt t]\cup[a+2\sqrt t, \infty)\big)$ (we indicate straight horizontal lines by these intervals). In particular, we see that as $t\to 0$, $$u_t^1(z)\to\frac 1 2 (z+a+z-a)=z\quad\text{ and }\quad u_t^2(z)\to\frac 1 2 (z+a-z+a)=a$$ in $\{\Im(z-a)>0\}$, and the other way around in $\{\Im(z-a)<0\}$ as shown in Lemma \ref{lem:u_t_small_t}, whereas for sufficiently large values of $t$, the saddle points are close to $$u_t^{1,2}(z)\sim u_t^\pm(z),$$ as defined in \eqref{eq:u_pm}, by Lemma \ref{lem:large_t}. In both these cases of small and large values of $t$, the maximally relevant saddle point is given by $u_t^1(z)$ for $z$ lying above the line $\gamma_a:=\{z: \Im(z)=\Im(a)\}$, and by $u_t^2(z)$ for $z$ lying below, for all $z\in \mathcal D_t$. To prove this, note that we can repeat the arguments already presented in the proof of Theorem \ref{thm:semicircle},  since the nodal set corresponding to our saddle points equals $$G(z,u_t^1(z))-G(z,u_t^2(z))=H\left(\frac{z-a}{\sqrt t}\right),$$ where the function $H$ is defined in \eqref{eq:function_H}.
Moreover, in this simple example, the maximally relevant saddle point is given by the above for \emph{all} values of $t$. Again, we can use the same line of argument we have given in the proofs of both Theorem \ref{thm:semicircle} and Theorem \ref{thm:small_t}. We omit the details.
In particular, this allows us to provide explicit formulas for the logarithmic potential, namely
\begin{align*}
	U_t(z)=
	\begin{cases}
		\log \left|\frac 1 2 (z-a+\sqrt{(z-a)^2 -4t})\right|+\frac{1}{2t} \Re\left(\left(\frac 1 2 (z-a-\sqrt{(z-a)^2-4t})\right)^2\right),\, z\in\{\Im(z-a)>0\}\\
		\log \left|\frac 1 2 (z-a-\sqrt{(z-a)^2 -4t})\right|+\frac{1}{2t} \Re\left(\left(\frac 1 2 (z-a+\sqrt{(z-a)^2-4t})\right)^2\right),\,z\in\{\Im(z-a)<0\}
	\end{cases}
\end{align*}
which admits a continuous extension to all of $\C$, and the  Stieltjes transform, namely
 \begin{align*}
 	m_t(z)=
 	\begin{cases}
 		\frac{1}{2t}(z-a-\sqrt{(z-a)^2-4t}),\, z\in\{\Im(z-a)>0\}\\
 		\frac{1}{2t}(z-a+\sqrt{(z-a)^2-4t}),\,z\in\{\Im(z-a)<0\},\\
 	\end{cases}
 \end{align*}
 which admits a continuous extension to $\C\setminus [a-2\sqrt t, a+2\sqrt t]$.
 The limiting distribution $\mu_t$ is supported on $[a-2\sqrt t, a+2\sqrt t]$ with (local) density $$\rho(z)=\frac{|u_t^1(z)-u_t^2(z)|}{2\pi t}\ind_{[a-2\sqrt t, a+2\sqrt t]}(z)=\frac{|\sqrt{(z-a)^2-4t}|}{2\pi t}\ind_{[a-2\sqrt t, a+2\sqrt t]}(z)$$ as we have proven in Corollary \ref{cor:density}. Note that this is indeed a shifted semicircle distribution.

Even in the case $d=2$, the above calculations become substantially more involved and the results no longer contribute to an illustrative benefit, which is why we limit ourselves to presenting numerical simulations as shown in Figure \ref{fig:ex}.

\section*{Acknowledgements}
We thank Andrew Campbell and the anonymous reviewers for helpful comments.

ZK was supported by the German Research Foundation under Germany's Excellence Strategy EXC 2044 - 390685587, Mathematics M\"{u}nster: Dynamics - Geometry - Structure. The authors have been supported by the DFG priority program SPP 2265 Random Geometric Systems.

All content was created by the authors. AI tools were used only for occasional discussion and final proofreading.


\begin{thebibliography}{HHJK25b}
	
	\bibitem[ABW13]{Axler}
	Sheldon Axler, Paul Bourdon, and Ramey Wade.
	\newblock {\em Harmonic Function Theory}.
	\newblock Graduate Texts in Mathematics. Springer New York, 2013.
	
	\bibitem[AGVP23]{arizmendi_garza_vargas_perales}
	Octavio Arizmendi, Jorge Garza-Vargas, and Daniel Perales.
	\newblock Finite free cumulants: multiplicative convolutions, genus expansion
	and infinitesimal distributions.
	\newblock {\em Trans. Amer. Math. Soc.}, 376(6):4383--4420, 2023.
	
	\bibitem[AGZ10]{AGZ}
	Greg~W. Anderson, Alice Guionnet, and Ofer Zeitouni.
	\newblock {\em An introduction to random matrices}, volume 118.
	\newblock Cambridge University Press, 2010.
	
	\bibitem[BB09]{borcea}
	Julius Borcea and Rikard B{\o}gvad.
	\newblock Piecewise harmonic subharmonic functions and positive cauchy
	transforms.
	\newblock {\em Pacific journal of mathematics}, 240(2):231--265, 2009.
	
	\bibitem[Ber88]{berry1988}
	M.~V. Berry.
	\newblock Stokes’ phenomenon; smoothing a victorian discontinuity.
	\newblock {\em Publications Mathématiques de l'Institut des Hautes Études
		Scientifiques}, 1988.
	
	\bibitem[Ber89]{berry1989}
	M.~V. Berry.
	\newblock Uniform asymptotic smoothing of stokes's discontinuities.
	\newblock {\em Proceedings of the Royal Society of London. Series A,
		Mathematical and Physical Sciences}, 422(1862), 1989.
	
	\bibitem[BGSF08]{Reeb}
	Silvia Biasotti, Daniela Giorgi, Michela Spagnuolo, and Bianca Falcidieno.
	\newblock Reeb graphs for shape analysis and applications.
	\newblock {\em Theoretical computer science}, 392(1-3):5--22, 2008.
	
	\bibitem[BHS24]{BHS24}
	Rikard B{\o}gvad, Christian H{\"a}gg, and Boris Shapiro.
	\newblock Rodrigues’ descendants of a polynomial and boutroux curves.
	\newblock {\em Constructive Approximation}, 59(3):737--798, 2024.
	
	\bibitem[Bia97]{Biane}
	Philippe Biane.
	\newblock On the free convolution with a semi-circular distribution.
	\newblock {\em Indiana University Mathematics Journal}, 46(3):705--718, 1997.
	
	\bibitem[BLR22]{Byun}
	Sung-Soo Byun, Jaehun Lee, and Tulasi Reddy.
	\newblock Zeros of random polynomials and their higher derivatives.
	\newblock {\em Transactions of the American Mathematical Society},
	375(09):6311--6335, 2022.
	
	\bibitem[Bra74]{braess}
	Dietrich Braess.
	\newblock Morse-{T}heorie für berandete {M}annigfaltigkeiten.
	\newblock {\em Mathematische Annalen}, 208:133--148, 1974.
	
	\bibitem[BSY25]{ByunSeoMeng}
	Sung-Soo Byun, Seong-Mi Seo, and Meng Yang.
	\newblock Free energy expansions of a conditional ginue and large deviations of
	the smallest eigenvalue of the lue.
	\newblock {\em Communications on Pure and Applied Mathematics},
	78(12):2247--2304, 2025.
	
	\bibitem[Cam26]{CampbellAppell}
	Andrew Campbell.
	\newblock Free infinite divisibility, fractional convolution powers, and appell
	polynomials.
	\newblock {\em Documenta Mathematica}, 2026.
	
	\bibitem[CJ26]{CJ}
	Andrew Campbell and Jonas Jalowy.
	\newblock Polya--schur problems and free probability.
	\newblock {\em arXiv preprint arXiv:2605.31356}, 2026.
	
	\bibitem[COR24a]{COR23}
	Andrew Campbell, Sean O'Rourke, and David Renfrew.
	\newblock The fractional free convolution of {R}-diagonal operators and random
	polynomials under repeated differentiation.
	\newblock {\em Int. Math. Res. Not. IMRN}, (13):10189--10218, 2024.
	
	\bibitem[COR24b]{CORuni}
	Andrew Campbell, Sean O'Rourke, and David Renfrew.
	\newblock Universality for roots of derivatives of entire functions via finite
	free probability.
	\newblock {\em arXiv preprint arXiv:2410.06403}, 2024.
	
	\bibitem[CSV94]{csordas_smith_varga}
	George Csordas, Wayne Smith, and Richard~Steven Varga.
	\newblock Lehmer pairs of zeros, the de {B}ruijn-{N}ewman constant {$\Lambda$},
	and the {R}iemann hypothesis.
	\newblock {\em Constr. Approx.}, 10(1):107--129, 1994.
	
	\bibitem[Cue24]{Cuenca}
	Cesar Cuenca.
	\newblock Cumulants in rectangular finite free probability and beta-deformed
	singular values.
	\newblock {\em arXiv preprint arXiv:2409.04305}, 2024.
	
	\bibitem[DS22]{Deano}
	Alfredo Dea{\~n}o and Nick Simm.
	\newblock Characteristic polynomials of complex random matrices and
	painlev{\'e} transcendents.
	\newblock {\em International Mathematics Research Notices}, 2022(1):210--264,
	2022.
	
	\bibitem[FBK{\etalchar{+}}12]{Fischmann}
	Jonit Fischmann, Wojciech Bruzda, Boris~A Khoruzhenko, Hans-J{\"u}rgen Sommers,
	and Karol {\.Z}yczkowski.
	\newblock Induced ginibre ensemble of random matrices and quantum operations.
	\newblock {\em Journal of Physics A: Mathematical and Theoretical},
	45(7):075203, 2012.
	
	\bibitem[Fis01]{Fischer}
	Gerd Fischer.
	\newblock {\em Plane algebraic curves}, volume~15.
	\newblock American Mathematical Soc., 2001.
	
	\bibitem[GNV25]{GNV25}
	Andr{\'e} Galligo, Joseph Najnudel, and Truong Vu.
	\newblock Dynamics of rotationally invariant polynomial root sets under
	iterated differentiations.
	\newblock {\em arXiv preprint arXiv:2506.06263}, 2025.
	
	\bibitem[HH22]{hallhobrown}
	Brian~C. Hall and Ching-Wei Ho.
	\newblock The brown measure of the sum of a self-adjoint element and an
	imaginary multiple of a semicircular element.
	\newblock {\em Letters in Mathematical Physics}, 112(2):19, 2022.
	
	\bibitem[HH25]{hallho}
	Brian~C. Hall and Ching-Wei Ho.
	\newblock The heat flow conjecture for polynomials and random matrices.
	\newblock {\em Letters in Mathematical Physics}, 115(3):60, 2025.
	
	\bibitem[HHJK25a]{GAF-paper}
	Brian~C. Hall, Ching-Wei Ho, Jonas Jalowy, and Zakhar Kabluchko.
	\newblock The heat flow, {GAF}, and ${SL}(2;\mathbb{R})$.
	\newblock {\em Indiana University Mathematics Journal}, 74:1153--1206, 2025.
	
	\bibitem[HHJK25b]{heatflowrandompoly}
	Brian~C. Hall, Ching-Wei Ho, Jonas Jalowy, and Zakhar Kabluchko.
	\newblock Zeros of random polynomials undergoing the heat flow.
	\newblock {\em Electronic Journal of Probability}, 30:1--55, 2025.
	
	\bibitem[HHJK26]{diff-paper}
	Brian~C. Hall, Ching-Wei Ho, Jonas Jalowy, and Zakhar Kabluchko.
	\newblock Roots of polynomials under repeated differentiation and repeated
	applications of fractional differential operators.
	\newblock {\em Trans. Amer. Math. Soc. Ser. B}, 13:190--239, 2026.
	
	\bibitem[HK21]{HK21}
	Jeremy Hoskins and Zakhar Kabluchko.
	\newblock Dynamics of zeroes under repeated differentiation.
	\newblock {\em Experimental Mathematics}, pages 1--27, 2021.
	
	\bibitem[H{\"o}r15]{hormander1}
	Lars H{\"o}rmander.
	\newblock {\em The analysis of linear partial differential operators I:
		Distribution theory and Fourier analysis}.
	\newblock Springer, 2015.
	
	\bibitem[JKM25a]{JKM1}
	Jonas Jalowy, Zakhar Kabluchko, and Alexander Marynych.
	\newblock Zeros and exponential profiles of polynomials i: Limit distributions,
	finite free convolutions and repeated differentiation.
	\newblock {\em arXiv preprint arXiv:2504.11593}, 2025.
	\newblock To appear in Constructive Approximation.
	
	\bibitem[JKM25b]{JKM2}
	Jonas Jalowy, Zakhar Kabluchko, and Alexander Marynych.
	\newblock Zeros and exponential profiles of polynomials ii: Examples.
	\newblock {\em arXiv preprint arXiv:2509.11248}, 2025.
	\newblock To appear in Constructive Approximation.
	
	\bibitem[Kab25a]{ZakharLeeYang}
	Zakhar Kabluchko.
	\newblock Lee--yang zeroes of the curie--weiss ferromagnet, unitary hermite
	polynomials, and the backward heat flow.
	\newblock {\em Annales Henri Lebesgue}, 8:1--34, 2025.
	
	\bibitem[Kab25b]{KabRep}
	Zakhar Kabluchko.
	\newblock Repeated differentiation and free unitary poisson process.
	\newblock {\em Transactions of the American Mathematical Society}, 2025.
	
	\bibitem[KKL25]{Kieburg}
	Mario Kieburg, Arno~BJ Kuijlaars, and Sampad Lahiry.
	\newblock Orthogonal polynomials in the normal matrix model with two
	insertions.
	\newblock {\em Nonlinearity}, 38(6):065013, 2025.
	
	\bibitem[KT22]{KT22}
	Alexander Kiselev and Changhui Tan.
	\newblock The flow of polynomial roots under differentiation.
	\newblock {\em Annals of PDE}, 8(2):16, 2022.
	
	\bibitem[KZ14]{KZ14}
	Zakhar Kabluchko and Dmitry Zaporozhets.
	\newblock Asymptotic distribution of complex zeros of random analytic
	functions.
	\newblock {\em The Annals of Probability}, 42(4):1374--1395, 2014.
	
	\bibitem[Lau11]{laudenbach}
	Fran{\c{c}}ois Laudenbach.
	\newblock A morse complex on manifolds with boundary.
	\newblock {\em Geometriae Dedicata}, 153(1):47--57, 2011.
	
	\bibitem[Mar66]{Marden}
	Morris Marden.
	\newblock {\em Geometry of polynomials}, volume No. 3 of {\em Mathematical
		Surveys}.
	\newblock American Mathematical Society, Providence, RI, second edition, 1966.
	
	\bibitem[MFR24]{martinezfinkelshtein}
	Andrei Martinez-Finkelshtein and Evgenii~A. Rakhmanov.
	\newblock Flow of the zeros of polynomials under iterated differentiation.
	\newblock {\em arXiv preprint arXiv:2408.13851}, 2024.
	
	\bibitem[Mil63]{milnor}
	John~Willard Milnor.
	\newblock {\em Morse theory}.
	\newblock Number~51. Princeton university press, 1963.
	
	\bibitem[MSS22]{MSS15}
	Adam~W. Marcus, Daniel~A. Spielman, and Nikhil Srivastava.
	\newblock Finite free convolutions of polynomials.
	\newblock {\em Probability Theory and Related Fields}, 182(3):807--848, 2022.
	
	\bibitem[O'S19]{O_Sullivan_2019}
	Cormac O'Sullivan.
	\newblock Revisiting the saddle-point method of perron.
	\newblock {\em Pacific Journal of Mathematics}, 298(1):157–199, February
	2019.
	
	\bibitem[OS21]{OSteiner}
	Sean O’Rourke and Stefan Steinerberger.
	\newblock A nonlocal transport equation modeling complex roots of polynomials
	under differentiation.
	\newblock {\em Proceedings of the American Mathematical Society},
	149(4):1581--1592, 2021.
	
	\bibitem[RS02]{rahman_schmeisser}
	Qazi~Ibadur Rahman and Gerhard Schmeisser.
	\newblock {\em Analytic Theory of Polynomials}.
	\newblock Oxford University Press, Oxford, 2002.
	
	\bibitem[RT20]{rodgers_tao}
	Brad Rodgers and Terence Tao.
	\newblock The de {B}ruijn--{N}ewman constant is non-negative.
	\newblock {\em Forum Math. Pi}, 8:e6, 62, 2020.
	
	\bibitem[RW98]{RockaWets}
	R.~Tyrrell Rockafellar and Roger J.-B. Wets.
	\newblock {\em Variational analysis}.
	\newblock Springer, 1998.
	
	\bibitem[Ste21]{Steiner21}
	Stefan Steinerberger.
	\newblock Free convolution powers via roots of polynomials.
	\newblock {\em Experimental Mathematics}, pages 1--6, 2021.
	
	\bibitem[Sze75]{szego75}
	G\'{a}bor Szeg\"{o}.
	\newblock {\em Orthogonal Polynomials}.
	\newblock American Mathematical Society, Providence, RI, 1975.
	
	\bibitem[Tao17]{tao_blog1}
	Terence Tao.
	\newblock Heat flow and zeroes of polynomials.
	\newblock
	\url{https://terrytao.wordpress.com/2017/10/17/heat-flow-and-zeroes-of-polynomials/},
	2017.
	
	\bibitem[Tao18]{tao_blog2}
	Terence Tao.
	\newblock Heat flow and zeroes of polynomials {II}.
	\newblock
	\url{https://terrytao.wordpress.com/2018/06/07/heat-flow-and-zeroes-of-polynomials-ii-zeroes-on-a-circle/},
	2018.
	
	\bibitem[Tot19]{totik}
	Vilmos Totik.
	\newblock Distribution of critical points of polynomials.
	\newblock {\em Trans. Amer. Math. Soc.}, 372(4):2407--2428, 2019.
	
	\bibitem[VEM90]{HermiteBound}
	S.~J.~L. Van~Eijndhoven and J.~L.~H. Meyers.
	\newblock New orthogonality relations for the hermite polynomials and related
	hilbert spaces.
	\newblock {\em Journal of Mathematical Analysis and Applications},
	146(1):89--98, 1990.
	
	\bibitem[VW22]{VW22}
	Michael Voit and Jeannette H.~C. Woerner.
	\newblock Limit theorems for {B}essel and {D}unkl processes of large dimensions
	and free convolutions.
	\newblock {\em Stochastic Processes and their Applications}, 143:207--253,
	2022.
	
	\bibitem[WW19]{Webb}
	Christian Webb and Mo~Dick Wong.
	\newblock On the moments of the characteristic polynomial of a ginibre random
	matrix.
	\newblock {\em Proceedings of the London Mathematical Society},
	118(5):1017--1056, 2019.
	
\end{thebibliography}
\newcommand{\etalchar}[1]{$^{#1}$}

\end{document}